\documentclass[a4paper]{amsart}

\usepackage[latin1]{inputenc}

\usepackage{enumerate}

\usepackage{float}
\restylefloat{table}

\usepackage{hyperref}
\usepackage{bm}
\usepackage{graphicx}
\usepackage{mathtools}
\usepackage{dirtytalk}

\usepackage{amsmath}

\usepackage{amssymb}

\usepackage{amsthm}

\usepackage{tikz-cd}

\usepackage{a4wide}

\usepackage{bbm}

\usepackage{comment}

\newtheoremstyle{plain}
{6pt}   
{6pt}   
{\itshape}  
{0pt}       
{\bfseries} 
{.}         
{5pt plus 1pt minus 1pt} 
{}          

\newtheoremstyle{definition}
{6pt}   
{6pt}   
{\normalfont}  
{0pt}       
{\bfseries} 
{.}         
{5pt plus 1pt minus 1pt} 
{}          

\theoremstyle{plain}
\newtheorem*{thm*}{Theorem}
\newtheorem{thm}{Theorem}[section]
\newtheorem{prop}[thm]{Proposition}
\newtheorem{cor}[thm]{Corollary}
\newtheorem{lem}[thm]{Lemma}
\newtheorem{theorem}{Theorem}

\theoremstyle{definition}
\newtheorem{defn}[thm]{Definition}

\newtheorem{ex}[thm]{Example}
\newtheorem{rmk}[thm]{Remark}

\numberwithin{equation}{thm}

\newcommand{\emphbf}[1]{\emph{\textbf{#1}}}

\DeclareMathAlphabet{\mathpzc}{OT1}{pzc}{m}{it}

\newcommand{\module}[1]{\mathsf{#1}}

\DeclareMathOperator{\radoperator}{rad}
\DeclareMathOperator{\Kopf}{top}

\DeclareMathOperator{\id}{id}
\DeclareMathOperator{\Hom}{Hom}
\DeclareMathOperator{\Ext}{Ext}

\DeclareMathOperator{\Mod}{Mod}
\DeclareMathOperator{\modu}{mod}
\DeclareMathOperator{\rep}{rep}

\DeclareMathOperator{\coker}{coker}

\DeclareMathOperator{\Id}{Id}

\DeclareMathOperator{\inj}{inj}

\DeclareMathOperator{\op}{op}

\DeclareMathOperator{\map}{\mathnormal{g}}

\DeclareMathOperator{\arrowin}{in}

\DeclareMathOperator{\Mimo}{Mimo}
\DeclareMathOperator{\Mo}{\mathfrak{Q}}
\DeclareMathOperator{\Mono}{Mono}
\DeclareMathOperator{\mono}{mono}

\begin{document}
	
	\title{A functorial approach to monomorphism categories II: Indecomposables}
	\author{Nan Gao}
	\author{Julian K\"ulshammer}
	\author{Sondre Kvamme}
	\author{Chrysostomos Psaroudakis}
	\date{\today}
	
	\address{Nan Gao\\
		Departement of Mathematics, Shanghai University\\
		Shanghai 200444, PR China
	} \email{nangao@shu.edu.cn}
	
	\address{Julian K\"ulshammer\\
		Department of Mathematics
		Uppsala University \\ Box 480 \\ 75106 Uppsala,
		Sweden} \email{julian.kuelshammer@math.uu.se}
	
	\address{Chrysostomos Psaroudakis\\
		Department of Mathematics
		Aristotle University of Thessaloniki \\ Thessaloniki, 54124, Greece}  \email{chpsaroud@math.auth.gr}
	
	\address{Sondre Kvamme\\
		Department of Mathematical Sciences, Norwegian University of Science and Technology
		\\ 7491 Trondheim, Norway
	} \email{sondre.kvamme@ntnu.no}

	\begin{abstract}
		We investigate the (separated) monomorphism category $\operatorname{mono}(Q,\Lambda)$ of a quiver over an Artin algebra $\Lambda$. We show that there exists an epivalence (called representation equivalence in the terminology of Auslander) from $\overline{\operatorname{mono}}(Q,\Lambda)$ to $\operatorname{rep}(Q,\overline{\operatorname{mod}}\, \Lambda)$, where  $\operatorname{mod}\Lambda$ is the category of finitely generated $\Lambda$-modules and $\overline{\operatorname{mod}}\, \Lambda$ and $\overline{\operatorname{mono}}(Q,\Lambda)$ denote the respective injectively stable categories. Furthermore, if $Q$ has at least one arrow, then we show that this is an equivalence if and only if $\Lambda$ is hereditary. In general, the epivalence induces a bijection between indecomposable objects in $\operatorname{rep}(Q,\overline{\operatorname{mod}}\, \Lambda)$ and non-injective indecomposable objects in $\operatorname{mono}(Q,\Lambda)$, and we show that the generalized Mimo-construction, an explicit minimal right approximation into $\operatorname{mono}{(Q,\Lambda)}$, gives an inverse to this bijection. We apply these results to describe the indecomposables in the monomorphism category of radical square zero Nakayama algebras, and to give a bijection between the indecomposables in the monomorphism category of two artinian uniserial rings of Loewy length $3$ with the same residue field.
		
		The main tool to prove these results is the language of a free monad of an exact endofunctor on an arbitrary abelian category. This allows us to avoid the technical combinatorics arising from quiver representations. The setup also specializes to more general settings, such as representations of modulations. In particular, we obtain new results on the singularity category of the algebras $H$ which were introduced by Geiss, Leclerc, and Schr\"oer in order to extend their results relating cluster algebras and  Lusztig's semicanonical basis to symmetrizable Cartan matrices. We also recover results on the $\imath$quivers algebras which were introduced by Lu and Wang to realize $\imath$quantum groups via semi-derived Hall algebras.
	\end{abstract}

 
	\maketitle
	
	\tableofcontents

	\section{Introduction}
	The fundamental theorem of finitely generated abelian groups is one of the oldest classification results in algebra. It was implicitly stated by Kronecker in \cite[\S 1]{Kro70}, and later more explicitly by Frobenius and Stickelberger in \cite[IV, p. 231]{FS79}. After its proof people started studying subgroups of finite abelian groups, with work by Miller \cite{Mil04,Mil05}, Hilton \cite{Hil07}, and Birkhoff \cite{Bir35}. In this case one is not only interested in the abstract groups, but also in their embedding. The classification of all such embeddings can be reduced to determining all indecomposable embeddings into a finitely generated $\mathbb{Z}/(p^n)$-module, for $p$ a prime. The difficulty of this problem increases with $n$, with $n=5$ only first determined in \cite{RW99} in 1999. For $n\geq 7$ the problem is controlled wild and there is no hope for a classification \cite{RS06}. The boundary case $n=6$ is still open, called the Birkhoff problem after the paper \cite{Bir35}, where Birkhoff posed the question of describing embeddings of finite abelian groups. In that paper he also constructed a $1$-parameter family of indecomposable embeddings of two fixed $\mathbb{Z}/(p^6)$-modules whose cardinality goes to infinity as $p$ goes to infinity, indicating that the problem is quite challenging.
	
	The study of submodules was revitalized by work of Ringel and Schmidmeier in a series of papers \cite{RS06, RS08, RS08b,RS24}. There they considered submodule categories of Artin algebras, with particular emphasis on the truncated polynomial ring $\Bbbk[x]/(x^n)$. These categories have connections to Hall algebras \cite{Sch12},  Littlewood--Richardson tableaux \cite{KS15,KS22,Sch11}, metabelian groups \cite{Sch05a}, preprojective algebras \cite{RZ14},  singularity categories \cite{Che11,HM21}, weighted projective lines \cite{KLM13}, and valuated groups \cite{Arn00,RW79}.  
	
	More general systems of submodules, called monomorphism categories, have also been studied in parallel, including their representation type \cite{Pla76,Sim02} and their Auslander--Reiten quivers \cite{Moo09,XZZ14}. They have connections to cotorsion pairs \cite{EHHS13,HJ19a}, Gorenstein homological algebra \cite{DELO21,EEG09,LZ13,XZZ14,Zha11}, parabolic vector bundles \cite{Moz20}, weighted projective lines \cite{Sim18}, and topological data analysis \cite{BBOS20}.  
	
	In this paper we make progress towards describing the indecomposable objects for any monomorphism category. As far as we know, our first main result (Theorem \ref{Representation equivalence}) is new even for classical submodule categories. To state it, we first explain the general setup. Let $\mathcal{B}$ be an abelian category and $Q$ a finite acyclic quiver with vertices $Q_0$ and arrows $Q_1$. The category of representation of $Q$ in $\mathcal{B}$, denoted $\operatorname{rep}(Q,\mathcal{B})$, is the category of functors $Q\to \mathcal{B}$, where $Q$ is considered as a category in the natural way. Explicitly, the objects are tuples $(B_\mathtt{i}, B_\alpha)_{\mathtt{i}\in Q_0, \alpha\in Q_1}$ where $B_{\mathtt{i}}$ is an object of $\mathcal{B}$ for each vertex $\mathtt{i}$ and $B_\alpha\colon B_{\mathtt{i}}\to B_{\mathtt{j}}$ is a morphism in $\mathcal{B}$ for each arrow $\alpha\colon\mathtt{i}\to \mathtt{j}$. The \emphbf{(separated) monomorphism category}, denoted $\operatorname{mono}(Q,\mathcal{B})$, is a full subcategory of $\operatorname{rep}(Q,\mathcal{B})$. It consists of all representations for which
	\[
	B_{\mathtt{i},\arrowin}\colon \bigoplus_{\substack{\alpha\in Q_1\\t(\alpha)=\mathtt{i}}} B_{s(\alpha)}\xrightarrow{(B_\alpha)_{\alpha}} B_\mathtt{i}
	\]
	is a monomorphism for all vertices $\mathtt{i}$, where $s(\alpha)$ and $t(\alpha)$ denote the source and target of the arrow $\alpha$. If $Q=(\mathtt{1}\to \mathtt{2})$ then we recover the classical submodule categories studied by Birkhoff, Ringel and Schmidmeier. In general, $\operatorname{mono}(Q,\mathcal{B})$ is an exact category, and if $\mathcal{B}=\operatorname{mod}\Lambda$ is the category of finitely generated right $\Lambda$-modules for an Artin algebra $\Lambda$, then it has almost split sequences \cite{LZ13,RS08}. 
	
	To describe the indecomposables in $\operatorname{mono}(Q,\mathcal{B})$ we construct an epivalence, i.e. a full and dense functor which reflects isomorphisms. Here we follow the terminology of \cite{Kel91}. Such functors were first considered by Auslander in \cite{Aus71} \say{as a way of saying that the representation theories of two additive categories are essentially the same}. He called such functors representation equivalences. They are also sometimes called detecting, e.g. see \cite{Bau95}.
	
	To state our result, let $\overline{\mathcal{B}}$ and $\overline{\operatorname{mono}}(Q,\mathcal{B})$ be the quotients of $\mathcal{B}$ and $\operatorname{mono}(Q,\mathcal{B})$ by the ideal of morphisms factoring through injective objects, where we use the exact structure of $\operatorname{mono}(Q,\mathcal{B})$ to define its injectives. For a description of the injectives see Section \ref{Section: InjectiveObjects} or \eqref{DescriptionInjectives}.
	
	\begin{theorem}[Theorems~\ref{Theorem: Canonical functor representation equivalence} and \ref{Theorem:EquivalenceHereditary}]\label{Representation equivalence}\label{Theorem A}
		Let $Q$ be a finite acyclic quiver and $\mathcal{B}$ an abelian category with enough injectives. Then the canonical functor
		\[
		\overline{\operatorname{mono}}(Q,\mathcal{B})\to \operatorname{rep}(Q,\overline{\mathcal{B}}).
		\]
		is an epivalence. Furthermore, if $Q$ has at least one arrow, then the functor is an equivalence if and only if $\mathcal{B}$ is hereditary.
	\end{theorem}
	This recovers the homology functor for perfect differential $kQ$-modules in \cite[Theorem 1.1 b)]{RZ17} and (the dual of) the equivalence in \cite[Theorem 1.5]{BBOS20} as special cases. It also gives a non-trivial characterization  of hereditary categories, and illustrates how constructions using such categories are often simpler, cf. the derived category \cite[Section 4]{Hap87}. 
	
	 The epivalence in Theorem \ref{Representation equivalence} induces a bijection between the isomorphism classes of indecomposable objects in $\overline{\operatorname{mono}}(Q,\mathcal{B})$ and in $\operatorname{rep}(Q,\overline{\mathcal{B}})$. If $\mathcal{B}$ is the module category of an Artin algebra, then these are further in bijection with the non-injective indecomposable objects in  $\operatorname{mono}(Q,\mathcal{B})$. From this we can deduce that a stable equivalence between two Artin algebras induces a bijection between the non-injective indecomposable objects in their corresponding monomorphism categories. We discuss consequences of this in more detail in Subsection \ref{Subsecition:StableEquivalences}. 
	
	Theorem \ref{Representation equivalence} is particularly useful for comparing local uniserial rings of finite Loewy length, such as $\mathbb{Z}/(p^n)$ and $\mathbbm{k}[x]/(x^n)$. On the one hand, there are several results on the representation theory of monomorphism categories over $\mathbbm{k}[x]/(x^n)$. For example,  the representation type of $\mono(\mathbb{A}_m,\operatorname{mod}\mathbbm{k}[x]/(x^n))$ for 
	\[
	\mathbb{A}_m=\mathtt{1}\to \mathtt{2}\to \dots \to \mathtt{m}
	\]
	was given in \cite{Sim02}, and the classification of indecomposables and  the Auslander--Reiten quiver  in the representation-finite and tame cases when $m\leq 3$ were given in \cite{Moo09,RS08b}. On the other hand, there haven't been many new results for $\Lambda=\mathbb{Z}/(p^n)$, which might be surprising because of the similarity between $\mathbb{Z}/(p^n)$ and $\mathbbm{k}[x]/(x^n)$, both being local uniserial rings of Loewy length $n$. However, it can be explained by the lack of certain tools like covering theory. As a prominent example, the Auslander--Reiten quiver of $\operatorname{mono}(\mathbb{A}_2,\operatorname{mod}\mathbbm{k}[x]/(x^6))$ was determined in \cite{RS08b} using covering theory. The analogous question for $\operatorname{mono}(\mathbb{A}_2,\operatorname{mod}\mathbb{Z}/(p^6))$ is the Birkhoff problem, which is still open. 
	
	In general there is a hope that the representation theory of monomorphism categories over $\mathbbm{k}[x]/(x^n)$ and $\mathbb{Z}/(p^n)$ are similar, see e.g. \cite{Sch08}. The following result goes a long way towards confirming this when $n\leq 3$. It is proved using Theorem \ref{Representation equivalence} and the fact that there is a stable equivalence 
	\[
	\overline{\operatorname{mod}}\, \mathbb{F}_p[x]/(x^n)\cong \overline{\operatorname{mod}}\,\mathbb{Z}/(p^n)
	\]
	in this case. Here $\mathbb{F}_p$ denotes the finite field with $p$ elements.
	
	\begin{theorem}[Theorem~\ref{Theorem:BijectionLoewyLength3}]\label{Introduction:BijectionLoewyLength3}
		Let $Q$ be a finite acyclic quiver and $n$ an integer less than or equal to $3$. Then, there exists a bijection which preserves partition vectors between indecomposable objects in $\mono(Q,\operatorname{mod}\mathbb{F}_p[x]/(x^n))$ and in $\mono(Q,\operatorname{mod}\mathbb{Z}/(p^n))$. 
	\end{theorem}
	This gives a non-trivial connection between the work of  Birkhoff and Ringel--Schmidmeier. 
	Here, by partition vector we mean the following: For a representation $(M_\mathtt{i},M_\alpha)$ over $R=\mathbb{F}_p[x]/(x^n)$ or $R=\mathbb{Z}/(p^n)$ each $M_\mathtt{i}$ can be written as $M_{\mathtt{i}}=R/\mathfrak{m}^{n_1}\oplus R/\mathfrak{m}^{n_2} \oplus \dots \oplus R/\mathfrak{m}^{n_k}$ where $\mathfrak{m}$ is the maximal ideal of $R$ and $n_1\geq n_2\geq \dots \geq n_k$. The associated sequence of numbers $(n_1,n_2,\dots, n_k)$ is the partition of $M_\mathtt{i}$. Doing this for each vertex $\mathtt{i}$ gives rise to the partition vector of the representation $(M_\mathtt{i},M_\alpha)$.
	
	Recall that Gabriel's theorem \cite{G72} classifies the hereditary finite-dimensional algebras over an algebraically closed field of finite representation type in terms of ADE Dynkin quivers. It is a fundamental result in representation theory of algebras. An application of Theorem \ref{Representation equivalence} gives a Gabriel-style classification of representation-finite mono\-morphism categories over radical square zero Nakayama Artin algebras. Here $m$ is the number of simple $\Lambda$-modules and $t$ is the number of non-injective simple $\Lambda$-modules.
	
	\begin{theorem}[Theorem~\ref{Theorem:rad2Nakayama}]\label{Gabriel-styleResult}
		Let $Q$ be a finite connected acyclic quiver and let $\Lambda$ be a non-semisimple radical square zero Nakayama Artin algebra. Then $\mono(Q,\operatorname{mod}\Lambda)$ is of finite type if and only if $Q$ is Dynkin. In this case, the number of indecomposable objects is $m \cdot |Q_0|+t \cdot |\Phi^+|$, where $|\Phi^+|$ denotes the set of positive roots corresponding to the Dynkin type.
	\end{theorem}
	We do not assume our algebra to be linear over a field. In particular, Theorem \ref{Gabriel-styleResult} also applies to $\mathbb{Z}/(p^2)$.
	
	Assume $\mathcal{B}$ is module category of an Artin algebra. Then Theorem \ref{Representation equivalence} provides a bijection.
	\begin{align}
	\renewcommand{\arraystretch}{1.8}
	\begin{array}{ccc}
	\renewcommand{\arraystretch}{1.1}
	\begin{Bmatrix}
	\text{Isomorphism classes of} \\
	\text{indecomposable non-injective} \\
	\text{objects in $\mono(Q,\mathcal{B})$}
	\end{Bmatrix}
	&
	\xrightarrow{\cong}
	&
	\renewcommand{\arraystretch}{1.2}
	\begin{Bmatrix}
	\text{Isomorphism classes of} \\
	\text{indecomposable objects} \\
	\text{in $\rep (Q,\overline{\mathcal{B}})$}
	\end{Bmatrix}
	\end{array}\label{BijectionFromRepEquiv}
	\renewcommand{\arraystretch}{1}
	\end{align}
	given by the functor $\operatorname{mono}(Q,\mathcal{B})\to \overline{\operatorname{mono}}(Q,\mathcal{B})\to \operatorname{rep}(Q,\overline{\mathcal{B}})$. In many cases, $\rep (Q,\overline{\mathcal{B}})$ is easier to study than $\operatorname{mono}(Q,\mathcal{B})$. For example, if $\mathcal{B}$ is the module category of a radical square zero Nakayama Artin algebra, then $\overline{\mathcal{B}}$ is just the module category of a product of skew fields, and hence $\operatorname{rep}(Q,\overline{\mathcal{B}})$ can be computed using classical methods. This is how Theorem \ref{Gabriel-styleResult} is shown. 
	
	To make best use of the bijection \eqref{BijectionFromRepEquiv}, we would like to construct its inverse explicitly, so that we can obtain a description of the indecomposables in $\operatorname{mono}(Q,\mathcal{B})$ from the ones in $\operatorname{rep}(Q,\overline{\mathcal{B}})$. This is done in the following way. Let $(B_\mathtt{i},B_\alpha)$ be an object in $\operatorname{rep}(Q,\overline{\mathcal{B}})$. For each vertex $\mathtt{i}$, choose an object $\widehat{B}_{\mathtt{i}}$ in $\mathcal{B}$ with no nonzero injective summands and which is isomorphic to $B_\mathtt{i}$ in $\overline{\mathcal{B}}$. For each arrow $\alpha\colon \mathtt{i}\to \mathtt{j}$, choose a lift $\widehat{B}_\alpha\colon \widehat{B}_{\mathtt{i}}\to \widehat{B}_\mathtt{j}$ of $B_\alpha$ to $\mathcal{B}$. This gives a representation $(\widehat{B}_\mathtt{i},\widehat{B}_\alpha)$ in $\operatorname{rep}(Q,\mathcal{B})$, whose isomorphism class depends on the choice of lifts $\widehat{B}_\alpha$. Now, take the minimal right $\mono(Q,\mathcal{B})$-approximation of $(\widehat{B}_\mathtt{i},\widehat{B}_\alpha)$. This is denoted $\Mimo(\widehat{B}_\mathtt{i},\widehat{B}_\alpha)$ and called the Mimo-construction. For an explicit formula see Example \ref{Example:MimoFormulaQuiverRep} or \cite[Section 3a]{LZ13}. 
	\begin{theorem}[Theorems \ref{Theorem: Characterization of indecomposables} and \ref{Theorem: Characterization of indecomposablesArtin}]\label{MimoBijectionIndec}
		Let $Q$ be a finite acyclic quiver and $\mathcal{B}=\operatorname{mod}\Lambda$ for an Artin algebra $\Lambda$. The association $(B_\mathtt{i},B_\alpha)\mapsto \Mimo (\widehat{B}_\mathtt{i},\widehat{B}_\alpha)$ above gives an inverse to \eqref{BijectionFromRepEquiv}. In particular, it is independent of choice up to isomorphism.
	\end{theorem}
	This theorem holds more generally if $\mathcal{B}$ has injective envelopes and is noetherian or artinian or locally noetherian. 
 
 Note that the Mimo-construction was first introduced in \cite{RS08} for $\mathbb{A}_2$, where it played an important role in describing the Auslander--Reiten translation of the corresponding monomorphism category. In that case it can also be interpreted as a triangle functor \cite[Section 5]{Che12}. The formula for the Mimo-construction of an arbitrary finite acyclic quiver was given in \cite{LZ13}, where it was shown to give a right approximation. We are the first to show that it is a minimal right approximation in general. 
	
	To complete the picture, we give a description of the indecomposable injective objects in $\operatorname{mono}(Q,\mathcal{B})$. They are precisely given by 
	$f_!(J(\mathtt{i}))=(f_!(J(\mathtt{i}))_{\mathtt{j}},f_!(J(\mathtt{i}))_\alpha)$ where $J$ is indecomposable injective in $\mathcal{B}$ and
	\begin{equation}\label{DescriptionInjectives}
	f_!(J(\mathtt{i}))_\mathtt{k}=\bigoplus_{\substack{p\in Q_{\geq 0}\\ s(p)=\mathtt{i},t(p)=\mathtt{k}}}J \quad \text{and} \quad f_!(J(\mathtt{i}))_\alpha\colon \bigoplus_{\substack{p\in Q_{\geq 0}\\ s(p)=\mathtt{i},t(p)=\mathtt{k}}}J \to \bigoplus_{\substack{p\in Q_{\geq 0}\\ s(p)=\mathtt{i},t(p)=\mathtt{l}}}J
	\end{equation}
	for an arrow $\alpha\colon \mathtt{k}\to \mathtt{l}$, where $f_!(J(\mathtt{i}))_\alpha$ is induced by the identity map $J\xrightarrow{1} J$ between the components indexed by paths $p$ and $\alpha p$. We use this and Theorem \ref{MimoBijectionIndec} to give an explicit description of all indecomposable objects in monomorphisms categories for linearly oriented $\mathbb{A}_n$-quivers, for a non-linearly oriented $\mathbb{A}_4$-quiver,  and for the Kronecker quiver, see Subsections \ref{Subsection:Rad^2-zeroNakayama} and \ref{Subsection:Kronecker}.
	
	In most of the proofs we use the more abstract language of monads and Eilenberg--Moore categories, similar to  \cite{CL20} and  \cite{GKKP19}. This is to avoid the technical combinatorics arising from quiver representations, e.g. from the Mimo-construction (see Example \ref{Example:MimoFormulaQuiverRep}) and $f_!$ above. To see how they relate, consider the endofunctor
	\[
	X\colon \mathcal{C}\to \mathcal{C} \quad X(B_\mathtt{i})_{\mathtt{i}\in Q_0}=\left(\bigoplus_{\alpha\in Q_1,t(\alpha)=\mathtt{j}}B_{s(\alpha)}\right)_{\mathtt{j}\in Q_0}
	\]
	on $\mathcal{C}=\prod_{\mathtt{i}\in Q_0}\mathcal{B}$. The data of a representation of $Q$ in $\mathcal{B}$ is equivalent to an object $C\in \mathcal{C}$ and a morphism $X(C)\to C$. Furthermore, the representation lies in the monomorphism category if and only if $X(C)\to C$ is a monomorphism. Similarly, the data of a representation of $Q$ in $\overline{\mathcal{B}}$ is equivalent to an object $C\in \overline{\mathcal{C}}$ and a morphism $X(C)\to C$ in $\overline{\mathcal{C}}$. It follows from this that there are equivalences
	\[
	\operatorname{rep}(Q,\mathcal{B})\cong \mathcal{C}^{T(X)} \quad \text{and} \quad \Mono(X)\cong \operatorname{Mono}(Q,\mathcal{B}) \quad \text{and} \quad \operatorname{rep}(Q,\overline{\mathcal{B}})\cong \overline{\mathcal{C}}{}^{T(X)}
	\]
	where $\mathcal{C}^{T(X)}$ and $\overline{\mathcal{C}}{}^{T(X)}$ denote the Eilenberg--Moore categories of the free monad $T(X)$ on $\mathcal{C}$ and $\overline{\mathcal{C}}$, respectively, and $\Mono(X)$ is the full subcategory of $\mathcal{C}^{T(X)}$ where $X(C)\to C$ is a monomorphism. In this language several of the constructions become easier and more conceptual. For example, $f_!$ is the left adjoint of the forgetful functor $f^*\colon \mathcal{C}^{T(X)}\to \mathcal{C}$ from the Eilenberg--Moore category, and the Mimo-construction, given by the complicated formula in Example \ref{Example:MimoFormulaQuiverRep}, is just obtained by taking a particular pushout, see Definitions \ref{Definition: Q} and \ref{Definition: Mimo}.  Our results hold for any exact, locally nilpotent endofunctor $X$ which preserves injectives on an abelian category. The constructions and proofs use the Eilenberg--Moore category of its free monad as illustrated above. 
	
	Since we are working in a more general setting, our result also cover  generalizations of quiver representations, such as representations of modulations, see Example \ref{Example: Phyla}. They are known under several different names in the literature (and with varying hypothesis), such as representations of pro-species of algebras in \cite{Kul17}, representations of phyla in \cite{GKKP19}, representations over diagrams of abelian categories in \cite{DLLY22}, and twisted representations in \cite{GK05}. Monomorphism categories of modulations have connections to Gorenstein homological algebra \cite{Kul17,DLLY22}, and to cotorsion pairs and model structures  \cite{DLLY22}. 
 
 A well-studied class of modulations is given by prospecies over selfinjective rings. More explicitly, given a finite acyclic quiver $Q$, one associates to each vertex $\mathtt{i}$ a selfinjective algebra $\Lambda_\mathtt{i}$, and to each arrow $\alpha\colon \mathtt{i}\to \mathtt{j}$ a $\Lambda_\mathtt{i}$-$\Lambda_\mathtt{j}$-bimodule $M_\alpha$ which is projective as left $\Lambda_\mathtt{i}$-module and as right $\Lambda_\mathtt{j}$-module. In this case the monomorphism category is equal to the category of Gorenstein projectives modules over the tensor algebra $T(M)$ of $M=\bigoplus_{\alpha\in Q_1}M_\alpha$ by $\Lambda=\prod_{\mathtt{i}\in Q_0}\Lambda_{\mathtt{i}}$, and its stable category is equivalent to the singularity category of $T(M)$, see \cite{Kul17}. 
	
	As an important special case we have the prospecies in \cite{GLS16} associated to a tuple $(C,D,\Omega)$ where $C=(c_{\mathtt{i},\mathtt{j}})$ is a symmetrizable Cartan matrix with symmetrizer $D=\operatorname{diag}(d_\mathtt{i})$ and orientation $\Omega$. The associated tensor algebra is then isomorphic to the GLS-algebra $H=H(C,D,\Omega)$. We get the following result for these algebras. Here Cohen--Macaulay finite means having finitely many finitely generated indecomposable Gorenstein-projective modules up to isomorphism.
	
	\begin{theorem}\label{Theorem:GLSALgebras}
		Let $C=(c_{\mathtt{i},\mathtt{j}})_{\mathtt{i},\mathtt{j}\in I}$ be a symmetrizable Cartan matrix with symmetrizer $D=\operatorname{diag}(d_\mathtt{i}\mid \mathtt{i}\in I)$ and orientation $\Omega$. Assume $d_\mathtt{i}\leq 2$ for all $\mathtt{i}\in I$. Let $I'\subseteq I$ be the subset of all elements $\mathtt{i}$ for which $d_\mathtt{i}=2$, and let $C|_{I'\times I'}=(c_{\mathtt{i},\mathtt{j}})_{\mathtt{i},\mathtt{j}\in I'}$ be the corresponding submatrix of $C$. Then $H=H(C,D,\Omega)$ is Cohen--Macaulay finite if and only if  $C|_{I'\times I'}$ is Dynkin as a symmetric Cartan matrix. Furthermore, in this case there is a bijection between the positive roots of $C|_{I'\times I'}$ and the isomorphism classes of indecomposable objects in the singularity category of $H$. 
	\end{theorem}

	More generally, we show that there is a bijection between isomorphism classes of indecomposable objects in the singularity category of $H$, and finite-dimensional indecomposable representations over the quiver determined by the Cartan matrix $C|_{I'\times I'}$ with orientation $\Omega|_{I'\times I'}$. For this no finiteness assumptions are necessary. Note that for $D=\operatorname{diag}(2\mid i\in I)$ the  Gorenstein projective $H$-modules can be identified with monomorphic representations of $Q$ over $k[x]/(x^2)$, where $Q$ is the quiver associated to $C$ and $\Omega$. Hence, we recover Theorem \ref{Gabriel-styleResult} for $\Lambda=\mathbbm{k}[x]/(x^2)$.  
 
 Theorem \ref{Theorem:GLSALgebras} is a consequence of a more general result on modulations over radical square zero cyclic Nakayama algebras, see Theorem \ref{Theorem:MainThmModulations}. The bijection is obtained from an analogue of the epivalence in Theorem \ref{Theorem A} for modulations, and we get an explicit description of the indecomposable object in the singularity category associated to a given representation from an analogue of Theorem \ref{MimoBijectionIndec}.  The more general result also applies to modules over $\imath$quiver algebras studied in \cite{LW22,LW21a,LW21b}, see Subsection \ref{Subsection:ModulationsiQuiverGLS}. There the category of Gorenstein projective modules plays an important role, in particular since it is used to realize $\imath$quantum groups via Hall algebras \cite{LW22,LW21a}, and since for Dynkin quivers it is equivalent to the category of projectives over the regular Nakajima--Keller--Scherotzke categories \cite{LW21b}.
	
	The structure of the paper is as follows. Section \ref{Section:Preliminaries} contains the necessary background on monads and exact categories, respectively. In Section \ref{Section: The free monad and the monomorphism category} we study the free monad on an endofunctor. The section is divided into four parts: In Subsection \ref{subsection:The free monad}  we define the free monad of an endofunctor, and provide examples of it. In Subsection \ref{subsection:The free monad of an abelian category} we restrict to the abelian case.  Subsection \ref{Subsection: TopFunctor} introduces the top functor and recalls its basic properties. Finally, Subsection \ref{Subsection:The monomorphism category} introduces the key player of the paper, the monomorphism category. Section \ref{Section: InjectiveObjects} deals with injective objects and the existence of injective envelopes for the monomorphism category. In Section \ref{Section:A representation equivalence} we prove Theorem \ref{Representation equivalence}. It starts with a discussion of contravariantly finiteness of the monomorphism category in Subsection \ref{Subsection:Contravariantly finiteness}, proceeds with the proof of the existence of the epivalence in Subsection \ref{Subsection:The general case}, and finishes by discussing the hereditary case in Subsection \ref{Subsection:The hereditary case}. Section \ref{The Mimo construction} introduces the $\Mimo$-construction in our language. Subsection \ref{Subsection:Definition and properties} deals with the general case while Subsection \ref{Subsection:Mimo for modulations} makes the setup explicit in the case of modulations. Section \ref{subsection: A characterization of the indecomposable objects} contains the proof of Theorem \ref{MimoBijectionIndec}, with a short discussion on maximal injective summands in the beginning. 
 
 Section \ref{Section:Applications} discusses applications to representations of quivers over Artin algebras. Subsection \ref{Subsecition:StableEquivalences} is on stable equivalences and the induced bijections between indecomposables in monomorphism categories. In Subsection \ref{Subsection:Local uniserial rings of Loewy length 3} we prove Theorem \ref{Introduction:BijectionLoewyLength3} and discuss connections to other results in the literature. In Subsection \ref{Subsection:Rad^2-zeroNakayama} we prove Theorem \ref{Gabriel-styleResult}, and explicitly compute indecomposables for monomorphism categories over radical square zero Nakayama algebras, using Theorem \ref{MimoBijectionIndec}. In Subsection \ref{Subsection:Kronecker} we compute the indecomposables in the monomorphism category of the Kronecker quiver over $\mathbbm{k}[x]/(x^2)$. In Section \ref{Section:ApplicationToModulations} we apply our results to representations of modulations over radical square zero cyclic Nakayama algebras to obtain Theorem \ref{Theorem:MainThmModulations}. We then consider $\imath$quivers and GLS algebras in Subsection \ref{Subsection:ModulationsiQuiverGLS}, and prove Theorem \ref{Theorem:GLSALgebras} from the introduction.

	\section{Preliminaries}\label{Section:Preliminaries}
	
	\subsection{Notation}\label{Section: Notation}
	
	We ignore set-theoretic issues in this paper. 
	All categories are assume to be additive and idempotent complete, and all functors are assumed to be additive. For functors $F\colon \mathcal{C}\to \tilde{\mathcal{C}}$ and $G\colon \tilde{\mathcal{C}}\to \mathcal{C}$ we write $F\dashv G$
	to denote that $F$ is left adjoint to $G$.  For a ring $\Lambda$ we let $\Mod \Lambda$, respectively $\operatorname{mod}\Lambda$, denote the category of right $\Lambda$-modules, respectively finitely presented right $\Lambda$-modules. Throughout the paper $\mathbbm{k}$ is a commutative ring.

	\subsection{Monads and modules over monads}\label{Section:Monads and relative Nakayama functors}
	In this subsection we recall the definition of a monad and its Eilenberg--Moore category
 . Our main examples arise from modulations on quivers, see Example \ref{Example: Phyla}.

	\begin{defn}
		Let $\mathcal{C}$ be a category. A \emphbf{monad} on $\mathcal{C}$ is a tuple $(T,\eta,\mu)$ where $T\colon \mathcal{C}\to \mathcal{C}$ is a functor and $\eta\colon \Id\to T$ and $\mu\colon T^2\to T$ are natural transformations such that the diagrams
		\[\begin{tikzcd}
		T^3\arrow{r}[swap]{T(\mu)}\arrow{d}{\mu_T}&T^2\arrow{d}{\mu}\\
		T^2\arrow{r}[swap]{\mu}&T
		\end{tikzcd}
		\quad \text{ and }
		\begin{tikzcd}
		&T\arrow[equals]{rd}\arrow[equals]{ld}\\
		T\arrow{r}[swap]{T(\eta)}&T^2\arrow{u}{\mu}&T\arrow{l}{\eta_T}
		\end{tikzcd}\]
		commute.  By standard abuse of notation we sometimes denote a monad $(T,\eta,\mu)$ simply by $T$.
	\end{defn} 
	
	An important source of instances of monads are adjunctions: If $\eta$ and $\varepsilon$ denote the unit and counit of an adjunction $L\dashv R$, then the tuple $(R\circ L,\eta,R(\varepsilon_L))$
	defines a monad. We refer to \cite[Section VI.1, p.138]{McL98} for more details on this.
	
	Conversely, monads give rise to adjunctions via the Eilenberg--Moore category, which we recall next. This is sometimes called the category of algebras or the category of modules over the monad in the literature.
	
	\begin{defn}\label{Eilenberg--Moore category}
		Let $(T,\eta,\mu)$ be a monad on $\mathcal{C}$.
		\begin{enumerate}
			\item A $T$\emphbf{-module} is a pair $(M,h)$ where $M$ is an object in $\mathcal{C}$ and $h\colon T(M)\to M$ is a morphism  satisfying $h\circ \eta_M=1_M$ and $h\circ \mu_M = h\circ T(h)$.
			
			\item The \emphbf{Eilenberg--Moore category} $\mathcal{C}^{T}$ of $T$ is the category whose objects are $T$-modules, and where a morphism $\map\colon (M,h)\to (M',h')$ between $T$-modules is given by a morphism $\map\colon M\to M'$ in $\mathcal{C}$ satisfying $\map\circ h=h'\circ T(\map)$. 
			
			\item $f^*\colon \mathcal{C}^{T}\to \mathcal{C}$ denotes the forgetful functor given by $f^*(M,h)=M$ and $f^*(\map)=\map$.
			\item $f_!\colon \mathcal{C}\to \mathcal{C}^{T}$ denotes the functor given by $f_!(M)=(T(M),\mu_M)$ and $f_!(\map)=T(\map)$.
			
		\end{enumerate}
	\end{defn} 
	
	\begin{prop}
		Let $T$ be a monad on $\mathcal{C}$. Then $f^*\colon \mathcal{C}^T\to \mathcal{C}$ is right adjoint to $f_!\colon \mathcal{C}\to \mathcal{C}^T$.	
	\end{prop}	
	
	\begin{proof}
		This follows from \cite[Theorem VI.2.1]{McL98}.
	\end{proof}

	Fix $\eta$ and $\varepsilon$ to be the unit and counit of the adjunction $f_!\dashv f^*$.
	We finish by giving a sufficient criterion for the Eilenberg--Moore category to be abelian.
	
	\begin{prop}{\cite[Proposition 5.3]{EM65}}\label{Proposition: exactness in Eilenberg--Moore}
		Let $\mathcal{C}$ be an abelian category, and let $(T,\eta,\mu)$ be a monad on $\mathcal{C}$. Assume that $T$ is a right exact functor. The following hold:
		\begin{enumerate}
			\item The Eilenberg--Moore category $\mathcal{C}^T$ is abelian.
			\item A sequence $(M,h)\to (M',h')\to (M'',h'')$ in $\mathcal{C}^{T}$ is exact if and only if the sequence $M\to M'\to M''$ in $\mathcal{C}$ is exact.
		\end{enumerate}
	\end{prop}

	\subsection{Exact categories}\label{Exact categories}
	
	Here we recall some basic properties of exact categories, in particular results on injective envelopes.
	
	An \emphbf{exact category} is an additive category $\mathcal{E}$ endowed with a class of kernel-cokernel pairs, called \emphbf{conflations}, satisfying certain properties, see \cite[Appendix A]{Kel90} or \cite{Bue10} for more details. If  $E_1\xrightarrow{i}E_2\xrightarrow{p}E_3$ is a conflation, then $i$ is called an \emphbf{inflation} and $p$ is called a \emphbf{deflation}. 
	Let $\mathcal{E}'$ be a full subcategory of an exact category $\mathcal{E}$. We say that $\mathcal{E}'$ is \emphbf{extension-closed} if for any conflation $E_1\to E_2\to E_3$ with $E_1$ and $E_3$ in $\mathcal{E}'$, the middle term $E_2$ must be in $\mathcal{E}'$. In this case $\mathcal{E}'$ inherits an exact structure whose conflations are the conflations in $\mathcal{E}$ where all the terms are in $\mathcal{E}'$ see \cite[Lemma 10.20]{Bue10}. 
	
	An object $I$ in an exact category $\mathcal{E}$ is called $\emphbf{injective}$ if for any inflation $E\to E'$ the morphism $\operatorname{Hom}_{\mathcal{E}}(E',I)\to \operatorname{Hom}_{\mathcal{E}}(E,I)$ is surjective. The exact category $\mathcal{E}$ is said to have \emphbf{enough injectives} if for any object $E$ in $\mathcal{E}$ there exists an inflation $E\to I$ with $I$ injective. In this case we let $\overline{\mathcal{E}}$ denote the quotient of $\mathcal{E}$ by the ideal of morphisms factoring through injective objects.
	
	Let $g\colon E_1\to E_2$ be a morphism in $\mathcal{E}$. We say that $g$ is \emphbf{left minimal} if any morphism $g'\colon E_2\to E_2$ satisfying $g'\circ g=g$ is an isomorphism. An \emphbf{injective envelope} of an object $E$ in $\mathcal{E}$ is a left minimal inflation $i\colon E\to I$ where $I$ is injective. Note that an injective envelope of an object is unique up to isomorphism. The \emphbf{radical} of $\mathcal{E}$ is the ideal defined by
	\[
	\operatorname{Rad}_{\mathcal{E}}(E,E')\coloneqq\{g\in \operatorname{Hom}_{\mathcal{E}}(E,E')\mid 1_{E'}-g\circ g'\text{ is invertible for all } g'\in \operatorname{Hom}_{\mathcal{E}}(E',E)\}
	\]
	for all $E,E'\in \mathcal{E}$. Recall that the radical is symmetric, and so we have the equality 
	\[
	\operatorname{Rad}_{\mathcal{E}}(E,E')\coloneqq\{g\in \operatorname{Hom}_{\mathcal{E}}(E,E')\mid 1_E-g'\circ g\text{ is invertible for all } g'\in \operatorname{Hom}_{\mathcal{E}}(E',E)\}.
	\]
	We refer to \cite[Section 2.1]{Kra22} for more information on injective envelopes, and to \cite[Section 2]{Kra15} and \cite{Kel64} for the radical. We only need the following results relating them. 
	
	\begin{lem}\label{Lemma:RadicalProperties}
		Let $\mathcal{E}$ be a exact category, and let $E\in \mathcal{E}$. The following hold:
		\begin{enumerate}
			\item\label{Lemma:RadicalProperties:1} If $i\colon E\to I$ is an injective envelope, then $p\colon I\rightarrow \operatorname{coker}i$ is in the radical of $\mathcal{E}$.
			\item\label{Lemma:RadicalProperties:2} Assume $\mathcal{E}$ is abelian with the natural exact structure. Furthermore, assume $\mathcal{E}$ has injective envelopes. Let $g\colon J\to E$ be a morphism in $\mathcal{E}$. If $J$ is injective and $E$ has no nonzero injective summands, then the inclusion $\operatorname{ker}g\to J$ is an injective envelope. In particular, $g$ is in the radical of $\mathcal{E}$.
		\end{enumerate}
	\end{lem}
	\begin{proof}
		To prove statement \eqref{Lemma:RadicalProperties:1} let $k\colon \operatorname{coker}i\to I$ be an arbitrary morphism. Then the endomorphism $k'=1_I-k\circ p$ satisfies $k'\circ i=i$. Since $i$ is an injective envelope, we must have that $k'$ is an isomorphism. Since $k$ was arbitrary, it follows that $p$ is in the radical of $\mathcal{C}$.
		
		To prove \eqref{Lemma:RadicalProperties:2} let $u\colon \operatorname{ker}g\to E(\operatorname{ker}g)$ denote the injective envelope of $\operatorname{ker}g$. Then the monomorphism $i\colon\operatorname{ker}g\to J$ lifts to a split monomorphism  $E(\operatorname{ker}g)\to J$. Consider the commutative diagram where the lower row is a split exact sequence 
		\[
		\begin{tikzcd}
		0\arrow{r}&\operatorname{ker}g\arrow{r}{i}\arrow{d}{u}&J\arrow{d}{1_J}\arrow{r}{g}&E\arrow[dashed]{d}&\\
		0\arrow{r}&E(\operatorname{ker}g)\arrow{r}&J\arrow{r}&J'\arrow{r}&0.
		\end{tikzcd}
		\]
		Since the left square commutes, we get an induced morphism $\operatorname{coker}i\to J'$. Since $J'$ is injective, we can lift this to a morphism $E\to J'$ such that the rightmost square in the diagram commutes. Since $J\to J'$ is a split epimorphism, the morphism $E\to J'$ is a split epimorphism. Therefore $J'$ must be a summand of both $J$ and $E$. But being a summand of $J$ implies that $J'$ is injective, and $E$ has no nonzero injective summands. Therefore $J'=0$. Hence $i$ is an injective envelope. Since $g$ factors through the cokernel of $i$, it must be in the radical by part \eqref{Lemma:RadicalProperties:1}.
	\end{proof}
	
	We finish by proving a uniqueness result on the decomposition of an object by a maximal injective summand. It holds for abelian categories with injective envelopes by the previous lemma. It also holds for the monomorphism category by Lemma \ref{No injective summands}.
	
	\begin{lem}\label{Lemma:UniquenessMaxInjSummand}
		Let $\mathcal{E}$ be an exact category, and let $E_1$ and $E_2$ be objects in $\mathcal{E}$. Assume that any morphism from an injective object to $E_1$ or $E_2$ is in the radical of $\mathcal{E}$. Let  
		\[
		E_1\oplus I_1\xrightarrow{\cong} E_2\oplus I_2
		\]
		be an isomorphism where $I_1$ and $I_2$ are injective. Then the restrictions $E_1\to E_2$ and $I_1\to I_2$ are isomorphisms.
	\end{lem}
	
	\begin{proof}
		Let 
		\[
		\phi=\begin{pmatrix}
		\phi_1&\phi_2\\
		\phi_3&\phi_4
		\end{pmatrix}\colon E_1\oplus I_1\to E_2\oplus I_2
		\quad \text{and} \quad \psi=\begin{pmatrix}
		\psi_1&\psi_2\\
		\psi_3&\psi_4
		\end{pmatrix}\colon E_2\oplus I_2\to E_1\oplus I_1
		\]
		denote the isomorphism and its inverse. Then we have that $\psi_3\circ \phi_2 + \psi_4\circ \phi_4=1_{I_1}$. By assumption, $\phi_2$ is in the radical of $\mathcal{E}$ . Hence, by definition of the radical the composite $\psi_4\circ \phi_4=1_{I_1}-\psi_3\circ \phi_2$ is an isomorphism. By a similar argument the composite $\phi_4\circ \psi_4$ is also an isomorphism. Hence $\phi_4\colon I_1\to I_2$ must itself be an isomorphism. The fact that $\phi_1$ is an isomorphism is proved in the same way.
	\end{proof}

	\section{The monomorphism category of the free monad}\label{Section: The free monad and the monomorphism category}
	
	Fix a $\mathbbm{k}$-linear additive category $\mathcal{C}$.  Throughout the section $X\colon \mathcal{C}\to \mathcal{C}$ is assumed to be a $\Bbbk$-linear functor which is \emphbf{locally nilpotent}, i.e. for any object $M\in \mathcal{C}$ there exists an $n\geq 0$ such that $X^n(M)=0$. From Subsection \ref{subsection:The free monad of an abelian category} we assume $\mathcal{C}$ is abelian and $X$ is exact, and from Section \ref{Section: InjectiveObjects} onwards we also assume that $X$ preserves injective objects. The assumptions on $X$ are introduced  to capture the essential properties of monads arising from representations of finite acyclic quivers in additive and abelian categories, see Example \ref{Example:Quiver Representations}. This is done using the free monad on $X$, which we define. In addition to this we define the monomorphism category and the analogue of the top functor. 
	
	\begin{rmk}\label{Remark:RelativeNakayama}
		The assumption we make on $X$ differs from the ones in \cite{GKKP19}. Several of the results in \cite{GKKP19} rely on the existence of a relative Nakayama functor on the Eilenberg--Moore category of the free monad on $X$, and this is not necessary to assume in this paper. On the other hand, many of the proofs in \cite{GKKP19} still go through. Another technical condition we use in \cite{GKKP19} is a relative version of Nakayama's lemma, which says that if $M$ is non-zero, then there are no epimorphisms $X(M)\twoheadrightarrow M$, see \cite[Lemma 6.25]{GKKP19}. Under the assumption that $X$ is locally nilpotent, this is automatically satisfied when $X$ preserves epimorphisms, and in particular if $X$ is an exact functor on an abelian category.
	\end{rmk}
	
	\subsection{The free monad}\label{subsection:The free monad}
	The definition of the free monad mimics the construction of the path algebra of a quiver, and more generally the tensor algebra of a bimodule over an algebra.
	
	\begin{defn}\label{Def: Free monad} The \emphbf{free monad} on $X$ is the monad $(T(X),\eta,\mu)$ where $T(X)\colon \mathcal{C}\to \mathcal{C}$ is given by
		\[
		T(X)(M)=\coprod_{i\geq 0}X^i(M)
		\]
		and where $\eta\colon \Id\to T(X)$ is the canonical inclusion and $\mu\colon T(X)\circ T(X)\to T(X)$ is given componentwise by the canonical identification $X^iX^j(M)\xrightarrow{\cong}X^{i+j}(M)$.
	\end{defn}
	
	Since $X$ is locally nilpotent, the coproduct $\coprod_{i\geq 0}X^i(M)$ is finite for each object $M\in \mathcal{C}$. We therefore identify it with the direct sum $\bigoplus_{i\geq 0}X^i(M)$. 
	
	\begin{rmk}
		In \cite{GKKP19} it is assumed that $X$ preserves countable coproducts when defining the free monad. However this is only used to conclude that the canonical map $X(\coprod_{i\geq 0}X^i(M))\to \coprod_{i\geq 1}X^i(M)$ is an isomorphism, which follows here from the fact that since $X$ is locally nilpotent, the coproduct is finite and is therefore preserved by the additive functor $X$.
	\end{rmk}
	
	The Eilenberg--Moore category of a free monad $T(X)$ has an alternative simpler description. Let $(X\Downarrow \Id_{\mathcal{C}})$ be the category whose objects are pairs $(M,h_1)$ where $M\in \mathcal{C}$ and $h_1\colon X(M)\to M$ is a morphism, and where a morphism $\map\colon (M,h_1)\to (M',h_1')$ in $(X\Downarrow \Id_{\mathcal{C}})$ is a morphism $\map\colon M\to M'$ in $\mathcal{C}$ satisfying $\map\circ h_1= h_1'\circ X(\map)$.  Note that we have a functor
	\begin{align*}
	\mathcal{C}^{T(X)}\to (X\Downarrow \Id_{\mathcal{C}}) \quad (\bigoplus_{i\geq 0}X^i(M)\xrightarrow{h}M)\mapsto (M,h_1) 
	\end{align*} 
	where $h_1\colon X(M)\to M$ is the restriction of $h$ to $X(M)$.
	
	\begin{lem}\label{Eilenberg--Moore category of free monad}
		The functor $\mathcal{C}^{T(X)}\to (X\Downarrow \Id_{\mathcal{C}})$ above is an isomorphism of categories.
	\end{lem}
	
	\begin{proof}
		This follows from the proof of Lemma 5.18 in \cite{GKKP19}.
	\end{proof}
	
	Note that a similar description for the Eilenberg--Moore category is given in \cite[Remark 2.2]{CL20}. We will identify the categories $\mathcal{C}^{T(X)}$ and $(X\Downarrow \Id_{\mathcal{C}})$. We use sans serif typestyle $\module{M},\module{N},\dots$ to denote objects in $\mathcal{C}^{T(X)}$, so that the same letter without the sans serif typestyle denotes the underlying object in $\mathcal{C}$, i.e. $M=f^*(\module{M})$ and $N=f^*(\module{N})$. The induced morphism $X(M)\to M$ is then denoted by $h_{\module{M}}$ and called the \emphbf{structure map} of $\module{M}$.
	
	In Section \ref{Section:Monads and relative Nakayama functors} we saw that the forgetful functor $f^*\colon \mathcal{C}^{T(X)}\to \mathcal{C}$ has a left adjoint $f_!\colon \mathcal{C}\to \mathcal{C}^{T(X)}$.  It is given by $f_!(M)=(\bigoplus_{i\geq 0}X^i(M),\iota_M)$ where the structure map $\iota_M$ is the canonical inclusion
	\begin{equation}\label{StructureMapRelProj}
	\iota_M\colon \bigoplus_{i\geq 1}X^i(M)\to \bigoplus_{i\geq 0}X^i(M).
	\end{equation}
	Any summand of an object of the form $f_!(M)$ is called \emphbf{relative projective}. If $\mathcal{C}=\operatorname{Mod} \Lambda$ for a semisimple Artin algebra $\Lambda$, then the relative projectives coincide with the projectives in $\mathcal{C}^{T(X)}$, cf. Proposition \ref{Corollary:MonoAndInjectiveComponents}. The following lemma implies that the relative projectives behave as projectives for objectwise split epimorphisms in $\mathcal{C}^{T(X)}$.
	
	\begin{lem}\label{Lemma:RelativeProjective}
		Let $N$ be an object in $\mathcal{C}$ and let $g\colon \module{M}\to \module{M}'$ be a morphism in $\mathcal{C}^{T(X)}$. Assume $f^*(g)$ is a split epimorphism. Then any morphism $f_!(N)\to \module{M}'$ factors through $g$.
	\end{lem}
	
	\begin{proof}
		By the adjunction $f_!\dashv f^*$, the map 
		\[
		\operatorname{Hom}_{\mathcal{C}^{T(X)}}(f_!(N),g)\colon \operatorname{Hom}_{\mathcal{C}^{T(X)}}(f_!(N),\module{M})\to \operatorname{Hom}_{\mathcal{C}^{T(X)}}(f_!(N),\module{M}')
		\]
		is isomorphic to the map
		\[
		\operatorname{Hom}_{\mathcal{C}}(N,f^*(g))\colon \operatorname{Hom}_{\mathcal{C}}(N,f^*(\module{M}))\to \operatorname{Hom}_{\mathcal{C}}(N,f^*(\module{M}')).
		\]
		The latter must be an epimorphism since $f^*(g)$ is a split epimorphism. This proves the claim.
	\end{proof}
	
	A sequence $\module{M}_1\to \module{M_2}\to \module{M}_3$ in $\mathcal{C}^{T(X)}$ is called \emphbf{objectwise split exact} if the sequence 
	\[
	f^*(\module{M}_1)\to f^*(\module{M_2})\to f^*(\module{M}_3)
	\]
	is split exact in $\mathcal{C}$. We show that such a sequence is an exact sequence in $\mathcal{C}^{T(X)}$, i.e. a kernel-cokernel pair. Note that this holds even though $\mathcal{C}^{T(X)}$ is not assumed to be abelian. We also show that the free monad $T(X)$ has relative global dimension one, i.e. any object in $\mathcal{C}^{T(X)}$ has a objectwise split resolution of length one by relative projective objects. Here the structure morphism of $f_!(M)$ is considered as a morphism $\iota_M\colon f_!X(M)\to f_!(M)$ in $\mathcal{C}^{T(X)}$, and $\varepsilon_M$ denotes the counit of the adjuntion $f_!\dashv f^*$ at $M$. 
	
	\begin{lem}\label{relative global dimension 1}
		The following hold:
		\begin{enumerate}
			\item\label{relative global dimension 1:1} Any objectwise split exact sequence is an exact sequence.
			\item\label{relative global dimension 1:2} For each $\module{M}\in \mathcal{C}^{T(X)}$ the sequence
			\[
			0\to f_!X(M)\xrightarrow{\iota_M-f_!(h_\module{M})}f_!(M)\xrightarrow{\varepsilon_{\module{M}}}\module{M}\to 0
			\]
			is objectwise split exact.
		\end{enumerate}
		
	\end{lem}
	
	\begin{proof}
		Let $\module{M}_1\to \module{M}_2\to \module{M}_3$ be a objectwise split exact sequence. By \cite[Proposition 4.3.1]{Bor94a} the map $\module{M}_1\to \module{M_2}$ is a kernel of $\module{M}_2\to \module{M}_3$ since $f^*(\module{M}_1)\to f^*(\module{M}_2)$ is a kernel of $f^*(\module{M}_2)\to f^*(\module{M}_3)$. Since $T(X)$ preserves split exact sequences, \cite[Proposition 4.3.2]{Bor94a} implies that $\module{M}_2\to \module{M}_3$ is a cokernel of $\module{M}_1\to \module{M}_2$ since  $f^*(\module{M}_1)\to f^*(\module{M}_2)\to f^*(\module{M}_3)$ is split exact. This proves \eqref{relative global dimension 1:1}.
		
		The fact that 
		\[
		0\to f^*f_!X(M)\xrightarrow{f^*(\iota_M-f_!(h_\module{M}))}f^*f_!(M)\xrightarrow{f^*(\varepsilon_{\module{M}})}f^*(\module{M})\to 0
		\]
		is a split exact sequence can be shown in a similar way as in the proof of Lemma 6.17 in \cite{GKKP19}. Note that the infinite sums in that proof are finite in our case, since $X$ is locally nilpotent. 
	\end{proof}
	
	Next we show that the relative projectives are always isomorphic to objects of the form $f_!(M)$. They can also be characterized by their structure morphism being a split monomorphism.
	
	\begin{prop}\label{Proposition: direct summands of f_!}
		Let $\module{M}\in \mathcal{C}^{T(X)}$. The following are equivalent:
		\begin{enumerate}
			\item\label{Proposition: direct summands of f_!:1} $\module{M}$ is relative projective.
			\item\label{Proposition: direct summands of f_!:2} $h_{\module{M}}\colon X(M)\to M$ is a split monomorphism in $\mathcal{C}$.
			\item\label{Proposition: direct summands of f_!:3} $\module{M}$ is isomorphic to an object of the form $f_!(N')$ for $N'\in \mathcal{C}$.
		\end{enumerate}	
	\end{prop}
	
	\begin{proof}
		Clearly the class of objects $\module{M}\in \mathcal{C}^{T(X)}$ for which $h_{\module{M}}$ is a split monomorphism is closed under direct summands. Since $\iota_N$ is a split monomorphism for all $N\in \mathcal{C}$, this shows \eqref{Proposition: direct summands of f_!:1}$\Rightarrow$\eqref{Proposition: direct summands of f_!:2}. 
		
		For \eqref{Proposition: direct summands of f_!:2}$\Rightarrow$\eqref{Proposition: direct summands of f_!:3}, consider the split exact sequence
		\[
		0\to X(M)\xrightarrow{h_{\module{M}}}M\to N'\to 0
		\] 
		in $\mathcal{C}$ and choose a section $i\colon N'\to M$. Then for each $m\geq 0$ the morphism $X^m(i)\colon X^m(N')\to X^m(M)$ is a section to the split exact sequence
		\[
		0\to X^{m+1}(M)\xrightarrow{X^m(h_{\module{M}})}X^m(M)\to X^m(N')\to 0.
		\]
		Using the $X^m(i)$'s, we get isomorphisms 
		\[
		M\cong N'\oplus X(M)\cong N'\oplus X(N')\oplus X^2(M)\cong \dots \cong N'\oplus X(N')\oplus \dots \oplus X^n(N')
		\]
		where $n$ is some integer with $X^{n+1}(M)=0$. Since these isomorphisms commute with the structure morphisms of $\module{M}$ and $f_!(N')$, we get an isomorphism $\module{M}\cong f_!(N')$.  The remaining direction \eqref{Proposition: direct summands of f_!:3}$\Rightarrow$\eqref{Proposition: direct summands of f_!:1} is obvious.
	\end{proof}	
	
	\begin{cor}\label{cor:f_!Indecomposable}
		An object $M\in\mathcal{C}$ is indecomposable if and only if $f_!(M)$ is indecomposable.
	\end{cor}
	
	\begin{proof}
		By Proposition \ref{Proposition: direct summands of f_!} any summand of $f_!(M)$ is of the form $f_!(N)$. Taking the cokernel of their structure morphisms, we see that this is equivalent to $N$ being a summand of $M$. Hence $f_!(M)$ is indecomposable if and only if $M$ is indecomposable.
	\end{proof}

	\begin{ex}\label{Example: Phyla}
		Let $Q=(Q_0,Q_1)$ be a quiver, where $Q_0$ and $Q_1$ denote the set of vertices and arrows of $Q$, respectively. A $\mathbbm{k}$-\emphbf{modulation} $\mathfrak{B}$ of $Q$ is an assignment of a $\mathbbm{k}$-linear additive category $\mathcal{B}_\mathtt{i}$ to each vertex $\mathtt{i}\in Q_0$ and a $\mathbbm{k}$-linear functor $F_\alpha\colon \mathcal{B}_\mathtt{i}\to \mathcal{B}_\mathtt{j}$ to each arrow $\alpha\colon \mathtt{i}\to \mathtt{j}$. Associated to a modulation $\mathfrak{B}$ we have the category $\rep \mathfrak{B}$ of \emphbf{$\mathfrak{B}$-representations}. Explicitly, its objects are collections $(B_\mathtt{i},B_\alpha)_{\mathtt{i}\in Q_0,\alpha\in Q_1}$ where $B_\mathtt{i}$ is an object of $\mathcal{B}_\mathtt{i}$ for each vertex $\mathtt{i}$ and $B_\alpha\colon F_\alpha(B_\mathtt{i})\to B_\mathtt{j}$ is a morphism in $\mathcal{B}_\mathtt{j}$ for each arrow $\alpha\colon \mathtt{i}\to \mathtt{j}$. A morphism of $\mathfrak{B}$-representations $(B_\mathtt{i},B_\alpha)\to (B'_\mathtt{i},B'_\alpha)$ is a collection of morphism $(\varphi_\mathtt{i}\colon B_\mathtt{i}\to B'_\mathtt{i})_{\mathtt{i}\in Q_0}$ such that the following diagram commutes for every arrow $\alpha\colon \mathtt{i}\to \mathtt{j}$:
		\[
		\begin{tikzcd}
		F_\alpha(B_\mathtt{i})\arrow{r}{B_\alpha}\arrow{d}{F_\alpha(\varphi_\mathtt{i})} &B_\mathtt{j}\arrow{d}{\varphi_\mathtt{j}}\\
		F_\alpha(B'_\mathtt{i})\arrow{r}{B'_\alpha} &B'_\mathtt{j}
		\end{tikzcd}
		\] 
		Note that the category of representations can be identified with the sections of the Grothendieck construction of a certain functor obtained from the $\mathbbm{k}$-modulation, see Remark 4.6 in \cite{GKKP19}.
		
		Assume $Q$ is finite and acyclic and set $\mathcal{C}=\prod_{\mathtt{i}\in Q_0}\mathcal{B}_\mathtt{i}$. Define the functor
		\begin{align}\label{DefXModulation}
		& X\colon \mathcal{C}\to \mathcal{C}, \quad  \quad (B_\mathtt{i})_{\mathtt{i}\in Q_0}\mapsto (\bigoplus_{\substack{\alpha\in Q_1\\t(\alpha)=\mathtt{i}}}F_{\alpha}(B_{s(\alpha)}))_{\mathtt{i}\in Q_0}.
		\end{align} 
		By the assumptions on $Q$ it follows that $X$ is nilpotent, i.e. $X^n=0$ for some $n\gg 0$. The category $\rep \mathfrak{B}$ of $\mathfrak{B}$-representations can be identified with Eilenberg--Moore category $\mathcal{C}^{T(X)}$ of the free monad $T(X)$ on $\mathcal{C}$. We  describe the Eilenberg--Moore adjunction $f_!\dashv f^*$ explicitly, following \cite[Meta-Example 6.20]{GKKP19}. The forgetful functor is given by
		\[
		f^*\colon \rep \mathfrak{B}\to \mathcal{C}, \quad \quad f^*(B_\mathtt{i},B_\alpha)_{\mathtt{i}\in Q_0,\alpha\in Q_1}=(B_\mathtt{i})_{\mathtt{i}\in Q_0}.
		\] 
		For $f_!$, we need some notation. Let $Q_{\geq 0}$ denote the collection of all paths in $Q$, and for $p\in Q_{\geq 0}$ let $s(p)$ and $t(p)$ denote its source and target, respectively. If $p=\alpha_n\alpha_{n-1}\dots \alpha_1$ set
		\begin{align*}
		F_p&\coloneqq F_{\alpha_n}\circ F_{\alpha_{n-1}}\circ \dots\circ F_{\alpha_1}\colon \mathcal{B}_{s(p)}\to \mathcal{B}_{t(p)}.
		\end{align*} 
		The functor $f_!$ applied to $B=(B_\mathtt{i})_{\mathtt{i}\in Q_0}$ is then given by
		\begin{align}\label{Formula:f_!}
		f_!(B)&=\left(\bigoplus_{\substack{p\in Q_{\geq 0}\\t(p)=\mathtt{i}}}F_p(B_{s(p)}), f_!(B)_\alpha\right)_{\substack{\mathtt{i}\in Q_0\\\alpha\in Q_1}}.
		\end{align}
		Here
		\begin{align*}
		f_!(B)_\alpha&\colon \bigoplus_{\substack{p\in Q_{\geq 0}\\t(p)=s(\alpha)}} F_\alpha F_p(B_{s(p)})\to \bigoplus_{\substack{q\in Q_{\geq 0}
				\\t(q)=t(\alpha)}}F_q(B_{s(q)})
		\end{align*}
		is induced by the identity map $F_\alpha F_p(B_{s(p)})\xrightarrow{1}F_q(B_{s(q)})$ for $q=\alpha p$. For more details on this construction see \cite[Section 5]{GKKP19}. 
	\end{ex}

	\begin{ex}\label{Example:LcaollynilpotentAnyQuiver}
		Representations of infinite quivers and their monomorphism categories have also been considered in for example \cite{EOT04,EE05,EEG09}. Example \ref{Example: Phyla} can be extended to a modulation $\mathfrak{B}$ on any quiver $Q$ (not necessarily finite or acylic). 
   However, to ensure that powers of the functor $X$ given by \eqref{DefXModulation} are well-defined, we need to restrict to objects $(B_{\mathtt{i}})_{\mathtt{i}\in Q_0}$ of $\prod_{\mathtt{i} \in Q_0}\mathcal{B}_\mathtt{i}$ satisfying:
  \begin{enumerate}
			\item[(A1)] For any $n\geq 1$ and $\mathtt{j}\in Q_0$, if $\kappa_n$ is the cardinality of the set 
   \[
   \{\alpha\in Q_1\mid t(\alpha)=\mathtt{j} \text{ and there exists a path }p\text{ of length }n-1 \text{ with } t(p)=s(\alpha) \text{ and }B_{s(p)}\neq 0 \}
   \]
   then $\mathcal{B}_\mathtt{j}$ admits $\kappa_n$-coproducts. 
   \end{enumerate}
   In particular, this holds for a vertex $\mathtt{j}$ if there are only finitely many arrows with target $\mathtt{j}$, or if $\mathcal{B}_\mathtt{j}$ admits coproducts of any cardinality. 
   
   To ensure that $X$ is locally nilpotent, we must restrict to objects $(B_{\mathtt{i}})_{\mathtt{i}\in Q_0}$ satisfying: 
		\begin{enumerate}
			\item[(B)] There exists an integer $n\geq 0$ (depending on $(B_{\mathtt{i}})_{\mathtt{i}\in Q_0}$) such that $B_{s(p)}= 0$ for any path $p$ of length greater than $n$. 
		\end{enumerate}
  Let $\mathcal{C}$ denote the subcategory of $\prod_{\mathtt{i} \in Q_0}\mathcal{B}_\mathtt{i}$ consisting of objects satisfying (A1) and (B). Clearly, $X$ restricts to a locally nilpotent endofunctor on $\mathcal{C}$. Since $\mathcal{C}=\prod_{\mathtt{i} \in Q_0}\mathcal{B}_\mathtt{i}$ if $Q$ is finite and acyclic, this generalizes Example \ref{Example: Phyla}. Note that $\mathcal{C}$ is nonzero if and only if $Q$ has a sink vertex.  
	\end{ex}
	
	\begin{ex}\label{Example:Quiver Representations}
		Let $\mathcal{B}$ be a $\mathbbm{k}$-linear additive category. As a special case of the previous examples we can consider the modulation $\mathfrak{B}$ where $\mathcal{B}_\mathtt{i}=\mathcal{B}$ for all $\mathtt{i}\in Q_0$ and $F_\alpha=\operatorname{Id}$ for all $\alpha\in Q_1$. In this case, the category of representation of $\mathfrak{B}$ is just the category $\rep (Q,\mathcal{B})$ of representations of $Q$ in $\mathcal{B}$, whose objects are collections $(B_\mathtt{i},B_{\alpha})_{\mathtt{i}\in Q_0,\alpha\in Q_1}$ of objects $B_\mathtt{i}$ and morphisms $B_{\alpha}\colon B_{s(\alpha)}\to B_{t(\alpha)}$ in $\mathcal{B}$. Note that $\rep (Q,\mathcal{B})$ is equivalent to the category $\mathcal{B}^{\mathbbm{k} Q}$ of $\mathbbm{k}$-linear functors $\mathbbm{k} Q\to \mathcal{B}$, where $Q$ is considered as a category with objects $Q_0$ and with morphisms given by $Q_{\geq 0}$, and where $\mathbbm{k} Q$ denotes the $\mathbbm{k}$-linearization of the category $Q$. 
	\end{ex}

	\subsection{The free monad on an abelian category.}\label{subsection:The free monad of an abelian category}
	
	In this subsection we assume in addition that $\mathcal{C}$ is abelian and $X\colon \mathcal{C}\to \mathcal{C}$ is exact. We then have the following basic results for $T(X)$ and its Eilenberg--Moore category.
	
	\begin{lem}\label{Lemma: BasicPropFreeMonad}
		The following hold:
		\begin{enumerate}
			\item\label{Lemma: BasicPropFreeMonad:1} The functor $T(X)\colon \mathcal{C}\to \mathcal{C}$ is exact
			\item\label{Lemma: BasicPropFreeMonad:2} The category $\mathcal{C}^{T(X)}$ is abelian.
			\item\label{Lemma: BasicPropFreeMonad:3} The functors $f_!\colon \mathcal{C}\to \mathcal{C}^{T(X)}$ and $f^*\colon \mathcal{C}^{T(X)}\to \mathcal{C}$ are exact.
			\item\label{Lemma: BasicPropFreeMonad:4} If $X$ preserves injectives, then $T(X)$ preserves injectives.
		\end{enumerate}
	\end{lem}
	
	\begin{proof}
		Since $X$ is exact and taking coproducts is a right exact functor, $T(X)$ must be right exact. Since $X$ is locally nilpotent, we have an isomorphism $T(X)\cong \prod_{i\geq 0}X^i$. Since taking products is a left exact functor, $T(X)$ must also be left exact. Hence, $T(X)$ is exact. The fact that $\mathcal{C}^{T(X)}$ is abelian and $f^*$ is exact follows from Proposition \ref{Proposition: exactness in Eilenberg--Moore}. Now the exactness of $f_!$ follows from the exactness of $T(X)$ and the description of exact sequences in $\mathcal{C}^{T(X)}$ from Proposition \ref{Proposition: exactness in Eilenberg--Moore}. Finally, if $I\in \mathcal{C}$ is injective and $X$ preserves injectives, then $X^i(I)$ is injective for all $i\geq 0$. Since injective objects are closed under products, $T(X)(I)\cong \prod_{i\geq 0}X^i(I)$ must be injective, so $T(X)$ preserves injective objects. 
	\end{proof}
	
	If $X$ preserves injective objects, then $X$ descends to an endofunctor on the stable injective category $\overline{\mathcal{C}}$. By Lemma \ref{Lemma: BasicPropFreeMonad} \eqref{Lemma: BasicPropFreeMonad:4} the functor $T(X)$ also preserves injective objects, and therefore descends to a monad on $ \overline{\mathcal{C}}$, denoted in the same way.  
	
	\begin{lem}\label{Eilenberg--Moore category of stable free monad}
		Assume $\mathcal{C}$ has enough injectives and $X$ preserves injectives. Then $T(X)$, considered as a monad on $\overline{\mathcal{C}}$, is equal to the free monad of $X$ on $\overline{\mathcal{C}}$. In particular, we have an isomorphism
		$
		\overline{\mathcal{C}}{}^{T(X)}\xrightarrow{\cong} (X\Downarrow \Id_{\overline{\mathcal{C}}}).
		$
	\end{lem}
	
	\begin{proof}
		Since $X$ is locally nilpotent,  the sum $\bigoplus_{i\geq 0}X^i$ is finite when evaluated at any object in $\mathcal{C}$.  It is therefore preserved by the functor $\mathcal{C}\to \overline{\mathcal{C}}$. Hence, $T(X)$ must be the free monad of $X$ on $\overline{\mathcal{C}}$, which proves the claim. 
	\end{proof}
\begin{ex}\label{Example:ModulationStableCats}
		Let $Q$ be a finite and acyclic quiver and $\mathfrak{B}$ a modulation on $Q$ as in Example \ref{Example: Phyla}, and assume the categories $\mathcal{B}_\mathtt{i}$ are abelian. Then $X$ is exact if and only if if the functors $F_\alpha\colon \mathcal{B}_i\to \mathcal{B}_j$ are exact, and $X$ preserves injectives if and only if the functors $F_\alpha$ preserve injectives. Assume these two conditions hold and the categories $\mathcal{B}_\mathtt{i}$ have enough injectives for all vertices $\mathtt{i}$. Then we have another modulation $\overline{\mathfrak{B}}$ on $Q$, given by the injectively stable category $\overline{\mathcal{B}_\mathtt{i}}$ at vertex $\mathtt{i}$, and by the functor $G_\alpha\colon \overline{\mathcal{B}_\mathtt{i}}\to \overline{\mathcal{B}_\mathtt{j}}$ induced from $F_\alpha$ at an arrow $\alpha\colon \mathtt{i}\to \mathtt{j}$. Setting $\mathcal{C}=\prod_{\mathtt{i}\in Q_0}\mathcal{B}_\mathtt{i}$, we get equivalences $\overline{\mathcal{C}}\cong\prod_{\mathtt{i}\in Q_0}\overline{\mathcal{B}_\mathtt{i}}$ and $\operatorname{rep}\overline{\mathfrak{B}}\cong \overline{\mathcal{C}}{}^{T(X)}$. If $\mathcal{B}_\mathtt{i}=\mathcal{B}$ for all $\mathtt{i}\in Q_0$ and $F_\alpha=\operatorname{Id}$ for all $\alpha\in Q_1$ as in Example \ref{Example:Quiver Representations}, then  this gives $\operatorname{rep}(Q,\overline{\mathcal{B}})\cong \overline{\mathcal{C}}{}^{T(X)}$.
	\end{ex}

	\begin{ex}
		Let $Q$ be any quiver, not necessarily finite or acyclic. Assume the categories $\mathcal{B}_\mathtt{i}$ are abelian with enough injectives. Then the abelian category $\prod_{\mathtt{i} \in Q_0}\mathcal{B}_\mathtt{i}$ has enough injectives, and they are given componentwise, i.e. $(B_{\mathtt{i}})_{\mathtt{i}\in Q_0}$ is injective if and only if $B_\mathtt{i}$ is injective for all $\mathtt{i}\in Q_0$. Assume the functors $F_\alpha$ are exact and preserve injectives.  To ensure that powers of the functor $X$ given by \eqref{DefXModulation} 
		are exact and preserve injectives, we need to restrict to objects $(B_{\mathtt{i}})_{\mathtt{i}\in Q_0}$ of $\prod_{\mathtt{i} \in Q_0}\mathcal{B}_\mathtt{i}$ satisfying: 
		\begin{enumerate}
			\item[(A2)]  For each vertex $\mathtt{j}\in Q_0$ one of the following conditions hold:
			\begin{itemize}
				\item For any $n\geq 1$, the cardinality of the set 
   \[
   \{\alpha\in Q_1\mid t(\alpha)=\mathtt{j} \text{ and there exists a path }p\text{ of length }n-1 \text{ with } t(p)=s(\alpha) \text{ and }B_{s(p)}\neq 0 \}
   \]
    is finite, or 	
    \item $\mathcal{B}_{\mathtt{j}}$ is a locally noetherian Grothendieck category. 
			\end{itemize}
		\end{enumerate}
		Indeed, $X^n$ is exact and preserves injectives in the first case since finite direct sums are exact and preserve injectivity. It holds in the second case since infinite coproducts are exact in Grothendieck categories, and since infinite coproducts of injectives are injective in locally noetherian categories. Note that condition (A2) implies condition (A1) in Example \ref{Example:LcaollynilpotentAnyQuiver}. 
		
		Now let $\mathcal{C}$ be the subcategory $\prod_{\mathtt{i} \in Q_0}\mathcal{B}_\mathtt{i}$ consisting of all objects satisfying (B) in Example \ref{Example:LcaollynilpotentAnyQuiver} and (A2). Since $\mathcal{C}$ is closed under subobjects, extensions, and quotients, it is a Serre subcategory of $\prod_{\mathtt{i} \in Q_0}\mathcal{B}_\mathtt{i}$. In particular, it is abelian. It also has enough injectives, since for any $(B_{\mathtt{i}})_{\mathtt{i}\in Q_0}$ in $\mathcal{C}$ we can find a monomorphism $(B_{\mathtt{i}})_{\mathtt{i}\in Q_0}\to (J_{\mathtt{i}})_{\mathtt{i}\in Q_0}$ with $J_\mathtt{i}$ injective for all $\mathtt{i}\in Q_0$ and $J_\mathtt{i}\neq 0$ if and only if $B_\mathtt{i}\neq 0$. Now $X$ restricts to an exact, locally nilpotent endofunctor on $\mathcal{C}$ which preserves injective objects. Hence, it satisfies the standing assumptions in this section. Note that $\mathcal{C}$ is nonzero if and only if $Q$ has a sink vertex, and $\mathcal{C}=\prod_{\mathtt{i} \in Q_0}\mathcal{B}_\mathtt{i}$ if $Q$ is finite and acyclic.  
	\end{ex}
	
	\begin{ex}\label{Example:ModulationofModuleCats}
		Let $\mathfrak{B}$ be a modulation on a quiver $Q$ such that $\mathcal{B}_{s(\alpha)}=\operatorname{Mod}R$ and $\mathcal{B}_{t(\alpha)}=\operatorname{Mod}S$ for two rings $R$ and $S$ and an arrow $\alpha$. We investigate common situations in which the functor $F_\alpha\colon \mathcal{B}_{s(\alpha)}\to \mathcal{B}_{t(\alpha)}$ is exact and preserves injectives, analogous to \cite[Meta-Example 4.2]{GKKP19}.
		
		Assume $F_{\alpha}=\operatorname{Hom}_R(M,-)$ for an $S$-$R$-bimodule $M$. We claim that $F_\alpha$ is exact and preserves injectives if and only if $M$ is projective as a right $R$-module and flat as a left $S$-module. Indeed, exactness of $F_\alpha$ is clearly equivalent to $M$ being projective as a right $R$-module. Now $F_\alpha(I)=\operatorname{Hom}_R(M,I)$ is an injective $S$-module if and only if 
		\[
		\operatorname{Hom}_S(-,\operatorname{Hom}_R(M,I))\cong \operatorname{Hom}_R(-\otimes_S M,I)
		\]
		is an exact functor. This holds for every injective $R$-module if and only if $-\otimes_SM$ is exact, i.e. $M$ is flat as a left $S$-module.
		
		Now assume $F_{\alpha}=-\otimes_RN$ for an $R$-$S$-bimodule $N$, and that $N$ is finitely presented as a left $R$-module. We claim that $F_\alpha$ is exact and preserves injectives if and only if $N$ is projective as a left $R$-module and $\operatorname{Hom}_{R^{\operatorname{op}}}(N,R)$ is flat as a left $S$-module. Indeed, $F_\alpha$ being exact is equivalent to $N$ being flat as an $R$-module, and since $N$ is finitely presented this is again equivalent to $N$ being projective as an $R$-module. If $N$ is finitely presented and projective, then 
		\[
		F_\alpha=-\otimes_R N\cong\operatorname{Hom}_R(\operatorname{Hom}_{R^{\operatorname{op}}}(N,R),-).
		\]
		Hence $F_\alpha$ preserves injectives if and only if $\operatorname{Hom}_{R^{\operatorname{op}}}(N,R)$ is flat as a left $S$-module by the argument above. 
	\end{ex}
	
	\begin{ex}
		As a special case of Example \ref{Example:ModulationofModuleCats}, let $\mathfrak{B}$ be a modulation on $Q=(\mathtt{1}\to \mathtt{2})$ given by a tensor product $-\otimes_R N\colon \operatorname{Mod}R\to \operatorname{Mod}S$ by an $R$-$S$-bimodule $N$. Then the category of $\mathfrak{B}$-representations is equivalent to $\operatorname{Mod}\Lambda$, where $\Lambda$ is the triangular matrix ring
		\[
		\Lambda\coloneqq\begin{pmatrix}R&0\\ N&S
		\end{pmatrix}
		\] 
		Such rings have for example been studied in \cite{FGR75} and \cite[Section III.2]{ARS95}.  Their monomorphism categories (defined below, see Example \ref{Example: Monomorphism category}) occur in \cite{LZ10,XZ12,Zha13} when describing Gorenstein projective $\Lambda$-modules. 
	\end{ex}

	\subsection{The top functor}\label{Subsection: TopFunctor}
	We consider the cokernel functor $\Kopf_X$ and its right adjoint $S$:
	\begin{align*}
	\Kopf_{X}\colon\mathcal{C}^{T(X)}\to \mathcal{C} \quad \module{M}\mapsto \coker h_\module{M}  \\
	S\colon \mathcal{C}\to \mathcal{C}^{T(X)} \quad M\mapsto (X(M)\xrightarrow{0}M). 
	\end{align*}
	Applying $\Kopf_{X}$ is analogous to taking the top of a module, while $S(M)$ can be thought of as the analogue of the semisimple module concentrated at $M$.  Here we collect their most important properties. Lemma \ref{top} \eqref{top:7} can be considered as an analogue of Nakayama's Lemma. 
	
	\begin{lem}\label{top}
		The following hold:
		\begin{enumerate}
			\item\label{top:1} We have an adjunction $\Kopf_X\dashv S$.
			\item\label{top:2} $\Kopf_Xf_!(M)\cong M$ naturally in $M\in \mathcal{C}$.
			\item\label{top:3} $\Kopf_X(\iota_M)=0$ for all $M\in \mathcal{C}$.
			\item\label{top:7} If $\Kopf_X(\module{M})=0$, then $\module{M}=0$.
		\end{enumerate}
	\end{lem}
	
	\begin{proof}
		Part (\ref{top:1}) follows from \cite[Lemma 5.26]{GKKP19}. Part (\ref{top:2}) and (\ref{top:3}) are obvious.  Part (\ref{top:7}) follows since there are no epimorphisms $X(M)\twoheadrightarrow M$ if $M\neq 0$, see Remark \ref{Remark:RelativeNakayama}.
	\end{proof}
	
	We also need a result on the existence of the left derived functors of $\Kopf_X$.
	
	\begin{lem}\label{Lemma:DescriptionLeftDerivedKopf}
		The left derived functors $L_j\Kopf_X$ exist for all $j>0$. Furthermore, they vanish on relative projective objects.
	\end{lem}
	
	\begin{proof}
		By \cite[Proposition 3.1.4]{Kva20b} it suffices to show that the functors $f_!f^*$ and $\Kopf_X\circ f_!f^*$ are exact. But this follows immediately from the fact that $f_!$ and $f^*$ are exact and the isomorphism $\Kopf_X\circ f_!\cong \operatorname{Id}_{\mathcal{C}}$ in Lemma \ref{top} \eqref{top:2}.
	\end{proof}
	
	We finish with the following description of the left derived functors of $\Kopf_X$.
	
	\begin{lem}\label{DescriptionL_1Kopf}
		For $\module{M}\in \mathcal{C}^{T(X)}$ we have  
		\[
		L_1\Kopf_X(\module{M})=\ker h_\module{M} \quad \text{and} \quad L_i\Kopf_X(\module{M})=0 \quad i\geq 2.
		\]
	\end{lem}
	
	\begin{proof}
		This follows by the same proof as in \cite[Lemma 6.24]{GKKP19}.
	\end{proof}

	\subsection{The monomorphism category}\label{Subsection:The monomorphism category} 
	\begin{defn}
		The \emphbf{monomorphism category} $\Mono(X)$ is the full subcategory of $\mathcal{C}^{T(X)}$ consisting of objects $\module{M}$ where the structure map $h_\module{M}$ is a monomorphism.
	\end{defn}
	
	Note that the objects of the monomorphism category can equivalently be characterized using the first left derived functor of $\Kopf_X$.
	
	\begin{lem}\label{Lemma:MonomorphismLeftDerived}
		Let $\module{M}\in \mathcal{C}^{T(X)}$. Then $\module{M}\in \Mono(X)$ if and only if $L_1\Kopf_X(\module{M})=0$
	\end{lem}
	
	\begin{proof}
		This follows immediately from the description of $L_1\Kopf_X$ in Lemma \ref{DescriptionL_1Kopf}.
	\end{proof}
	
	Recall that a subcategory of an abelian category is called $\emphbf{resolving}$ if it is generating and closed under extensions, kernels of epimorphisms, and direct summands.
	
	\begin{lem}\label{Mono:resolving}
		The category $\Mono(X)$ is resolving and closed under subobjects in $\mathcal{C}^{T(X)}$. In particular, it is an exact category whose conflations are epimorphisms and whose inflations are monomorphisms with cokernel in $\Mono(X)$.
	\end{lem}
	
	\begin{proof}
		For any object $\module{M}\in \mathcal{C}^{T(X)}$ the counit $\varepsilon_{\module{M}}\colon f_!f^*(\module{M})\twoheadrightarrow \module{M}$ is an epimorphism. Since any object in the image of $f_!$ is contained in $\Mono(X)$, this shows that $\Mono(X)$ is generating. The fact that $\Mono(X)$ is closed under extensions and subobjects follows from the fact that monomorphisms are closed under extensions and subobjects. Since kernels of epimorphisms and direct summands are special cases of subobjects, this shows that $\Mono(X)$ is resolving. 
		
		Finally, since $\Mono(X)$ is closed under extensions, it inherits an exact structure from $\mathcal{C}^{T(X)}$ where the inflations and the deflations are the monomorphism and the epimorphisms whose cokernel and kernel lies in $\Mono(X)$, respectively. Since $\Mono(X)$ is closed under kernels of epimorphisms, the deflations coincide with the epimorphisms in $\Mono(X)$.
	\end{proof}

	Next we show how $\Kopf_X$ detects inflations and isomorphisms in $\Mono(X)$.
	
	\begin{lem}\label{Lemma:TopDetecingMono}
		Let $g\colon \module{M}\to \module{N}$ be a morphism in $\Mono(X)$. The following hold:
		\begin{enumerate}
			\item\label{Lemma:TopDetecingMono:1} $g$ is an inflation in $\Mono(X)$ if and only if $\Kopf_X(g)$ is a monomorphism.
			\item\label{Lemma:TopDetecingMono:2} $g$ is an isomorphism if and only if $\Kopf_X(g)$ is an isomorphism.
		\end{enumerate}
	\end{lem}
	
	\begin{proof}
		To prove part \eqref{Lemma:TopDetecingMono:1} assume that $g$ is a monomorphism with cokernel in $\Mono(X)$. Applying $\Kopf_X$ to the exact sequence
		\[
		0\to \module{M}\xrightarrow{g} \module{N}\to \operatorname{coker}g\to 0
		\]
		and using that $L_1\Kopf_X(\operatorname{coker}g)=0$, we get that $\Kopf_X(g)$ is a monomorphism.
		
		Conversely, assume $\Kopf_X(g)$ is a monomorphism, and consider the exact sequence
		\[
		0\to \ker g\to \module{M}\to \operatorname{Im}g\to 0.
		\]
		Since $\operatorname{Im}g$ is a subobject of $\module{N}$, it is contained in $\Mono(X)$. Therefore $L_1\Kopf_X(\operatorname{Im}g)=0$, so applying $\Kopf_X$ to the inclusion $\ker g\to \module{M}$ gives a monomorphism $\Kopf_X(\ker g)\to \Kopf_X(\module{M})$. Furthermore, since the composite $\Kopf_X(\ker g)\to \Kopf_X(\module{M})\xrightarrow{\Kopf_X(g)} \Kopf_X(\module{N})$ is $0$ and $\Kopf_X(g)$ is a monomorphism, it follows that $\Kopf_X(\ker g)=0$. Hence $\ker g=0$ by Lemma \ref{top} \eqref{top:7}, so $g$ must be a monomorphism. Finally, applying $\Kopf_X$ to the short exact sequence
		\[
		0\to \module{M}\xrightarrow{g} \module{N}\to \coker g\to 0
		\]
		and using that $\Kopf_X(g)$ is a monomorphism and $L_1\Kopf_X(\module{N})=0$, we get that $L_1\Kopf_X(\coker g)=0$. Hence $\coker g\in \Mono(X)$ by Lemma \ref{Lemma:MonomorphismLeftDerived}. This proves \eqref{Lemma:TopDetecingMono:1}. 
		
		To prove part \eqref{Lemma:TopDetecingMono:2} observe that if $g$ is an isomorphism, then $\Kopf_X(g)$ must be an isomorphism. Conversely, assume that $\Kopf_X(g)$ is an isomorphism. Then $g$ must be a monomorphism by part \eqref{Lemma:TopDetecingMono:1}. Hence we have a right exact sequence
		\[
		\Kopf_X(\module{M})\xrightarrow{\Kopf_X(g)} \Kopf_X(\module{N})\to \Kopf_X(\coker g)\to 0
		\]
		Since $\Kopf_X(g)$ is an isomorphism, we must have that $\Kopf_X(\coker g)=0$. Therefore $\coker g=0$ by Lemma \ref{top} \eqref{top:7}, so $g$ must be an isomorphism.  
	\end{proof}
	
	As a consequence of this we get a criterion for a morphism to be in the radical.
	
	\begin{lem}\label{Lemma:TopReflectRadical}
		Let $g\colon \module{M}\to \module{N}$ be a morphism in $\Mono(X)$, and assume $\Kopf_X(g)$ is in the radical of $\mathcal{C}$. Then $g$ is in the radical of $\Mono(X)$.
	\end{lem}
	
	\begin{proof}
		This follows immediately from the definition of the radical in Subsection \ref{Exact categories} and the fact that $\Kopf_X$ reflects isomorphisms by Lemma \ref{Lemma:TopDetecingMono} \eqref{Lemma:TopDetecingMono:2}.
	\end{proof}

	\begin{ex}\label{Example: Monomorphism category}
		Let $Q$ be a finite and acyclic quiver and $\mathfrak{B}$ a modulation on $Q$ by abelian categories $\mathcal{B}_\mathtt{i}$ and exact functors $F_\alpha$ as in Example \ref{Example: Phyla}. Let $(B_\mathtt{i},B_\alpha)_{\mathtt{i}\in Q_0,\alpha\in Q_1}$ be a $\mathfrak{B}$-representation. For each vertex $\mathtt{i}\in Q_0$ consider the map \[B_{\mathtt{i},\operatorname{in}}\colon \bigoplus_{\substack{\alpha\in Q_1\\t(\alpha)=\mathtt{i}}}F_{\alpha}(B_{s(\alpha)})\xrightarrow{(B_\alpha)} B_\mathtt{i}.\] Then $\Kopf_X$ is given by
		\[
		\Kopf_X(B_\mathtt{i},B_\alpha)_{\mathtt{i}\in Q_0,\alpha\in Q_1}=(\coker B_{\mathtt{i},\operatorname{in}})_{\mathtt{i}\in Q_0}.
		\] 
		Furthermore, by Lemma \ref{DescriptionL_1Kopf} it follows that 
		\[
		L_1\Kopf_X(B_\mathtt{i},B_\alpha)_{\mathtt{i}\in Q_0,\alpha\in Q_1}=(\ker B_{\mathtt{i},\operatorname{in}})_{\mathtt{i}\in Q_0}.
		\] 
		Hence, $(B_\mathtt{i},B_\alpha)_{\mathtt{i}\in Q_0,\alpha\in Q_1}\in \Mono (X)$ if and only if $B_{\mathtt{i},\operatorname{in}}$ is a monomorphism for all $\mathtt{i}\in Q_0$.
	\end{ex}

	\section{Injective objects}\label{Section: InjectiveObjects}
	Recall that $\Mono(X)$ is an exact category by Proposition \ref{Mono:resolving}. In this section we investigate the injective objects in $\Mono(X)$ under the exact structure. We show that up to isomorphism they are the objects of the form $f_!(I)$, where $I$ is injective in $\mathcal{C}$, see Proposition \ref{Proposition:f_!(I)Injective} and Corollary \ref{Corollary:InjectivesinMono}. This improves on Lemma 6.5 in \cite{GKKP19} for the free monad, since that result only implies that injective objects are summands of objects of the form $f_!(I)$. Furthermore the proof in \cite{GKKP19} relies on the existence of a relative Nakayama functor, and this assumption is not necessary in the proofs below. We finish this section by characterizing when $\Mono(X)$ has enough injectives and when it has injective envelopes. Throughout $\mathcal{C}$ is a $\mathbbm{k}$-linear abelian category, and $X\colon \mathcal{C}\to \mathcal{C}$ is an exact functor which is locally nilpotent and preserves injective objects.
	
	We first show that the functor $\Kopf_X$ induces a surjective map on morphisms spaces when the codomain is of the form $f_!(I)$ with $I$ injective.
	
	\begin{lem}\label{Lemma:Kopf_XFullCodomainInj}
		Let $g\colon \Kopf_X \module{M}\to I$ be a morphism in $\mathcal{C}$ with $I$ injective and $\module{M}\in \Mono(X)$. Then there exists a morphism $k\colon \module{M}\to f_!(I)$ in $\mathcal{C}^{T(X)}$ with $\Kopf_X(k)=g$.
	\end{lem}
	
	\begin{proof}
		Let $k_0$ denote the composite $M\xrightarrow{} \operatorname{coker}h_{\module{M}}\xrightarrow{g} I$. Applying $X$, we get a morphism $X(k_0)\colon X(M)\to X(I)$, and since $X(I)$ is injective and $h_\module{M}\colon X(M)\to M$ is a monomorphism, we can find a morphism $k_1\colon M\to X(I)$ satisfying $k_1\circ h_\module{M}=X(k_0)$. Repeating this procedure, we get morphisms $k_i\colon M\to X^i(I)$ satisfying $k_i\circ h_{\module{M}}=X(k_{i-1})$ for each $i\geq 1$. These induce a morphism $M\to \bigoplus_{i\geq 0}X^i(I)$, and since $k_0\circ h_{\module{M}}=0$ this lifts to a morphism $k\colon \module{M}\to f_!(I)$ in $\mathcal{C}^{T(X)}$. By construction we have that $\Kopf_X(k)=g$, so we are done.
	\end{proof}
	
	We can now show that $f_!(I)$ is injective in $\Mono(X)$ when $I$ is injective in $\mathcal{C}$.
	
	\begin{prop}\label{Proposition:f_!(I)Injective}
		Let $I\in \mathcal{C}$ be injective and $\module{M}\in \Mono(X)$. Then 
		\[
		\Ext^i_{\mathcal{C}^{T(X)}}(\module{M},f_!(I))=0 \quad \text{for all $i>0$.}
		\]
	\end{prop}
	
	\begin{proof}
		We prove the result using Yoneda-$\Ext$. Let $\xi\in \Ext^i_{\mathcal{C}^{T(X)}}(\module{M},f_!(I))$. We want to find a representative of $\xi$ whose leftmost map is a split monomorphism. To this end, note that since $\Mono(X)$ is
		resolving, it satisfies the dual of condition (C2) in \cite[Section 12]{Kel96}. Therefore, the dual of \cite[Theorem 12.1]{Kel96} implies that the induced functor
		\[
		D^-(\Mono(X))\to D^-(\mathcal{C}^{T(X)})
		\]
		between the derived categories is fully faithful, where we consider $\Mono(X)$ as an exact category. Hence, we can find a representative of $\xi$ of the form
		\[
		0\to f_!(I)\to \module{N}_1\to \dots \to \module{N}_i\to \module{M}\to 0
		\]
		where all the terms are in $\Mono(X)$. Since $\Mono(X)$ is closed under subobjects, also all the intermediate kernels are in $\Mono(X)$. Let $i\colon f_!(I)\to \module{N}_1$ denote the leftmost morphism. By Lemma \ref{Lemma:TopDetecingMono} the morphism $\Kopf_X (i)\colon I\to \Kopf_X\module{N}_1$ is monic. Since $I$ is injective, $\Kopf_X (i)$ is a split monomorphism, so we can choose a left inverse $g\colon \Kopf_X\module{N}_1\to I$ of it. By Lemma \ref{Lemma:Kopf_XFullCodomainInj} we can find a morphism $k\colon \module{N}_1\to f_!(I)$ satisfying $\Kopf_X(k)=g$. By construction, if we apply $\Kopf_X$ to the composite $k\circ i\colon f_!(I)\to f_!(I)$ we get the identity morphism on $I$. Hence, by Lemma \ref{Lemma:TopDetecingMono} \eqref{Lemma:TopDetecingMono:2} the morphism $k\circ i$ is an isomorphism, so $i$ must be a split monomorphism. This proves the claim.
	\end{proof}
	
	To show the converse of Proposition \ref{Proposition:f_!(I)Injective} we use  the following result.
	
	\begin{prop}\label{Corollary:MonoAndInjectiveComponents}
		Let $\module{M}\in \Mono(X)$, and assume $M$ is injective in $\mathcal{C}$. Then $\module{M}\cong f_!(J)$ for some injective object $J$ in $\mathcal{C}$.
	\end{prop}
	
	\begin{proof}
		Since $\module{M}\in \Mono(X)$, the map $h_\module{M}\colon X(M)\to M$ is a monomorphism. Since $M$ is injective, $X(M)$ must be injective, so $h_\module{M}$ is a split monomorphism. By Proposition \ref{Proposition: direct summands of f_!} the claim follows.
	\end{proof}
	
	\begin{cor}\label{Corollary:InjectivesinMono}
		Let $\module{M}$ be an injective object in $\Mono(X)$ considered as an exact category. Then $\module{M}\cong f_!(J)$ for some injective object $J$ in $\mathcal{C}$.
	\end{cor}
	
	\begin{proof}
		It follows from the assumptions that $\operatorname{Ext}^1_{\mathcal{C}^{T(X)}}(f_!(N),\module{M})=0$ for any $N\in \mathcal{C}$. Now the adjunction $f_!\dashv f^*$ induces an isomorphism
		\[
		\operatorname{Ext}^1_{\mathcal{C}^{T(X)}}(f_!(N),\module{M})\cong \operatorname{Ext}^1_{\mathcal{C}}(N,f^*(\module{M}))
		\]
		see \cite[Lemma 3.2]{LO17}. Therefore $\operatorname{Ext}^1_{\mathcal{C}}(N,f^*(\module{M}))=0$ for all $N\in \mathcal{C}$. This implies that $f^*(\module{M})$ is injective in $\mathcal{C}$. Hence by Proposition \ref{Corollary:MonoAndInjectiveComponents} we get that $\module{M}\cong f_!(J)$ for some injective object $J\in \mathcal{C}$.
	\end{proof}
	
	We use our results to investigate when $\Mono(X)$ has enough injectives.
	
	\begin{prop}
		Assume $\mathcal{C}$ has enough injectives. Then $\Mono(X)$ has enough injectives.
	\end{prop}
	
	\begin{proof}
		Let $\module{M}\in \Mono(X)$ be arbitrary, and choose a monomorphism $i\colon \Kopf_X(\module{M})\rightarrow I$ in $\mathcal{C}$ with $I$ injective. Note that $f_!(I)$ is injective by Proposition \ref{Proposition:f_!(I)Injective}. By Lemma \ref{Lemma:Kopf_XFullCodomainInj} we can find a morphism $j\colon \module{M}\to f_!(I)$ satisfying $\Kopf_X(j)=i$. Now by Lemma \ref{Lemma:TopDetecingMono} \eqref{Lemma:TopDetecingMono:1} we have that $j$ is an inflation in $\Mono(X)$. This proves the claim.
	\end{proof}
	
	We finish by showing the existence of injective envelopes in $\Mono(X)$.
	
	\begin{prop}\label{Theorem:InjectiveEnvelopesMono}
		The following hold:
		\begin{enumerate}
			\item\label{Proposition:InjectiveEnvelopesMono:1} Let $g\colon \module{M}\to f_!(I)$ be a morphism in $\Mono(X)$ with $I\in \mathcal{C}$ injective. Then $g$ is an injective envelope in $\Mono(X)$ if and only if $\Kopf_X(g)\colon \Kopf_X(M)\to I$ is an injective envelope in $\mathcal{C}$.
			\item\label{Proposition:InjectiveEnvelopesMono:2} If $\mathcal{C}$ has injective envelopes, then $\Mono(X)$ has injective envelopes.
		\end{enumerate}
	\end{prop}
	
	\begin{proof}
		To prove part \eqref{Proposition:InjectiveEnvelopesMono:1} we first assume $\Kopf_X(g)$ is an injective envelope. Then $g$ must be a monomorphism with cokernel in $\Mono(X)$ by Lemma \ref{Lemma:TopDetecingMono}. It remains to show that $g$ is left minimal. Assume $k\colon f_!(I)\to f_!(I)$ is a morphism satisfying $k\circ g=g$. Applying $\Kopf_X$ we get that $\Kopf_X(k)\circ \Kopf_X(g)=\Kopf_X(g)$. Since $\Kopf_X(g)$ is an injective envelope, it is left minimal, so $\Kopf_X(k)$ must be an isomorphism. Hence $k$ must be an isomorphism by Lemma \ref{Lemma:TopDetecingMono} \eqref{Lemma:TopDetecingMono:2}, so $g$ is left minimal.
		
		Conversely, assume $g$ is an injective envelope. Since $g$ is an inflation, $\Kopf_X(g)\colon \Kopf_X(\module{M})\to I$ must be a monomorphism. Hence, it only remains to show that $\Kopf_X(g)$ is left minimal, so let $k\colon I\to I$ be a morphism satisfying $k\circ \Kopf_X(g)=\Kopf_X(g)$. Consider  $k'\coloneqq g-f_!(k)\circ g\colon \module{M}\to f_!(I)$. Since $\Kopf_X(k')=0$, there exists a morphism $r\colon \module{M}\to f_!X(I)$ such that the left triangle in
		\[
		\begin{tikzcd}
		f_!(I)&\module{M}\arrow{l}[swap]{k'}\arrow[dashed]{d}{r}\arrow{r}{g}&f_!(I)\arrow[dashed]{ld}{s}\\
		&f_!X(I)\arrow{lu}{\iota}&
		\end{tikzcd}
		\]
		is commutative. Since $f_!X(I)$ is injective in $\Mono(X)$ and $g\colon \module{M}\to f_!(I)$ is a monomorphism, the morphism $r$ extends to a morphism $s\colon f_!(I)\to f_!X(I)$ via $g$. Then clearly 
		\[
		g=(\iota \circ s+f_!(k))\circ g
		\]
		and hence $\iota \circ s+f_!(k)$ is an isomorphism since $g$ is an injective envelope. Finally, since $\Kopf_X(\iota)=0$, it follows that $\Kopf_X(\iota\circ s+f_!(k))=k$, which must therefore also be an isomorphism. This shows that $\Kopf_X(g)$ is left minimal.
		
		To prove part \eqref{Proposition:InjectiveEnvelopesMono:2} let $\module{M}\in \Mono(X)$ be arbitrary, and let $i\colon \Kopf_X\module{M}\to I$ be an injective envelope in $\mathcal{C}$. By Lemma \ref{Lemma:Kopf_XFullCodomainInj} we can find a morphism $j\colon \module{M}\to f_!(I)$ satisfying $\Kopf_X(j)=i$. Then $j$ must be an injective envelope by the first part of this lemma.
	\end{proof}
	
	\section{The epivalence}\label{Section:A representation equivalence}
	Throughout this section we let $\mathcal{C}$ be a $\mathbbm{k}$-linear abelian category with enough injectives and $X\colon \mathcal{C}\to \mathcal{C}$ an exact functor which is locally nilpotent and preserves injectives. Our goal is to show that the canonical functor
	\[
	\overline{\Mono}(X)\to \overline{\mathcal{C}}{}^{T(X)}
	\]
	is an epivalence, and an actual equivalence if $\mathcal{C}$ is hereditary. The proof that the functor is dense is using a particular construction of right $\Mono(X)$-approximations, which we investigate first.
	\subsection{Contravariantly finiteness}\label{Subsection:Contravariantly finiteness}
	
	Recall from Lemma \ref{relative global dimension 1} that we have an exact sequence 
	\[
	0\to f_!X(M)\xrightarrow{\iota_M-f_!(h_\module{M})} f_!(M)\xrightarrow{\varepsilon_{\module{M}}} \module{M}\to 0
	\]
	for all $\module{M}\in \mathcal{C}^{T(X)}$. Applying $\Kopf_{X}$ to this gives an exact sequence  
	\[
	0\to L_1\Kopf_{X} \module{M}\to X(M)\xrightarrow{h_\module{M}} M\to \Kopf_{X} \module{M}\to 0
	\]
	since $L_1\Kopf_{X} f_!(M)=0$. 
	\begin{defn}\label{Definition: Q}
		Fix a monomorphism $j\colon L_1\Kopf_{X}\module{M}\hookrightarrow J$ into an injective module $J$, and fix a lift $e\colon X(M)\to J$ of $j$. Define the object $\Mo\module{M}\in \mathcal{C}^{T(X)}$ and the morphism $p_{\module{M}}\colon \Mo\module{M}\to \module{M}$ to be such that the lower sequence in the following diagram is a pushout of the upper sequence:
		\begin{equation}\label{Defining Mo}
		\begin{tikzcd}[column sep=1.5cm]
		0\arrow{r} &f_!X(M)\arrow{r}{\iota_M-f_!(h_\module{M})}\arrow{d}{f_!(e)}&f_!(M)\arrow{d}{r}\arrow{r}{\varepsilon_{\module{M}}}&\module{M}\arrow[equals]{d}\arrow{r}&0\\
		0\arrow{r} &f_!(J)\arrow{r}{s}&\Mo\module{M}\arrow{r}{p_{\module{M}}}&\module{M}\arrow{r}&0
		\end{tikzcd}
		\end{equation}
	\end{defn}
	
	Note that $\Mo\module{M}$ is not well-defined up to isomorphism, since it depends on the choice of $J$, the monomorphism $j$, and the lift $e$.  In Proposition \ref{Universal property} we show that it is well-defined in a quotient category of $\mathcal{C}^{T(X)}$.
	
	We first show that $\Mo \module{M}$ gives a right $\Mono (X)$-approximation.
	
	\begin{thm}\label{Contravariantly finite}
		The morphism $p_{\module{M}}\colon \Mo\module{M}\to \module{M}$ is a right $\Mono(X)$-approximation for $\module{M}\in \mathcal{C}^{T(X)}$.  
	\end{thm}
	\begin{proof}
		We first show that $\Mo \module{M}$ is contained in $\Mono(X)$. To this end, we apply $\Kopf_{X}$ to the diagram \eqref{Defining Mo}. This yields the diagram
		\[
		\begin{tikzcd}
		L_1\Kopf_{X} \module{M}\arrow{r}\arrow[equals]{d}&X(M)\arrow{d}{e}\arrow{r}{h_\module{M}} &M\arrow{d}\arrow{r} &\Kopf_{X}\module{M}\arrow[equals]{d}\\
		L_1\Kopf_{X}\module{M}\arrow{r}&J\arrow{r}&\Kopf_{X}\Mo\module{M}\arrow{r}&\Kopf_{X}\module{M}
		\end{tikzcd}
		\]
		with exact rows. By commutativity of the leftmost square the map $L_1\Kopf_{X}\module{M}\to J$ is equal to $j$, whence is a monomorphism. Furthermore, the lower row can be extended to an exact sequence
		\[
		0\to L_1\Kopf_{X}\Mo\module{M}\to L_1\Kopf_{X}\module{M}\xrightarrow{j} J 
		\] 
		since $L_1\Kopf_X(f_!(J))=0$. Hence, it follows that $L_1\Kopf_{X}\Mo\module{M}=0$. Thus, $\Mo\module{M}\in \Mono(X)$.	
		
		To see that $p_\module{M}$ is a right approximation, apply $\Hom_{\mathcal{C}^{T(X)}}(\module{N},-)$ with $\module{N}\in \Mono(X)$ to the exact sequence
		\[
		0\to f_!(J)\to \Mo\module{M}\xrightarrow{p_{\module{M}}} \module{M}\to 0.
		\]
		This gives an epimorphism
		\[
		\Hom_{\mathcal{C}^{T(X)}}(\module{N},p_{\module{M}})\colon \Hom_{\mathcal{C}^{T(X)}}(\module{N},\Mo\module{M})\xrightarrow{} \Hom_{\mathcal{C}^{T(X)}}(\module{N},\module{M}).
		\]
		since $\Ext^1_{\mathcal{C}^{T(X)}}(\module{N},f_!(J))=0$ by Proposition \ref{Proposition:f_!(I)Injective}. This proves the claim.
	\end{proof}

	Next we show that $\Mo\module{M}$ satisfies a universal property. Here $\frac{\mathcal{C}^{T(X)}}{f_!(\inj \mathcal{C})}$ denotes the quotient of the category $\mathcal{C}^{T(X)}$ by the ideal of morphisms factoring through an object of the form $f_!(J)$ where $J\in \inj \mathcal{C}$.

	\begin{prop}\label{Universal property}
		Let $g\colon\module{N}\to \module{M}$ be a morphism in $\frac{\mathcal{C}^{T(X)}}{f_!(\inj \mathcal{C})}$ with $\module{N}\in\Mono(X)$. Then there exists a unique morphism 
		\[
		g'\colon \module{N}\to \Mo\module{M}
		\]
		in $\frac{\mathcal{C}^{T(X)}}{f_!(\inj \mathcal{C})}$ such that the equality $p_{\module{M}}\circ g' = g$ holds.
	\end{prop}
	
	\begin{proof}
		The existence of $g'$ follows from $p_{\module{M}}$ being a right $\Mono(X)$-approximation. We show uniqueness: Assume $g'\colon \module{N}\to \Mo\module{M}$ and $g''\colon \module{N}\to \Mo\module{M}$ are two morphisms in $\mathcal{C}^{T(X)}$ such that $p_{\module{M}}\circ g'$ and $p_{\module{M}}\circ g''$ are equal as morphisms in $\frac{\mathcal{C}^{T(X)}}{f_!(\inj \mathcal{C})}$.  Then there exists $J'\in \inj \mathcal{C}$ and morphisms $u\colon \module{N}\to f_!(J')$ and $v\colon f_!(J')\to \module{M}$ such that 
		\[
		v\circ u = p_{\module{M}}\circ (g'-g'')
		\]
		in $\mathcal{C}^{T(X)}$. Furthermore, since $p_{\module{M}}$ is a right $\Mono(X)$-approximation and $f_!(J')\in \Mono(X)$, there exists a morphism $w\colon f_!(J')\to \Mo\module{M}$ satisfying 
		\[
		v=p_{\module{M}}\circ w.
		\]
		Hence, $p_{\module{M}}\circ (g'-g''-w\circ u)=0$ which implies that $g'-g''-w\circ u$ factors through $\ker p_{\module{M}}=f_!(J)$. Therefore, $g'-g''=(g'-g''-w\circ u)+w\circ u$ factors through $f_!(J)\oplus f_!(J')\cong f_!(J\oplus J')$. This shows that $g'$ and $g''$ are equal in $\frac{\mathcal{C}^{T(X)}}{f_!(\inj \mathcal{C})}$. 
	\end{proof}
	
	It follows from Proposition \ref{Universal property} that $\Mo\module{M}$ and $p_{\module{M}}$ are unique up to isomorphism in $\frac{\mathcal{C}^{T(X)}}{f_!(\inj \mathcal{C})}$, independently of the choice of $j\colon L_1\Kopf_{X}\module{M}\to J$ and the lift $e\colon X(M)\to J$. In fact, the universal property of $\Mo\module{M}$ implies that the assignment $\module{M}\mapsto \Mo\module{M}$ can be made into a functor, which the following result shows. Here we write 
	\[
	\overline{\Mono}(X)=\tfrac{\Mono(X)}{f_!(\inj \mathcal{C})}
	\] 
	since by Proposition \ref{Proposition:f_!(I)Injective} and Proposition \ref{Corollary:MonoAndInjectiveComponents} the subcategory of injectives in $\Mono(X)$ is $f_!(\inj \mathcal{C})$.
	
	\begin{cor}\label{Mo is a functor}
		The assignment $\module{M}\mapsto \Mo\module{M}$ induces a functor
		\[
		\Mo\colon \tfrac{\mathcal{C}^{T(X)}}{f_!(\inj \mathcal{C})} \to \overline{\Mono}(X)
		\]
		which is right adjoint to the inclusion functor $i \colon \overline{\Mono}(X)\to \frac{\mathcal{C}^{T(X)}}{f_!(\inj \mathcal{C})}$. Furthermore, the counit of the adjunction $i\dashv \Mo$ at $\module{M}$ is $p_{\module{M}}$.
	\end{cor}
	
	\begin{proof}
		Composing with $p_{\module{M}}$ gives an isomorphism
		\[
		\Hom_{\overline{\Mono}(X)}(\module{N},\Mo\module{M})\xrightarrow{\cong}\Hom_{\frac{\mathcal{C}^{T(X)}}{f_!(\inj \mathcal{C})}}(\module{N},\module{M}) 
		\]
		for $\module{N}\in \Mono(X)$ by Proposition \ref{Universal property}. Since $\Mo\module{M}\in \Mono(X)$, it follows from Yoneda's lemma that the assignment $\module{M}\mapsto \Mo\module{M}$ defines a functor 
		$\Mo\colon \frac{\mathcal{C}^{T(X)}}{f_!(\inj \mathcal{C})} \to \overline{\Mono}(X)$ which makes the isomorphism
		\[
		\Hom_{\frac{\mathcal{C}^{T(X)}}{f_!(\inj \mathcal{C})}}(i(-),\module{M})\cong \Hom_{\overline{\Mono}(X)}(-,\Mo\module{M}) 
		\]
		natural in $\module{M}$. If $\module{N}=\Mo\module{M}$, then the image of the identity $1_{\Mo\module{M}}$ is $p_{\module{M}}$, which is therefore the counit at $\module{M}$. This proves the claim.
	\end{proof}

 \begin{rmk}
     In \cite[Example 5.3]{Che12} they explain how the association $\module{M}\mapsto \Mo\module{M}$ can be interpreted as a triangle functors when $\mathcal{C}^{T(X)}=\operatorname{rep}(\mathtt{1}\to \mathtt{2},\mathcal{B})$ where $\mathcal{B}$ is Frobenius exact.
 \end{rmk}
	\subsection{The general case}\label{Subsection:The general case}
	
	In this subsection, we show that the functor from the monomorphism category to the Eilenberg--Moore category of the stable category induces an epivalence.
	
	\begin{defn}
		A functor is called an \emphbf{epivalence} if it is full, dense and reflects isomorphisms.
	\end{defn}
	
	The remainder of the subsection is concerned with proving the following result.
	
	\begin{thm}\label{Theorem: Canonical functor representation equivalence}
		The canonical functor $\overline{\Mono}(X)\to \overline{\mathcal{C}}{}^{T(X)}$ is an epivalence.
	\end{thm}

	\begin{rmk}
		It was observed in \cite[Chapter II]{Aus71} that epivalences preserve and reflect several important representation-theoretic concepts. In particular, Theorem \ref{Theorem: Canonical functor representation equivalence} has the following consequences:
		\begin{enumerate}
			\item[(i)] An object in $\overline{\Mono}(X)$ is indecomposable if and only if its image in $\overline{\mathcal{C}}{}^{T(X)}$ is indecomposable.
			\item[(ii)] There is a bijection between isomorphism classes of objects in $\overline{\Mono}(X)$ and isomorphism classes of objects in $\overline{\mathcal{C}}{}^{T(X)}$, which restricts to a bijection between the indecomposables.
			\item[(iii)] $\overline{\Mono}(X)$ is Krull--Remak--Schmidt if and only if $\overline{\mathcal{C}}{}^{T(X)}$ is
			Krull--Remak--Schmidt. 
		\end{enumerate}
		Under mild assumptions on $\mathcal{C}$ there is a bijection between the indecomposable non-injective objects in $\Mono(X)$ and the indecomposable objects in $\overline{\Mono}(X)$. Therefore, given an indecomposable object in $\overline{\mathcal{C}}{}^{T(X)}$, we have an associated unique (up to isomorphism) indecomposable non-injective object in $\Mono(X)$. In Section \ref{subsection: A characterization of the indecomposable objects} we show that the Mimo-construction, a refinement of $\Mo$ as defined in the previous subsection,  gives an explicit way to describe it. In particular, Theorem \ref{Theorem: Characterization of indecomposables} reduces the study of indecomposable objects in $\Mono(X)$ to the study of indecomposable objects in $\overline{\mathcal{C}}{}^{T(X)}$, which is often much simpler.
	\end{rmk}

	\begin{lem}\label{Lemma:FunctorDense}
		The canonical functor $\Mono(X)\to \overline{\mathcal{C}}{}^{T(X)}$ is dense. 
	\end{lem}
	
	\begin{proof}
		An object $\module{M}$ in $\overline{\mathcal{C}}{}^{T(X)}$ is given by a morphism $h_\module{M}\colon X(M)\to M$ in $\overline{\mathcal{C}}$ by Lemma \ref{Eilenberg--Moore category of stable free monad}. Choose a lift $h'\colon X(M)\to M$ to $\mathcal{C}$ of $h_{\module{M}}$, and let $\module{M}'$ be the  object in $\mathcal{C}^{T(X)}$ corresponding to it. By Theorem \ref{Contravariantly finite} the object $\Mo \module{M}'$ is in $\Mono(X)$, and by construction it must be isomorphic to $\module{M}$ in $\overline{\mathcal{C}}{}^{T(X)}$. This proves the claim.
	\end{proof}
	Next we show that the functor $\Mono(X)\to \overline{\mathcal{C}}{}^{T(X)}$ is full.
	
	\begin{lem}\label{Lemma:FunctorFull}
		Let $\module{M}\in \Mono(X)$ and $\module{N}\in \mathcal{C}^{T(X)}$. Then any morphism $\module{M}\to \module{N}$ in $\overline{\mathcal{C}}{}^{T(X)}$ can be lifted to a morphism in $\mathcal{C}^{T(X)}$. In particular, the functor $\Mono(X)\to \overline{\mathcal{C}}{}^{T(X)}$ is full.
	\end{lem}
	
	\begin{proof}
		Let $\overline{g}\colon \module{M}\to \module{N}$ be a morphism in $\overline{\mathcal{C}}{}^{T(X)}$, and let $g'\colon M\to N$ be an arbitrary lift of $f^*(\overline{g})$ to $\mathcal{C}$. Since $f^*(\overline{g})\circ h_{\module{M}}$ and $h_{\module{N}}\circ X(f^*(\overline{g}))$ are equal in $\overline{\mathcal{C}}$, the difference 
		\[
		g'\circ h_{\module{M}}-h_{\module{N}}\circ X(g')\colon X(M)\to N
		\]
		is equal to a composite $X(M)\xrightarrow{u'} J\xrightarrow{v'}N$ where $J$ is injective. Furthermore, since $h_\module{M}$ is a monomorphism, the map $u'$ can be lifted to a morphism $u_0\colon M\to J$. Similarly, since $X(J)$ is injective, the map $X(u_0)\colon X(M)\to X(J)$ can be lifted to a morphism $u_1\colon M\to X(J)$. Repeating this argument, we get maps $u_i\colon M\to X^i(J)$ for each $i\geq 0$ such that $u_i\circ h_\module{M}=X(u_{i-1})$ for all $i>0$. Now let $I\coloneqq\bigoplus_{i\geq 0}X^i(J)$, and note that $I$ is injective since $X$ is locally nilpotent and therefore $\bigoplus_{i\geq 0}X^i(J)$ is a finite sum. Let $v\colon I\to N$ be the unique map given on component $X^i(J)$ as the composite 
		\[
		X^i(J)\xrightarrow{X^i(v')} X^i(N)\xrightarrow{X^{i-1}(h_\module{N})}X^{i-1}(N)\xrightarrow{X^{i-2}(h_\module{N})}\dots \xrightarrow{h_\module{N}}N.
		\]
		Furthermore,  let $u\colon M\to I$ be the unique map given on component $X^i(J)$ as $u_i\colon M\to X^i(J)$. If we let $g=g'-v\circ u$, then a short computation shows that $g\circ h_{\module{M}}=h_\module{N}\circ X(g)$, so $g$ is a morphism $\module{M}\to \module{N}$ in $\mathcal{C}^{T(X)}$. Since $g$ is equal to $\overline{g}$ in $\overline{\mathcal{C}}{}^{T(X)}$ and $\overline{g}$ was arbitrary, this proves the claim.
	\end{proof}

	\begin{lem}\label{Lemma:FunctorReflectisos}
		The canonical functor $\overline{\Mono}(X)\to \overline{\mathcal{C}}{}^{T(X)}$ reflects isomorphisms.    
	\end{lem}
	
	\begin{proof}
		Let $g\colon \module{M}\rightarrow\module{N}$ be a morphism in $\mathcal{C}^{T(X)}$ which becomes an isomorphism in $\overline{\mathcal{C}}{}^{T(X)}$. Consider the commutative diagram with exact rows
		\[
		\begin{tikzcd}[column sep=2cm]
		0 \arrow{r} & f_!X(M)\arrow{d}{f_!Xf^*(g)}\arrow{r}{\iota_M-f_!(h_\module{M})}&f_!(M)\arrow{d}{f_!f^*(g)}\arrow{r}{\varepsilon_{\module{M}}}& \module{M} \arrow{d}{g}\arrow{r} & 0 \\
		0 \arrow{r}& f_!X(N)\arrow{r}{\iota_{N}-f_!(h_\module{N})}&f_!(N) \arrow{r}{\varepsilon_{\module{N}}} & \module{N} \arrow{r} & 0.
		\end{tikzcd}
		\]
		Since $f^*(g)\colon M\to N$ is an isomorphism in $\overline{\mathcal{C}}$, the maps $f_!f^*(g)$ and $f_!Xf^*(g)$ are isomorphisms in $\overline{\Mono}(X)$. Now $\module{M}$ is isomorphic to the cone of $\iota_M-f_!(h_\module{M})$ and $\module{N}$ is isomorphic to the cone of $\iota_N-f_!(h_\module{N})$ when we consider $\overline{\Mono}(X)$ as a right triangulated category as in \cite{KV87}, see also \cite{BM94} and \cite{ABM98}. Since the two leftmost vertical maps $f_!Xf^*(g)$ and $f_!f^*(g)$ are isomorphisms in $\overline{\Mono}(X)$, the map $g$ between the cones must also be an isomorphism in $\overline{\Mono}(X)$, see for example the proof of \cite[Corollary 1.5]{ABM98}. This proves the claim.   
	\end{proof}

	\begin{proof}[Proof of Theorem \ref{Theorem: Canonical functor representation equivalence}]
		This follows from Lemma \ref{Lemma:FunctorDense}, Lemma \ref{Lemma:FunctorFull} and Lemma \ref{Lemma:FunctorReflectisos}.
	\end{proof}
	
	\begin{rmk}\label{Remark:DualNumbers}
		Let $\mathbbm{k}$ be a field and $Q$ a finite acyclic quiver. Consider  $\operatorname{rep}(Q,\operatorname{mod}\mathbbm{k}[x]/(x^2))$ as in Example \ref{Example:Quiver Representations}. Then the category $\Mono(X)$ is equivalent to the category of perfect differential $kQ$-modules considered in \cite{RZ17}. Moreover, since $\overline{\operatorname{mod}}\, \mathbbm{k}[x](x^2)\cong \operatorname{mod}\mathbbm{k}$, it follows that
		\[
		\operatorname{rep}(Q,\overline{\operatorname{mod}}\, \mathbbm{k}[x](x^2))\cong \operatorname{rep}(Q,\operatorname{mod}\mathbbm{k}).
		\]
		Hence we have an epivalence $\overline{\Mono(X)}\to \operatorname{rep}(Q,\operatorname{mod}\mathbbm{k})$. The composite 
		\[
		\Mono(X)\to \overline{\Mono(X)}\to \operatorname{rep}(Q,\operatorname{mod}\mathbbm{k})
		\]
		can be identified with the homology functor in \cite[Theorem 1.1 b)]{RZ17}. 
	\end{rmk}
	
	\subsection{The hereditary case}\label{Subsection:The hereditary case}

	Recall that an object $M\in \mathcal{C}$ is called a \emphbf{cosyzygy} (of an object $N$) if there exists an exact sequence $0\to N\to I\to M\to 0$ in $\mathcal{C}$ with $I$ injective. The category $\mathcal{C}$ is called \emphbf{hereditary} if any cosyzygy in $\mathcal{C}$ is injective. Our goal is to show that the functor in Theorem \ref{Theorem: Canonical functor representation equivalence} becomes an actual equivalence when $\mathcal{C}$ is hereditary. We start by proving that relative projective objects are closed under cosyzygies in $\Mono(X)$. 
	
	\begin{prop}\label{Proposition: cosyzygies of f_!}
		Assume we have an exact sequence 
		\[
		0\to f_!(M)\xrightarrow{g} f_!(J)\to \module{N}\to 0
		\]
		in $\mathcal{C}^{T(X)}$ with $\module{N}\in \Mono(X)$ and $J\in \mathcal{C}$ injective. Then $\module{N}\cong f_!(N')$ for some $N'\in \mathcal{C}$.
	\end{prop}
	
	\begin{proof}
		Choose an exact sequence
		\[
		0\to M\xrightarrow{i}I\xrightarrow{q} M'\to 0
		\]
		in $\mathcal{C}$ with $I$ injective. Applying $f_!(-)$, we get an exact sequence
		\[
		0\to f_!(M)\xrightarrow{f_!(i)}f_!(I)\xrightarrow{f_!(q)} f_!(M')\to 0
		\]
		in $\mathcal{C}^{T(X)}$. This implies that both $\module{N}$ and $f_!(M')$ are cosyzygies of $\module{M}$ in $\Mono(X)$. But the cosyzygy of an object is well-defined in $\overline{\Mono}(X)$, see \cite{Hel60}. Hence, there exist injective objects $J_1,J_2\in \mathcal{C}$ and an isomorphism
		\[
		\module{N}\oplus f_!(J_1)\cong f_!(M')\oplus f_!(J_2).
		\]
		In particular, $\module{N}$ is a direct summand of $f_!(M'\oplus J_2)$, and therefore by Proposition \ref{Proposition: direct summands of f_!} we have an isomorphism $\module{N}\cong f_!(N')$ for some $N'\in \mathcal{C}$.
	\end{proof}
	
	Next we show that if a morphism in $\Mono(X)$ factors componentwise through an injective, then it must factor through a relative projective.
	
	\begin{prop}\label{Proposition:FactorThroughRelProj}
		Let $g\colon \module{M}\to \module{N}$ be a morphism in $\Mono(X)$, and assume $f^*(g)\colon M\to N$ factors through an injective object in $\mathcal{C}$. Then $g$ factors through an object of the form $f_!(K)$ where $K$ is a cosyzygy of $X(M)$. 
	\end{prop}
	
	\begin{proof}
		Let $r\colon M\to J$ be a monomorphism with $J$ injective. Since $f^*(g)$ factors through an injective object, it must also factor through $r$, so we can write $f^*(g)=s\circ r$ for some map $s\colon J\to N$. Let $\module{K}'$ be the pushout of the counit $\varepsilon_{\module{M}}\colon f_!(M)\to \module{M}$ along $f_!(r)\colon f_!(M)\to f_!(J)$. We then get the following commutative diagram
		\[
		\begin{tikzcd}
		f_!(M)\arrow{d}{\varepsilon_{\module{M}}}\arrow{r}{f_!(r)}&f_!(J)\arrow{d}\arrow{r}{f_!(s)}& f_!(N) \arrow{d}{\varepsilon_{\module{N}}} \\
		\module{M}\arrow{r}{}&\module{K}' \arrow{r}{} & \module{N},
		\end{tikzcd}
		\]
		where the left hand square is the pushout square and the bottom right morphism is uniquely defined such that the right hand square commutative and the composite $\module{M}\to \module{K}'\to \module{N}$ is equal to $g$. Note that the left hand square is also a pullback square since $f_!(r)$ is a monomorphism. Therefore, since $\varepsilon_{\module{M}}$ is an epimorphism with kernel $f_!X(M)$, the same must hold for $f_!(J)\to \module{K}'$, so we get an exact sequence
		\[
		0\to f_!X(M)\to f_!(J)\to \module{K}'\to 0.
		\]
		By Proposition \ref{Proposition: cosyzygies of f_!} it follows that $\module{K}'\cong f_!(K)$ for some object $K\in \mathcal{C}$. Applying $\Kopf_X$ to the exact sequence, we get the exact sequence
		\[
		0\to X(M)\to J\to K\to 0.
		\]
		This shows that $K$ is a cosyzygy of $X(M)$, which proves the claim.
	\end{proof}
	
	We can now prove the main result of this subsection.
	
	\begin{thm}\label{Theorem:EquivalenceHereditary}
		The following are equivalent:
		\begin{enumerate}
			\item\label{Theorem:EquivalenceHereditary:1} Any object in the image of $X$ has injective dimension at most $1$.
			\item\label{Theorem:EquivalenceHereditary:2} $\overline{\Mono}(X)\to \overline{\mathcal{C}}{}^{T(X)}$ is an equivalence.
		\end{enumerate}
		In particular, this holds if $\mathcal{C}$ is hereditary.
	\end{thm}
	
	\begin{proof}
		We already know that the functor in \eqref{Theorem:EquivalenceHereditary:2} is full and dense, so the statement is equivalent to the functor being faithful. By Proposition \ref{Proposition:FactorThroughRelProj} any morphism $g$ in $\Mono(X)$ which is $0$ in $\overline{\mathcal{C}}{}^{T(X)}$ must factor through an object of the form $f_!(K)$ where $K$ is a cosyzygy in $\mathcal{C}$ of an object of the form $X(M)$. If \eqref{Theorem:EquivalenceHereditary:1} holds, then $K$ must be injective, so $g$ must be $0$ in $\overline{\Mono}(X)$. 
		
		Next assume \eqref{Theorem:EquivalenceHereditary:2} holds, and let $M\in \mathcal{C}$ be arbitrary. Choose an exact sequence
		\[
		0\to X(M)\to J\to K\to 0
		\]
		with $J$ injective. Our goal is to show that $K$ is injective. Applying $f_!$ and taking the pushout along $\iota_M$ gives a commutative diagram
		\[
		\begin{tikzcd}
		0\arrow{r}&f_!(X(M))\arrow{r}\arrow{d}{\iota_M}&f_!(J)\arrow{d}{g}\arrow{r}{}&f_!(K)\arrow[equals]{d}\arrow{r}&0\\
		0\arrow{r}&f_!(M)\arrow{r}&\module{E}\arrow{r}{p}&f_!(K)\arrow{r}&0
		\end{tikzcd}
		\]
		with exact rows. Note that $\module{E}\in \Mono(X)$ since it is an extension of two objects in $\Mono(X)$. Since $f^*(\iota_M)$ is a split monomorphism, $f^*(g)$ must be a split monomorphism and $f^*(p)$ must factor through $f^*f_!(J)$. Hence, $p$ is $0$ in $\overline{\mathcal{C}}{}^{T(X)}$. It follows from  assumption \eqref{Theorem:EquivalenceHereditary:2} that $p$ factors through an object of the form $f_!(I)$ with $I$ injective in $\mathcal{C}$. Hence $\Kopf_X(p)\colon \Kopf_X\module{E}\to K$ factors through $I$. We claim that $\Kopf_X(p)$ is also a split epimorphism. Indeed, applying $\Kopf_X$ to the diagram above gives a commutative diagram with exact rows
		\[
		\begin{tikzcd}
		0\arrow{r}&X(M)\arrow{r}\arrow{d}{\Kopf_X(\iota_M)}&J\arrow{d}{\Kopf_X(g)}\arrow{r}{}&K\arrow[equals]{d}\arrow{r}&0\\
		0\arrow{r}&M\arrow{r}&\Kopf_X\module{E}\arrow{r}{\Kopf_X(p)}&K\arrow{r}&0.
		\end{tikzcd}
		\]
		Since $\Kopf_X(\iota_M)=0$ and the left hand square is a pushout, the map $M\to \Kopf_X\module{E}$ is a split monomorphism. Therefore $\Kopf_X(p)$ must be a split epimorphism. Since $\Kopf_X(p)$ factors through $I$, it follows that the induced map $I\to K$ is also a split epimorphism. Hence $K$ is a summand of an injective object, and must therefore be injective. 
	\end{proof}
	
	\begin{rmk}
		Consider the category $\operatorname{rep}(Q,\mathcal{B})$ of representations of a finite acyclic quiver $Q$ in $\mathcal{B}$ as in Example \ref{Example:Quiver Representations}, and assume $\mathcal{B}$ has enough injectives. Let $\Mono(Q,\mathcal{B})$ be the monomorphism category, and consider the functor 
		\[
		\overline{\Mono}(Q,\mathcal{B})\to\operatorname{rep}(Q,\overline{\mathcal{B}})
		\]
		in Theorem \ref{Theorem:EquivalenceHereditary}.  If $Q$ has at least one arrow, then the image of $X$ contains at least one copy of $\mathcal{B}$. Then it follows from Theorem \ref{Theorem:EquivalenceHereditary} that the functor above is an equivalence if and only if $\mathcal{B}$ is hereditary. In contrast, if $Q$ has no arrows, than $X$ is the zero functor, so \eqref{Theorem:EquivalenceHereditary:1} is automatically satisfied and in \eqref{Theorem:EquivalenceHereditary:2}, both categories coincide with $\prod_{\mathtt{i}\in Q_0} \overline{\mathcal{C}}$.
	\end{rmk}
	
	We are particularly interested in situations where $\mathcal{C}$ is hereditary and $\overline{\mathcal{C}}$ is abelian. This implies that $\overline{\mathcal{C}}{}^{T(X)}$ is abelian, and hence $\overline{\Mono}(X)$ is abelian by Theorem \ref{Theorem:EquivalenceHereditary}. It follows that the indecomposable non-injective objects in $\Mono(X)$ can be obtained by studying the indecomposable objects in an abelian category.
	
	\begin{rmk}\label{Remark:HereditaryStableAbelian}
		Assume $\mathcal{C}$ is the category of left $\Lambda$-modules for an Artin algebra $\Lambda$. Then $\mathcal{C}$ is hereditary if and only if $\Lambda$ is left hereditary. It was shown in \cite[Theorem 9.5]{MZ15} that for such rings the stable category of $\mathcal{C}$ modulo projectives, denoted $\underline{\mathcal{C}}$, is abelian if and only if the injective envelope of $\Lambda$ is projective. Since $\overline{\mathcal{C}}\cong \underline{\mathcal{C}}$ by \cite[Proposition 5.5]{Kra97}, this is also equivalent to $\overline{\mathcal{C}}$ being abelian. Now the injective envelope of $\Lambda$ being projective holds precisely if $\Lambda$ is isomorphic to a finite direct products of complete blocked triangular matrix algebras over division rings, see \cite[Remark 7.6]{MZ15}. 
	\end{rmk}
	
	\begin{rmk}
		Consider the category $\operatorname{rep}(Q,\mathcal{B})$ of representations of a finite acyclic quiver $Q$ in $\mathcal{B}$ as in Example \ref{Example:Quiver Representations}. Assume furthermore that $\mathcal{B}=\operatorname{Mod}\mathbbm{k}\mathbb{A}_n$ where $\mathbbm{k}\mathbb{A}_n$ is the path algebra of a linearly oriented $\mathbb{A}_n$-quiver over a field $\mathbbm{k}$. Since $\overline{\operatorname{Mod}}\,\mathbbm{k}\mathbb{A}_n\cong \operatorname{Mod}\mathbbm{k}\mathbb{A}_{n-1}$ we get that
		\[
		\mathcal{C}=\prod_{\mathtt{i}\in Q_0}\operatorname{Mod}\mathbbm{k}\mathbb{A}_n \quad \text{and} \quad \overline{\mathcal{C}}\cong \prod_{\mathtt{i}\in Q_0}\operatorname{Mod}\mathbbm{k}\mathbb{A}_{n-1}.
		\]
		Hence by Theorem \ref{Theorem:EquivalenceHereditary}  we get an equivalence
		\[
		\overline{\Mono}(Q,\operatorname{Mod}\mathbbm{k}\mathbb{A}_{n})\cong \rep (Q,\operatorname{Mod}\mathbbm{k}\mathbb{A}_{n-1}).
		\]
		By letting $Q$ be the linearly oriented $\mathbb{A}_m$-quiver, we recover (the dual of) \cite[Theorem 1.5]{BBOS20}. Note that 
		\[
		\prod_{\mathtt{i}\in Q_0}\operatorname{Mod}\mathbbm{k}\mathbb{A}_n\cong \operatorname{Mod}\Lambda
		\]
		where $\Lambda=\prod_{\mathtt{i}\in Q_0}\mathbbm{k}\mathbb{A}_n$ satisfies the condition in Remark \ref{Remark:HereditaryStableAbelian}. 
	\end{rmk}
	
	\begin{rmk}\label{Remark:DedekindDomain}
		There are other examples of hereditary categories $\mathcal{C}$ such that $\overline{\mathcal{C}}$ is abelian. Indeed, let $\mathcal{C}=\operatorname{art}\Lambda$ be the category of artinian modules over a Dedekind domain $\Lambda$. This is abelian, closed under injective envelopes \cite[Theorem 2]{Vam68}, and hereditary. We claim that any artinian module must be a finite sum of indecomposable injective modules and modules of finite length. Indeed, since any indecomposable finite length module over $\Lambda$ is of the form $\Lambda/\mathfrak{m}^n$ where $\mathfrak{m}$ is a maximal ideal, the finite length modules form a uniserial category in the sense of \cite{Kra22}. By \cite[Proposition 2.4.20]{Kra22} any artinian module is a filtered colimit of finite length modules, so the claim follows from \cite[Theorem 13.1.28]{Kra22}.
		
		Consider the stable category $\overline{\mathcal{C}}$. Since any injective module over a Dedekind domain has infinite length \cite[Corollary 2]{Har69}, the objects in $\overline{\mathcal{C}}$ are up to isomorphism precisely the modules of finite length. Furthermore, since $\Lambda$ is hereditary there are no nonzero morphisms from injective modules to modules of finite length, so $\overline{\mathcal{C}}$ must be equivalent to the category $\operatorname{fl}\Lambda$ of finite length modules over $\Lambda$, which is abelian. 
	\end{rmk}
	
	\section{The Mimo-construction}\label{The Mimo construction}
	We fix a $\mathbbm{k}$-linear abelian category $\mathcal{C}$ with enough injectives, and an exact functor $X\colon \mathcal{C}\to \mathcal{C}$ which is locally nilpotent and preserves injectives. The goal of this section is to extend the Ringel--Schmidmeier's $\Mimo$-construction \cite{RS08} to our setting. In particular, we show that it is a minimal right $\Mono(X)$--approximation.

	\subsection{Definition and properties}\label{Subsection:Definition and properties}
	
	We start with the definition of the Mimo-construction. 
	
	\begin{defn}\label{Definition: Mimo}
		Let $\module{M}$ be an object in $\mathcal{C}^{T(X)}$, and consider the construction $\Mo \module{M}$ in Definition \ref{Definition: Q}. If $J$ is an injective envelope of $L_1\Kopf_{X}\module{M}$, then we write $\Mimo \module{M}\coloneqq \Mo \module{M}$ and call this the \emphbf{Mimo-construction} of $\module{M}$.
	\end{defn}
	
	Note that $\Mimo\module{M}$ might not exist for all $\module{M}$, unless $\mathcal{C}$ has injective envelopes. We show that if $\Mimo\module{M}$ exists, then it is a minimal right approximation and therefore unique up to isomorphism. 
	
	\begin{thm}\label{Mimo=Minimal}
		Let $\module{M}\in \mathcal{C}^{T(X)}$, and assume $\Mimo \module{M}$ exists. Then the canonical morphism $p_{\module{M}}\colon \Mimo\module{M}\to \module{M}$ is a minimal right $\Mono(X)$-approximation.
	\end{thm}
	
	\begin{proof}
		By Theorem \ref{Contravariantly finite} we only need to show $p_\module{M}$ is right minimal. Let $\varphi\colon \Mimo\module{M}\to \Mimo\module{M}$ be a morphism satisfying $p_{\module{M}}\circ \varphi = p_{\module{M}}$. Consider the following commutative diagram with exact rows
		\[
		\begin{tikzcd}
		0\arrow{r}&f_!(J)\arrow{r}\arrow[dashed]{d}{\psi}&\Mimo\module{M}\arrow{d}{\varphi}\arrow{r}{p_{\module{M}}}&\module{M}\arrow{r}\arrow[equals]{d}&0\\
		0\arrow{r}&f_!(J)\arrow{r}&\Mimo\module{M}\arrow{r}{p_{\module{M}}}&\module{M}\arrow{r}&0
		\end{tikzcd}
		\]
		where $\psi$ is induced from the commutativity of the right hand square. Applying $\Kopf_{X}$ yields the commutative diagram:
		\[
		\begin{tikzcd}
		0\arrow{r} &L_1\Kopf_{X}\module{M}\arrow[equals]{d}\arrow{r}{j}&J\arrow{d}{\Kopf_{X}(\psi)}\arrow{r}&\Kopf_{X}(\Mimo\module{M})\arrow{d}{\Kopf_{X}(\varphi)}\arrow{r}&\Kopf_{X}\module{M}\arrow{r}\arrow[equals]{d}&0\\
		0\arrow{r} &L_1\Kopf_{X}\module{M}\arrow{r}{j}&J\arrow{r}&\Kopf_{X}(\Mimo\module{M})\arrow{r}&\Kopf_{X}\module{M}\arrow{r}&0 .
		\end{tikzcd}
		\]
		As $j$ is a minimal left approximation, it follows that $\Kopf_{X}(\psi)$ is an isomorphism. Therefore $\psi$ is an isomorphism by Lemma \ref{Lemma:TopDetecingMono} part \eqref{Lemma:TopDetecingMono:2}. The $5$-Lemma then implies that $\varphi$ is an isomorphism.
	\end{proof}
	
	Next we show that the isomorphism class of $\Mimo \module{M}$ only depends on the isomorphism class of $\module{M}$ in  $\overline{\mathcal{C}}{}^{T(X)}$. This plays an important role in Section \ref{subsection: A characterization of the indecomposable objects}. First we prove an analogue for $\Mo \module{M}$.
	
	\begin{lem}\label{Independent of a morphism factoring through an injective}
		Let $\module{M},\module{N}\in \mathcal{C}^{T(X)}$. If $\module{M}\cong \module{N}$ in $\overline{\mathcal{C}}{}^{T(X)}$, then $\Mo\module{M}\cong \Mo\module{N}$ in $\overline{\Mono}(X)$.
	\end{lem}

	\begin{proof}
		Consider the composite of the isomorphisms $\Mo \module{M}\xrightarrow{\cong} \module{M}\xrightarrow{\cong} \module{N}\xrightarrow{\cong} \Mo\module{N}$ in $\overline{\mathcal{C}}{}^{T(X)}$. By Lemma \ref{Lemma:FunctorFull} the functor $\overline{\Mono}(X)\to \overline{\mathcal{C}}{}^{T(X)}$ is full, and hence this composite can be lifted to a morphism $\Mo \module{M}\to \Mo\module{N}$ in $\overline{\Mono}(X)$. By Lemma \ref{Lemma:FunctorReflectisos} the functor $\overline{\Mono}(X)\to \overline{\mathcal{C}}{}^{T(X)}$ reflects isomorphisms, so the lift has to be an isomorphism in $\overline{\Mono}(X)$. This proves the claim.  
	\end{proof}
	\begin{rmk}
		We are not claiming that $\Mo$ is a functor on $\overline{\mathcal{C}}{}^{T(X)}$ in Lemma \ref{Independent of a morphism factoring through an injective}. As one can see from the proof, the choice of the isomorphism $\Mo\module{M}\cong \Mo\module{N}$ is not unique.
	\end{rmk}
	
	To obtain a similar result for $\Mimo$, we first need to prove a result on the non-existence of nonzero injective summands.
	
	\begin{lem}\label{No injective summands}
		Assume $\mathcal{C}$ has injective envelopes. Let $\module{M}\in \mathcal{C}^{T(X)}$, and assume $M$ has no nonzero injective summands. Let $g\colon f_!(I)\to \Mimo\module{M}$ be a morphism in $\Mono(X)$. If $I$ is injective, then $g$ must be in the radical of $\Mono(X)$. In particular, $\Mimo\module{M}$ has no nonzero summands in $\operatorname{add} f_!(\inj \mathcal{C})$.
	\end{lem} 
	
	\begin{proof}
		By definition, $g$ is in the radical if for any morphism $g'\colon \Mimo \module{M}\to f_!(I)$ the difference $1_{f_!(I)}-g'\circ g$ is an isomorphism. For this, it suffices to show that $g'\circ g$ is in the radical, since then by definition $1_{f_!(I)}-1_{f_!(I)}\circ (g'\circ g)=1_{f_!(I)}-g'\circ g$ is an isomorphism. Let $r\colon f_!(M)\to \Mimo\module{M}$ and $s\colon f_!(J)\to \Mimo\module{M}$ be as in diagram \eqref{Defining Mo} (where $\Mo\module{M}$ is replaced by $\Mimo\module{M}$). Consider the diagram 
		\[
		\begin{tikzcd}
		f_!(I)\arrow{r}{g}\arrow[dashed]{rd}{k}&\Mimo\module{M}\arrow{r}{g'}&f_!(I)\\
		&f_!(M)\oplus f_!(J)\arrow{u}{(r,s)}\arrow{ru}[swap]{(g'\circ r,g'\circ s)}&
		\end{tikzcd}
		\]
		We want to show the existence a morphism $k$ making the left triangle commutative. To this end, note that the triangle identities for the adjunction $f_!\dashv f^*$ imply that $f^*(\varepsilon_\module{M})$ is a split epimorphism. Thus, the sequence
		\[
		0\to f^*f_!X(M)\xrightarrow{f^*(\iota_M)-f^*f_!(h_\module{M})} f^*f_!(M)\xrightarrow{f^*(\varepsilon_{\module{M}})} M\to 0
		\]
		is split exact. Hence, $f^*(\iota_M)-f^*f_!(h_\module{M})$ is a split monomorphism. Therefore the short exact sequence (obtained by the definition of $\Mimo\module{M}$ as a pushout)
		\begin{equation*}
		\resizebox{\textwidth}{!}{$ 0\to f^*f_!X(M)\xrightarrow{\begin{pmatrix} f^*f_!(e) \\ f^*(\iota_M)-f^*f_!(h_\module{M}) \end{pmatrix}} f^*f_!(J)\oplus f^*f_!(M)\xrightarrow{\begin{pmatrix} f^*(s) & f^*(r) \end{pmatrix}} f^*(\Mimo\module{M})\to 0$ }
		\end{equation*}
		is split exact. It follows that $(f^*(s), f^*(r))$ is a split epimorphism. By Lemma \ref{Lemma:RelativeProjective} this implies the existence of the morphism $k$ above. 
		
		Next we claim that the morphism $(\Kopf_X(g'\circ r),\Kopf_X(g'\circ s))$ is in the radical. Indeed, the morphism $\Kopf_X(g'\circ r)\colon M\to I$ is in the radical by Lemma \ref{Lemma:RadicalProperties} \eqref{Lemma:RadicalProperties:2} since $M$ has no nonzero injective summands. Also, the composite 
		\[
		L_1\Kopf_{X}\module{M}\xrightarrow{j}J\xrightarrow{\Kopf_{X}(g'\circ s)}I   
		\]
		must be $0$, by construction of $\Mimo$. Hence $\Kopf_{X}(g'\circ s)$ factors through the cokernel of $j$, which is a radical morphism by Lemma \ref{Lemma:RadicalProperties} \eqref{Lemma:RadicalProperties:1}, since $j$ is an injective envelope. Hence $\Kopf_{X}(g'\circ s)$ must itself be a radical morphism. This shows that $(\Kopf_X(g'\circ r),\Kopf_X(g'\circ s))$ is in the radical. By Lemma \ref{Lemma:TopReflectRadical} the morphism  $(g'\circ r,g'\circ s)$ is in the radical of $\Mono(X)$. Since $g'\circ g$ factors through $(g'\circ r,g'\circ s)$, this proves the claim.
	\end{proof}
	We can now prove a similar result to Lemma \ref{Independent of a morphism factoring through an injective} for the Mimo-construction. This is an analogue of \cite[Proposition 4.1]{RS08}.
	
	\begin{cor}\label{Mimo well-defined up to factoring through injective}
		Assume $\mathcal{C}$ has injective envelopes. Let $\module{M}$ and $\module{N}$ be objects in $\mathcal{C}^{T(X)}$. Assume that $M$ and $N$ have no nonzero injective summands, and that $\module{M}\cong \module{N}$
		in $\overline{\mathcal{C}}{}^{T(X)}$. Then $\Mimo\module{M}\cong \Mimo\module{N}$ in $\Mono(X)$.
	\end{cor}
	
	\begin{proof}
		By Lemma \ref{Independent of a morphism factoring through an injective} it follows that $\Mimo\module{M}\cong \Mimo\module{N}$ in $\overline{\Mono}(X)$. Hence, there exist injective objects $I_1,I_2\in \mathcal{C}$ such that $ \Mimo\module{M}\oplus f_!(I_1)$ and $\Mimo\module{N}\oplus f_!(I_2)$ are isomorphic. The claim now follows from Lemma \ref{No injective summands} and Lemma \ref{Lemma:UniquenessMaxInjSummand}.
	\end{proof}

	\subsection{Mimo for modulations}\label{Subsection:Mimo for modulations}
	
	To describe $\Mimo \module{M}$ for a modulation, we first give a description of it as an object in $(X\Downarrow \operatorname{Id}_{\mathcal{C}})$ and of $r,s$ and $p_{\module{M}}$ in \eqref{Defining Mo} as morphisms in $\mathcal{C}$. Recall that $f^*f_!(J)=\bigoplus_{i\geq 0}X^i(J)$ and $f^*f_!(M)=\bigoplus_{i\geq 0}X^i(M)$. 
	
	\begin{lem}\label{Lemma:Mimo explicit}
		Fix the notation as in Definition \ref{Definition: Q}. Then there exists an isomorphism $$f^*(\Mo \module{M})\cong M\oplus \bigoplus_{i\geq 0}X^i(J)$$ such that:
		\begin{enumerate}
			\item\label{Lemma:Mimo explicit:1} $f^*(s)=\begin{pmatrix}0\\1\end{pmatrix}\colon \bigoplus_{i\geq 0}X^i(J)\to M\oplus \bigoplus_{i\geq 0}X^i(J)$.
			\item\label{Lemma:Mimo explicit:2} $f^*(r)\colon  \bigoplus_{i\geq 0}X^i(M)\to M\oplus \bigoplus_{i\geq 0}X^i(J)$ is induced by the maps 
			\begin{itemize}
				\item $X^k(M)\xrightarrow{h_\module{M}\circ \dots \circ X^{k-1}(h_{\module{M}})}M$ for all $k\geq 0$. (where $k=0$ gives the identity map).
				\item $X^{j+k+1}(M)\xrightarrow{X^{j}(e)\circ X^{j+1}(h_{\module{M}})\circ \dots \circ X^{j+k}(h_{\module{M}})} X^j(J)$ for all $j,k\geq 0$.
			\end{itemize}
			\item\label{Lemma:Mimo explicit:3} $f^*(p_{\module{M}})=\begin{pmatrix}1&0\end{pmatrix}\colon  M\oplus \bigoplus_{i\geq 0}X^i(J)\to M$.
			\item\label{Lemma:Mimo explicit:4}  $h_{\Mo\module{M}}\colon X(M)\oplus \bigoplus_{i\geq 0}X^{i+1}(J)\to M\oplus \bigoplus_{i\geq 0}X^i(J)$ is induced by the maps 
			\begin{itemize}
				\item $X^{i+1}(J)\xrightarrow{1} X^{i+1}(J)$ for all $i\geq 0$.
				\item $X(M)\xrightarrow{h_{\module{M}}}M$.
				\item $X(M)\xrightarrow{e}J$.
			\end{itemize}  
		\end{enumerate}
	\end{lem}
	
	\begin{proof}
		Let $\module{M}'\in (X\Downarrow \operatorname{Id}_{\mathcal{C}})$ be the object defined by $f^*(\module{M}')=M\oplus \bigoplus_{i\geq 0}X^i(J)$ and by a map $h_{\module{M}'}\colon Xf^*(\module{M}')\to f^*(\module{M'})$ as in \eqref{Lemma:Mimo explicit:4}. Recall that $h_{f_!(M)}\colon \bigoplus_{i\geq 0}X^{i+1}(M)\to \bigoplus_{i\geq 0}X^{i}(M)$ and $h_{f_!(J)}\colon \bigoplus_{i\geq 0}X^{i+1}(J)\to \bigoplus_{i\geq 0}X^{i}(J)$ are induced by the maps $X^{i+1}(M)\xrightarrow{1}X^{i+1}(M)$ and $X^{i+1}(J)\xrightarrow{1}X^{i+1}(J)$ for $i\geq 0$, respectively. A straightforward computation then gives that 
		\begin{align*}
		h_{\module{M}'}\circ X(r')=r'\circ h_{f_!(M)} \quad \text{and} \quad  h_{\module{M}'}\circ X(s')=s'\circ h_{f_!(J)}
		\end{align*}
		where $r'$ and $s'$ are defined using the formulas in \eqref{Lemma:Mimo explicit:2} and \eqref{Lemma:Mimo explicit:1}, respectively. Hence there exist well-defined  morphisms $r\colon f_!(M)\to \module{M}'$ and $s\colon f_!(J)\to \module{M}'$ in $\mathcal{C}^{T(X)}\cong (X\Downarrow \operatorname{Id}_{\mathcal{C}})$ such that $f^*(r)=r'$ and $f^*(s)=s'$. Now observe that the sequence 
		\[
		0\to f_!X(M)\xrightarrow{\begin{pmatrix}\iota_M-f_!(h_{\module{M}})\\f_!(e)\end{pmatrix}}f_!(M)\oplus f_!(J)\xrightarrow{\begin{pmatrix}r&-s\end{pmatrix}} \module{M}'\to 0
		\]
		is exact, since it is exact when applying $f^*$. This implies that the left square in
		\begin{equation*}
		\begin{tikzcd}[column sep=2cm]
		0\arrow{r} &f_!X(M)\arrow{r}{\iota_M-f_!(h_\module{M})}\arrow{d}{f_!(e)}&f_!(M)\arrow{d}{r}\arrow{r}{\varepsilon_{\module{M}}}&\module{M}\arrow[equals]{d}\arrow{r}&0\\
		0\arrow{r} &f_!(J)\arrow{r}{s}&\module{M'}\arrow[dashed]{r}{p_{\module{M}}}&\module{M}\arrow{r}&0
		\end{tikzcd}
		\end{equation*}
		is a pushout square. But since $\Mo \module{M}$ is also defined by this pushout square, we get an isomorphism $\module{M'}\cong \Mo \module{M}$. Finally, since $f^*(p_{\module{M}})$ as defined in \eqref{Lemma:Mimo explicit:3} is the cokernel of $f^*(s)$, it must lift to a map $p_{\module{M}}\colon \module{M'}\to \module{M}$ such that the lower sequences in the diagram is exact. Since $f^*(p_{\module{M}})\circ f^*(r)=f^*(\varepsilon_{\module{M}})$, the right square must be commutative. This proves the claim.
	\end{proof}
	
	\begin{ex}\label{Example: Mimo construction}
		Let $Q$ be a finite and acyclic quiver and $\mathfrak{B}$ a modulation by abelian categories with enough injectives and exact functors preserving injective objects, see Example \ref{Example: Phyla}. Let $(B_\mathtt{i},B_\alpha)_{\mathtt{i}\in Q_0,\alpha\in Q_1}$ be a $\mathfrak{B}$-representation. We want to compute $\Mo (B_\mathtt{i},B_\alpha)$ as in Definition \ref{Definition: Q}. To this end, for each $\mathtt{k}\in Q_0$ choose an injective object $J_\mathtt{k}$ in $\mathcal{B}_{\mathtt{k}}$ and a map $$e_{\mathtt{k}}\colon \bigoplus_{\substack{\alpha\in Q_1, t(\alpha)=\mathtt{k}}} F_{\alpha}(B_{s(\alpha)})\to J_{\mathtt{k}}$$ whose restriction to $L_1\Kopf_X(B_\mathtt{i},B_\alpha)_{\mathtt{k}}=\ker B_{\mathtt{k},\operatorname{in}}$ is a monomorphism, see Example \ref{Example: Monomorphism category}. From Lemma \ref{Lemma:Mimo explicit} it follows that
		\[
		\Mo (B_\mathtt{i},B_\alpha)_{\mathtt{k}}= B_{\mathtt{k}}\oplus \bigoplus_{p\in Q_{\geq 0},t(p)=\mathtt{k}}F_p(J_{s(p)})
		\] 
		By Lemma \ref{Lemma:Mimo explicit} \eqref{Lemma:Mimo explicit:4} the morphism $\Mo (B_\mathtt{i},B_\alpha)_{\beta}\colon \Mo (B_\mathtt{i},B_\alpha)_{\mathtt{j}}\to \Mo (B_\mathtt{i},B_\alpha)_{\mathtt{k}}$ associated to an arrow $\beta\colon \mathtt{j}\to \mathtt{k}$ is the  map
		\[
		F_{\beta}(B_{\mathtt{j}})\oplus \bigoplus_{p\in Q_{\geq 0},t(p)=\mathtt{j}}F_\beta F_p(J_{s(p)})\to B_{\mathtt{k}}\oplus \bigoplus_{q\in Q_{\geq 0},t(q)=\mathtt{k}}F_q(J_{s(q)})
		\]
		which is induced by the identity $F_\beta F_p(J_{s(p)})\xrightarrow{1}F_q(J_{s(q)})$ for $q=\beta p$, the map $B_\beta\colon F_{\beta}(B_\mathtt{j})\to B_\mathtt{k}$, and the composite $F_{\beta}(B_\mathtt{j})\xrightarrow{} \bigoplus_{\substack{\alpha\in Q_1, t(\alpha)=\mathtt{k}}} F_{\alpha}(B_{s(\alpha)})\xrightarrow{e_{\mathtt{k}}} J_{\mathtt{k}}$ where the first map is the inclusion. If the restriction of $e_\mathtt{k}$ to $\ker B_{\mathtt{k},\operatorname{in}}$ is an injective envelope for all $\mathtt{k}\in Q_0$, then $\Mo(B_\mathtt{i},B_\alpha)=\Mimo (B_\mathtt{i},B_\alpha)$ and we get a formula for the Mimo-construction.
	\end{ex}
	
	\begin{ex}\label{Example:MimoFormulaQuiverRep}
		Consider the category of representations $\operatorname{rep}(Q,\mathcal{B})$ as in Example \ref{Example:Quiver Representations}. In this case the Mimo-construction $\operatorname{Mimo}(B_\mathtt{i},B_\alpha)=(B'_\mathtt{i},B'_\alpha)$ of an object $(B_\mathtt{i},B_\alpha)$ is given as follows: Choose an injective envelope $j_\mathtt{i}\colon K_\mathtt{i}\to J_\mathtt{i}$ for each $\mathtt{i}\in Q_0$, where $K_\mathtt{i}$ is the kernel of the morphism
		\[
		B_{\mathtt{i},\arrowin}\colon \bigoplus_{\substack{\alpha\in Q_1\\t(\alpha)=\mathtt{i}}} B_{s(\alpha)}\xrightarrow{(B_\alpha)_\alpha} B_\mathtt{i}.
		\]
		Let $e_{\mathtt{i}}\colon \bigoplus_{\alpha\in Q_1, t(\alpha)=\mathtt{i}} B_{s(\alpha)}\to J_\mathtt{i}$ be a lift of $j_i$ via the inclusion $K_\mathtt{i}\to \bigoplus_{\alpha\in Q_1, t(\alpha)=\mathtt{i}} B_{s(\alpha)}$. Then 
		\[
		B'_{\mathtt{i}}= B_{\mathtt{i}}\oplus \bigoplus_{p\in Q_{\geq 0},t(p)=\mathtt{i}}J_{s(p)}
		\] 
		where $Q_{\geq 0}$ is the set of paths in $Q$, and $s(p)$ and $t(p)$ denotes the source and target of $p$, respectively. For an arrow $\beta\colon \mathtt{i}\to \mathtt{k}$, the morphism
		\[
		B'_\beta\colon B_{\mathtt{i}}\oplus \bigoplus_{p\in Q_{\geq 0},t(p)=\mathtt{i}}J_{s(p)}\to B_{\mathtt{k}}\oplus \bigoplus_{q\in Q_{\geq 0},t(q)=\mathtt{k}}J_{s(q)}
		\]
		is induced by the identity $J_{s(p)}\xrightarrow{1}J_{s(q)}$ for $q=\beta p$, the structure map $B_\beta\colon B_\mathtt{i}\to B_\mathtt{k}$, and the composite $B_\mathtt{i}\xrightarrow{} \bigoplus_{\substack{\alpha\in Q_1, t(\alpha)=\mathtt{k}}} B_{s(\alpha)}\xrightarrow{e_{\mathtt{k}}} J_{\mathtt{k}}$ where the first map is the canonical inclusion. 
		
		Note that this formula has already been obtained in \cite[Section 3a]{LZ13}. In \cite[Lemma 3.2 and Proposition 3.3]{LZ13} they show that it gives a right $\Mono(X)$-approximation. In Theorem \ref{Contravariantly finite} we prove the same results. We also show that it is a minimal right approximation, which was not known before. Note that the proofs are shorter and more transparent in our language. 
	\end{ex}

	\section{A characterization of the indecomposable objects in $\Mono(X)$}\label{subsection: A characterization of the indecomposable objects}
	By Theorem \ref{Theorem: Canonical functor representation equivalence} there is a bijection between isomorphism classes of indecomposable objects in $\overline{\Mono}(X)$ and $\overline{\mathcal{C}}{}^{T(X)}$, induced  by the epivalence
	\[
	\overline{\Mono}(X)\to \overline{\mathcal{C}}{}^{T(X)}.
	\]
	Under some mild additional assumptions this is also in bijection with isomorphism classes of non-injective indecomposable objects in $\Mono(X)$. The goal in this section is to provide an explicit formula for this latter bijection. More precisely, we show that the Mimo-construction extends to objects in $\overline{\mathcal{C}}{}^{T(X)}$, and that it gives a bijection between indecomposable objects in $\overline{\mathcal{C}}{}^{T(X)}$ and non-injective indecomposable objects in $\Mono(X)$. We illustrate the usefulness of this result on examples in Section \ref{Section:Applications}.  For the result to hold $\mathcal{C}$ must admit maximal injective summands, see Definition \ref{Definition:MaximalInjSummand}. We therefore start by showing that any noetherian, artinian, or locally noetherian category admits maximal injective summands.
	\subsection{Maximal injective summand}\label{subsection:LocallyNoetherian}
	
	\begin{defn}\label{Definition:MaximalInjSummand}
		Let $\mathcal{C}$ be an abelian category with injective envelopes. We say that $\mathcal{C}$ \emphbf{admits maximal injective summands} if for any $M\in \mathcal{C}$ there exists an isomorphism 
		\[
		M\cong M'\oplus I
		\]
		where $I$ is injective and $M'$ has no nonzero injective summands.
	\end{defn}
	Note that $M'$ and $I$ are unique up to isomorphism by Lemma \ref{Lemma:RadicalProperties} \eqref{Lemma:RadicalProperties:2} and Lemma \ref{Lemma:UniquenessMaxInjSummand}.
	
	Recall that an object $M\in \mathcal{C}$ is \emphbf{artinian} if any decreasing sequence
	\[
	\dots\subseteq M_1\subseteq M_0 \subseteq M
	\]
	of subobjects of $M$ stabilizes, and \emphbf{noetherian} if any increasing sequence
	\[
	M_0\subseteq M_1\subseteq M_2\subseteq \dots \subseteq M
	\]
	of subobjects of $M$ stabilizes.  The category $\mathcal{C}$ is \emphbf{artinian} if each object in $\mathcal{C}$ is artinian, and \emphbf{noetherian} if each object in $\mathcal{C}$ is noetherian. Finally, $\mathcal{C}$ is called \emphbf{locally noetherian} if it is a Grothendieck category with a generating set of noetherian objects. Note that any locally noetherian category has injective envelopes, see \cite[Corollary 2.5.4]{Kra22}.
	
	\begin{prop}\label{Proposition:AbelianCatWithMaxInjSummand}
		Let $\mathcal{C}$ be an abelian category with injective envelopes. Then $\mathcal{C}$ admits maximal injective summands if one of the following conditions hold:
		\begin{enumerate}
			\item\label{Proposition:AbelianCatWithMaxInjSummand:1} Any injective object in $\mathcal{C}$ can be written as a finite direct sum of indecomposables. 
			\item\label{Proposition:AbelianCatWithMaxInjSummand:2} $\mathcal{C}$ is artinian.
			\item\label{Proposition:AbelianCatWithMaxInjSummand:3} $\mathcal{C}$ is noetherian.
			\item\label{Proposition:AbelianCatWithMaxInjSummand:4} $\mathcal{C}$ is locally noetherian.
		\end{enumerate}
	\end{prop}
	
	\begin{proof}
		Assume property \eqref{Proposition:AbelianCatWithMaxInjSummand:1}. Since $\mathcal{C}$ has injective envelopes, the endomorphism ring of any indecomposable injective object in $\mathcal{C}$ has local endomorphism ring by the proof of \cite[Lemma 2.5.7]{Kra22}. Hence, the injective objects in $\mathcal{C}$ form a Krull--Remak--Schmidt category, so any injective object can be written uniquely as a sum of indecomposable objects up to permutation and isomorphism.
		
		Now let $M\in \mathcal{C}$ be arbitrary, and let $M\to E(M)$ be its injective envelope. Since any injective summand of $M$ must be an injective summand of $E(M)$, it follows that the number of indecomposable summands of $M$ must be bounded by the number of indecomposable summands of $E(M)$, which is finite by hypothesis. So we can choose an injective summand $I$ of $M$ with the maximal amount of indecomposable summands. Writing $M\cong I\oplus M'$ we see that $M'$ has no injective summands, since otherwise there would exist an injective summand of $M$ with more indecomposable summands than $I$. This shows that $\mathcal{C}$ admits maximal injective summands.
		
		Next we show that $\mathcal{C}$ being noetherian or artinian implies condition \eqref{Proposition:AbelianCatWithMaxInjSummand:1}. Indeed, assume an injective object $I$ cannot be written as a finite direct sum of indecomposable injectives. Then there exist nonzero injective objects $I_n,I_n'$ with $I=I_0'\oplus I_0$ and $I_n=I'_{n+1}\oplus I_{n+1}$ for $n\geq 0$. We then have a strictly decreasing and a strictly increasing sequence of subobjects of $I$
		\begin{align*}
		\dots \subset I_2\subset I_1\subset I_0\subset I \quad \text{and} \quad 	I_0'\subset I_0'\oplus I_1'\subset I_0'\oplus I_1'\oplus I_2'\subset \dots \subset I. 
		\end{align*}
		Hence, $\mathcal{C}$ can't be noetherian or artinian.
		
		Finally, assume $\mathcal{C}$ is a locally noetherian category. Then injective objects are closed under filtered colimits, see \cite[Theorem 11.2.12]{Kra22}.  Therefore, Zorn's lemma implies that $M$ has a maximal injective subobject $I$. Furthermore, the inclusion $I\to M$ must be split, since $I$ is injective. Hence, we have an isomorphism $M\cong M'\oplus I$ for some object $M'$. If $M'$ has a nonzero injective summand $J$, then $I\oplus J$ is an injective subobject of $M$ which strictly contains $I$. This contradicts the maximality of $I$. Hence $M'$ has no nonzero injective summands.
	\end{proof}
	
	\begin{ex}
		If $\mathcal{C}$ is the category of quasi-coherent sheaves over a noetherian scheme, or the category of all modules over a noetherian ring, then $\mathcal{C}$ is locally noetherian. The existence of maximal injective summands in the latter case was first shown in \cite{Mat58}.
	\end{ex}
	
	\begin{ex}
		Following \cite{Jan69} a ring $\Lambda$ is called \emphbf{right co-noetherian} if injective envelopes of simple right $\Lambda$-modules are artinian. By \cite[Proposition 2*]{Vam68} this is equivalent to injective envelopes of artinian right modules being artinian. Hence, the category $\mathcal{C}=\operatorname{art}\Lambda$ of artinian right $\Lambda$-modules is an artinian abelian category with injective envelopes, and therefore admits maximal injective summands. Examples of co-noetherian rings are commutative noetherian rings \cite[Proposition 3]{Mat60}, Quasi-Frobenius rings \cite[Proposition 1]{Fai66}, Noetherian PI rings \cite[Theorem 2]{Jat76}, module finite algebras over commutative noetherian rings \cite[Corollary 2.3]{Hir00}, finite normalizing extensions of a right co-noetherian ring \cite[Theorem 2.2]{Hir00}, and the first Weyl algebra of a commutative ring finitely generated as an algebra over the integers \cite[Corollary 2.7]{Hir00}. A commutative ring is co-noetherian if and only if its localizations at any maximal ideal is noetherian \cite[Theorem 2]{Vam68}.
	\end{ex} 
	
	\begin{rmk}
		There exist abelian categories with injective envelopes which do not satisfy Definition \ref{Definition:MaximalInjSummand}. Indeed, let $\mathcal{C}$ be any Grothendieck category which is not locally noetherian, e.g. the category of all modules over a non-noetherian ring. Since $\mathcal{C}$ is Grothendieck, it has injective envelopes, see \cite[Corollary 2.5.4]{Kra22}. Also, since $\mathcal{C}$ is not locally noetherian, there exists a set $\mathcal{J}$ of injective objects such that the sum $M=\bigoplus_{J\in \mathcal{J}}J$ is not injective, see \cite[Theorem 11.2.12]{Kra22}. We claim that $M$ has no maximal injective summand. Assume otherwise, i.e. that $M\cong M'\oplus I$ where $I$ is injective and $M'$ has no nonzero injective summands. Let $\mathcal{J}'\subset \mathcal{J}$ be a finite subset, and let $I'=\bigoplus_{J\in \mathcal{J'}}J$ be the corresponding injective object. Choose a left inverse $M\to I'$ to the inclusion $I'\to M$. Via the isomorphism $M\cong M'\oplus I$ we get morphisms
		\[
		\begin{pmatrix}g_1\\g_2\end{pmatrix}\colon I'\to M'\oplus I \quad \text{and} \quad \begin{pmatrix}g'_1&g'_2\end{pmatrix}\colon M'\oplus I\to I'
		\] 
		such that $g'_1\circ g_1+g'_2\circ g_2=1_{I'}$. By Lemma \ref{Lemma:RadicalProperties} \eqref{Lemma:RadicalProperties:2} the morphism $g_1$ is in the radical of $\mathcal{C}$, so $g_2'\circ g_2=1_{I'}-g_1'\circ g_1$ must be an isomorphism. In particular, $g_2\colon I'\to I$ is a monomorphism. Now consider the morphism $\bigoplus_{J\in \mathcal{J}}J=M\cong M'\oplus I\to I$. We have shown that this is a monomorphism when restricted to the direct sum of any finite subset of $\mathcal{J}$. Since $\bigoplus_{J\in \mathcal{J}}J$ is the filtered colimit of such finite sums, and filtered colimits in Grothendieck categories are exact, the morphism $\bigoplus_{J\in \mathcal{J}}J\to I$ must itself be a monomorphism. Since it is clearly an epimorphism, it must be an isomorphism. But this implies that $M=\bigoplus_{J\in \mathcal{J}}J$ is injective, which is a contradiction.
	\end{rmk}

	\subsection{Construction and main result}	
 
	Let $\mathcal{C}$ be an abelian category with injective envelopes and maximal injective summands, and let $X\colon \mathcal{C}\to \mathcal{C}$ be an exact functor which is locally nilpotent and preserves injective objects. Our goal is to define the Mimo-construction directly on objects in the Eilenberg--Moore category of the stable category $\overline{\mathcal{C}}$. To do this, we need the following lemma.
	
	\begin{lem}\label{Lemma: Lift from stable cat}
		Let $\module{M}\in \overline{\mathcal{C}}{}^{T(X)}$. Then there exists an object $\widehat{\module{M}}\in \mathcal{C}^{T(X)}$ which is isomorphic to $\module{M}$ in $\overline{\mathcal{C}}{}^{T(X)}$ and for which $f^*(\widehat{\module{M}})=\widehat{M}$ has no injective summands.
	\end{lem}
	
	\begin{proof}
		By Lemma \ref{Eilenberg--Moore category of stable free monad} the data of an object $\module{M}$ in $\overline{\mathcal{C}}{}^{T(X)}$ is the same as the data of a morphism $h_\module{M}\colon X(M)\to M$ in $\overline{\mathcal{C}}$. Choose a lift $h'\colon X(M)\to M$ to $\mathcal{C}$ of $h_{\module{M}}$, and write $M\cong\widehat{M}\oplus J$ where $J$ is injective and $\widehat{M}$ has no injective summands. Now $h'$ gives a morphism $X(\widehat{M})\oplus X(J)\to \widehat{M}\oplus J$. Let $h''\colon X(\widehat{M})\to \widehat{M}$ be the restriction of $h'$, and let $\widehat{\module{M}}=(\widehat{M},h'')$ be the corresponding object in $\mathcal{C}^{T(X)}$. Clearly $\widehat{\module{M}}$ is isomorphic to $\module{M}$ in $\overline{\mathcal{C}}{}^{T(X)}$ and $\widehat{M}$ has no injective summands by construction.  
	\end{proof}
	
	We can now extend the Mimo-construction to objects in $\overline{\mathcal{C}}{}^{T(X)}$. Here the superscript $\cong$ indicates that we are considering isomorphism classes of objects.
	
	\begin{prop}\label{Proposition:MimoOnStableCat}
		For any $\module{M}\in \overline{\mathcal{C}}{}^{T(X)}$, choose an object $\widehat{\module{M}}$ in $\mathcal{C}^{T(X)}$ which is isomorphic to $\module{M}$ in $\overline{\mathcal{C}}{}^{T(X)}$ and such that $f^*(\widehat{\module{M}})=\widehat{M}$ has no nonzero injective summands in $\mathcal{C}$. Then the association $\module{M}\mapsto \Mimo \widehat{\module{M}}$ induces a well-defined map
		\begin{align}\label{MimoMapOnStable} 
		\{\text{objects in }\overline{\mathcal{C}}{}^{T(X)}\}^{\cong} \xrightarrow{\Mimo} \{\text{objects in } \Mono(X)\}^{\cong}.
		\end{align} 
		Furthermore, it is independent of choice of $\widehat{\module{M}}$.
	\end{prop}
	
	\begin{proof}
		It is well-defined and independent of choice by Corollary \ref{Mimo well-defined up to factoring through injective} and Lemma \ref{Lemma: Lift from stable cat}.
	\end{proof}
	
	We want to show that the map in Proposition \ref{Proposition:MimoOnStableCat} restricts to an isomorphism between indecomposable objects in $\overline{\mathcal{C}}{}^{T(X)}$ and non-injective indecomposable objects in $\Mono(X)$. We first show that it can be used to give an inverse to the bijection 
	\begin{align}\label{Bijection:RepEquivIndec}
	\{\text{objects in } \overline{\Mono}(X)\}^{\cong}\to \{\text{objects in }\overline{\mathcal{C}}{}^{T(X)}\}^{\cong}  
	\end{align}
	coming from the epivalence $\overline{\Mono}(X)\to \overline{\mathcal{C}}{}^{T(X)}$ in Theorem \ref{Theorem: Canonical functor representation equivalence}. 
	\begin{prop}\label{MimoInverseToRepEquivalence}
		Composing \eqref{MimoMapOnStable} with the functor $\Mono(X)\to \overline{\Mono}(X)$ gives an inverse to the bijection \eqref{Bijection:RepEquivIndec}. In particular, \eqref{MimoMapOnStable} is injective and preserves indecomposables.
	\end{prop}
	
	\begin{proof}
		Let $\module{M}\in \overline{\mathcal{C}}{}^{T(X)}$. Since the canonical morphism $\Mimo \widehat{\module{M}}\to \widehat{\module{M}}$ is an isomorphism in $\overline{\mathcal{C}}{}^{T(X)}$, and there is a canonical isomorphism $\widehat{\module{M}}\cong \module{M}$ in $\overline{\mathcal{C}}{}^{T(X)}$, the composite of \eqref{MimoMapOnStable} with the functor $\Mono(X)\to \overline{\Mono}(X)$ must be an inverse to \eqref{Bijection:RepEquivIndec}. In particular, the composite is injective, so \eqref{MimoMapOnStable} must be injective. To see that \eqref{MimoMapOnStable} preserves indecomposables, note that its composite with the functor $\Mono(X)\to \overline{\Mono}(X)$ preserves indecomposables, since it is an inverse \eqref{Bijection:RepEquivIndec}. Furthermore,  $\Mimo \widehat{\module{M}}$ has no nonzero injective summands by Lemma \ref{No injective summands}, and must therefore be indecomposable in $\Mono(X)$ if it is indecomposable in $\overline{\Mono}(X)$. This proves the claim.
	\end{proof}

	We now characterize the indecomposables in $\Mono(X)$. This is a very useful result, in particular when $\overline{\mathcal{C}}$ is abelian, so that the indecomposable objects in $\overline{\mathcal{C}}{}^{T(X)}$ are easier to compute. We illustrate its power in the next section for quiver representations of different classes of Artin algebras.
	
	\begin{thm}\label{Theorem: Characterization of indecomposables}
		We have bijections 
		\begin{align}
		\renewcommand{\arraystretch}{1.8}
		\begin{array}{ccc}
		\renewcommand{\arraystretch}{1.1}
		\begin{Bmatrix}
		\text{Indecomposable injective} \\
		\text{objects in $\mathcal{C}$}
		\end{Bmatrix}^{\cong}
		&
		\xrightarrow{\cong}
		&
		\renewcommand{\arraystretch}{1.2}
		\begin{Bmatrix}
		\text{Indecomposable injective} \\
		\text{objects in $\Mono(X)$}
		\end{Bmatrix}^{\cong} \quad I \mapsto  f_!(I).  \\ 
		\end{array}\label{Bijection:Injective}
		\renewcommand{\arraystretch}{1}
		\end{align}
		\begin{align}
		\renewcommand{\arraystretch}{1.8}
		\begin{array}{ccc}
		\renewcommand{\arraystretch}{1.1}
		\begin{Bmatrix}
		\text{Indecomposable} \\
		\text{objects in $\overline{\mathcal{C}}{}^{T(X)}$}
		\end{Bmatrix}^{\cong}
		&
		\xrightarrow{\cong}
		&
		\renewcommand{\arraystretch}{1.2}
		\begin{Bmatrix}
		\text{Indecomposable non-injective} \\
		\text{objects in $\Mono(X)$}
		\end{Bmatrix}^{\cong} \quad \module{M} \mapsto  \Mimo \module{M}.\\ 
		\end{array}\label{Bijection:Non-injective}
		\renewcommand{\arraystretch}{1}
		\end{align}
	\end{thm}
	
	\begin{proof}
		By Proposition \ref{Proposition:f_!(I)Injective} and Corollary \ref{Corollary:InjectivesinMono} the injective objects in $\Mono(X)$ are precisely the objects of the form $f_!(J)$ with $J$ injective in $\mathcal{C}$. Furthermore, $f_!(J)$ is indecomposable if and only if $J$ is indecomposable by Corollary \ref{cor:f_!Indecomposable}. This implies that the map \eqref{Bijection:Injective} is well-defined and surjective. For injectivity, note that $f_!(J)\cong f_!(J')$ implies $J\cong \Kopf f_!(J)\cong \Kopf f_!(J')\cong J'$. 
		
	 Next, consider the map in \eqref{Bijection:Non-injective}. It is well-defined and injective by Proposition \ref{MimoInverseToRepEquivalence}. To see that it is surjective, let $\module{M}'\in \Mono(X)$ be an arbitrary indecomposable non-injective object, and let $\module{M}$ be the image of $\module{M}'$ in $\overline{\mathcal{C}}{}^{T(X)}$. Since composing \eqref{MimoMapOnStable} with the functor $\Mono(X)\to \overline{\Mono}(X)$ gives an inverse to \eqref{Bijection:RepEquivIndec} by Proposition \ref{MimoInverseToRepEquivalence}, we get that $\module{M}'\cong \Mimo \widehat{\module{M}}$ in $\overline{\Mono}(X)$. Hence, there exist inverse isomorphisms 
		\[
		\begin{pmatrix}
		\phi_1&\phi_2\\
		\phi_3&\phi_4
		\end{pmatrix}\colon \module{M}'\oplus f_!(I)\to \Mimo \widehat{\module{M}}\oplus f_!(J)
		\quad \text{and} \quad \begin{pmatrix}
		\psi_1&\psi_2\\
		\psi_3&\psi_4
		\end{pmatrix}\colon \Mimo \widehat{\module{M}}\oplus f_!(J)\to \module{M}'\oplus f_!(I)
		\]
		in $\Mono(X)$ for some injective objects $I$ and $J$ in $\mathcal{C}$. In particular, we have that 
		\[
		1_{\Mimo \widehat{\module{M}}}=\phi_1\circ \psi_1+\phi_2\circ \psi_3.
		\] 
		Now by Lemma \ref{No injective summands} the morphism $\phi_2$ is in the radical of $\Mono(X)$, so $\phi_1\circ \psi_1=1_{\Mimo \widehat{\module{M}}}-\phi_2\circ \psi_3$ is an isomorphism. Hence $\phi_1\colon \module{M}'\to \Mimo \widehat{\module{M}}$ is a split epimorphism, and since $\module{M}'$ is indecomposable it must be an isomorphism. This shows that \eqref{Bijection:Non-injective} is surjective.
	\end{proof}
	
	\begin{rmk}
		Theorem \ref{Theorem: Characterization of indecomposables} recovers \cite[Theorem 2]{RZ17} in the setting of Remark \ref{Remark:DualNumbers}.
	\end{rmk}
	
		\begin{ex}\label{Example:MimoStableModulation}
		Let $Q$ be a finite and acyclic quiver and $\mathfrak{B}$ a modulation by abelian categories $\mathcal{B}_\mathtt{i}$ and exact functors $F_\alpha$ preserving injective objects, see Example \ref{Example: Phyla}. Set $\mathcal{C}=\prod_{\mathtt{i}\in Q_0}\mathcal{B}_\mathtt{i}$. Then $\mathcal{C}$ has injective envelopes (resp. maximal injective summands) if and only if $\mathcal{B}_\mathtt{i}$ has injective envelopes (resp. maximal injective summands), for all $\mathtt{i}\in Q_0$. The Mimo map in Proposition \ref{Proposition:MimoOnStableCat} goes from $\operatorname{rep}\overline{\mathfrak{B}}$ to the monomorphism category of $\mathfrak{B}$, where $\overline{\mathfrak{B}}$ is the modulation by stable categories as in Example \ref{Example:ModulationStableCats}. Explicitly, it sends $(B_\mathtt{i},B_\alpha)$ to $\Mimo (\widehat{B}_\mathtt{i},B'_\alpha)$ where $\widehat{B}_\mathtt{i}$ is any object in $\mathcal{B}_\mathtt{i}$ with no nonzero injective summands and which is isomorphic to $B_{\mathtt{i}}$ in $\overline{\mathcal{B}}_\mathtt{i}$, and $B'_\alpha\colon F_{\alpha}(\widehat{B}_{s(\alpha)})\to \widehat{B}_{t(\alpha)}$ is any lift of $B_\alpha$ to $\mathcal{B}_{t(\alpha)}$, and $\Mimo (\widehat{B}_\mathtt{i},B'_\alpha)$ is the Mimo-construction described in Example \ref{Example: Mimo construction}.
	\end{ex}

	\section{Applications to quiver representations over Artin algebras}\label{Section:Applications}
	
	The goal of this section is to study monomorphism categories of quivers over Artin algebras, using our results. Throughout the section $Q$ denotes a finite acyclic quiver, $\mathbbm{k}$ is a commutative artinian ring, and $\Lambda$ is an Artin $\mathbbm{k}$-algebra. The categories of representations of $Q$ are denoted by $\rep (Q,\Mod \Lambda)$ and $\rep (Q,\operatorname{mod} \Lambda)$, see Example \ref{Example:Quiver Representations}. They can be identified with the module categories $\operatorname{Mod}\Lambda Q^{\op}$ and $\operatorname{mod}\Lambda Q^{\op}$, respectively. If we set $\Lambda Q_0\coloneqq\prod_{\mathtt{i}\in Q_0}\Lambda$ then we get 
	\[
	\operatorname{Mod}\Lambda Q_0=\prod_{\mathtt{i}\in Q_0}\operatorname{Mod}\Lambda \quad \text{and} \quad \operatorname{mod}\Lambda Q_0=\prod_{\mathtt{i}\in Q_0}\operatorname{mod}\Lambda.
	\]
	We have adjoint functors 
	\[
	f_!\colon \operatorname{Mod}\Lambda Q_0\to \rep (Q,\Mod \Lambda) \quad \text{and} \quad f^*\colon \rep (Q,\Mod \Lambda)\to \operatorname{Mod}\Lambda Q_0
	\]
	which restrict to
	\[
	f_!\colon \operatorname{mod}\Lambda Q_0\to \rep (Q,\operatorname{mod} \Lambda) \quad \text{and} \quad f^*\colon \rep (Q,\operatorname{mod} \Lambda)\to \operatorname{mod}\Lambda Q_0.
	\]
	The monomorphism subcategories of $\rep (Q,\Mod \Lambda)$ and $\rep (Q,\operatorname{mod} \Lambda)$ are denoted by $\Mono_Q(\Lambda)$ and $\mono_Q(\Lambda)$, respectively. They consist of representations $(M_\mathtt{i},M_\alpha)_{\mathtt{i}\in Q_0, \alpha\in Q_1}$ for which 
	\[
	M_{\mathtt{i},\operatorname{in}}\colon \bigoplus_{\substack{\alpha\in Q_1\\t(\alpha)=\mathtt{i}}}M_{s(\alpha)}\xrightarrow{(M_\alpha)} M_\mathtt{i}
	\]
	is a monomorphism for all $\mathtt{i}\in Q_0$, see Example \ref{Example: Monomorphism category}. Next we recall the bijection in Theorem \ref{Theorem: Characterization of indecomposables} in this context. For $M\in \operatorname{mod}\Lambda$ we let $M(\mathtt{i})\in \operatorname{mod}\Lambda Q_0$ denote the object given by $$M(\mathtt{i})_\mathtt{j}=
	\begin{cases}
	M, & \text{if}\ \mathtt{j}=\mathtt{i} \\
	0, & \text{if}\ \mathtt{j}\neq\mathtt{i}.
	\end{cases}$$
	
	\begin{thm}\label{Theorem: Characterization of indecomposablesArtin}
		We have bijections
		\begin{align*}
		\renewcommand{\arraystretch}{1.8}
		\begin{array}{ccc}
		\renewcommand{\arraystretch}{1.1}
		\begin{Bmatrix}
		\text{Indecomposable objects} \\
		\text{in $\rep (Q,\overline{\operatorname{Mod}}\, \Lambda)$}
		\end{Bmatrix}^{\cong}
		&
		\xrightarrow{\cong}
		&
		\renewcommand{\arraystretch}{1.2}
		\begin{Bmatrix}
		\text{Indecomposable non-injective} \\
		\text{objects in $\Mono_Q(\Lambda)$}
		\end{Bmatrix}^{\cong} \quad \module{M} \mapsto  \Mimo \module{M}\\ 
		\end{array}
		\renewcommand{\arraystretch}{1}\\
		\renewcommand{\arraystretch}{1.8}
		\begin{array}{ccc}
		\renewcommand{\arraystretch}{1.1}
		\begin{Bmatrix}
		\text{Indecomposable objects} \\
		\text{in $\rep (Q,\overline{\operatorname{mod}}\, \Lambda)$}
		\end{Bmatrix}^{\cong}
		&
		\xrightarrow{\cong}
		&
		\renewcommand{\arraystretch}{1.2}
		\begin{Bmatrix}
		\text{Indecomposable non-injective} \\
		\text{objects in $\mono_Q(\Lambda)$}
		\end{Bmatrix}^{\cong} \quad \module{M} \mapsto  \Mimo \module{M}\\ 
		\end{array}
		\renewcommand{\arraystretch}{1} \\
		\renewcommand{\arraystretch}{1.8}
		\begin{array}{ccc}
		\renewcommand{\arraystretch}{1.1}
		\begin{Bmatrix}
		\text{Indecomposable injective } \\
		\text{right $\Lambda$-modules}
		\end{Bmatrix}^{\cong}\times Q_0
		&
		\xrightarrow{\cong}
		&
		\renewcommand{\arraystretch}{1.2}
		\begin{Bmatrix}
		\text{Indecomposable injective} \\
		\text{objects in $\mono_Q(\Lambda)$}
		\end{Bmatrix}^{\cong} \quad (I,\mathtt{j}) \mapsto  f_!(I(\mathtt{j}))\\ 
		\end{array}
		\renewcommand{\arraystretch}{1}
		\end{align*}
	\end{thm}

	\subsection{Stable equivalences}\label{Subsecition:StableEquivalences}
	
	Throughout this subsection we fix two Artin $\mathbbm{k}$-algebras $\Lambda$ and $\Gamma$. Our goal is to investigate the categories $\mono_Q(\Lambda)$ and $\mono_Q(\Gamma)$ when $\Lambda$ and $\Gamma$ are stably equivalent. We show that there is a bijection between indecomposable non-injective objects in these categories, and we investigate when this bijection commutes with an induced map on the split  Grothendieck groups of $\operatorname{mod}\Lambda$ and $\operatorname{mod}\Gamma$, see Theorem \ref{Theorem:StableEquivBijection}. 
	
	We start by showing that a stable equivalence induces a bijection between indecomposable non-injective objects in the monomorphism categories.
	
	\begin{thm}\label{Theorem:SimpleStableEquiv}
		Assume we are given an equivalence $\overline{\modu}\, \Lambda\xrightarrow{\cong} \overline{\modu}\, \Gamma$. Then there exists a bijection $\phi_Q$ between isomorphism classes of indecomposable non-injective objects in $\mono_Q(\Lambda)$ and $\mono_Q(\Gamma)$.    
	\end{thm}
	
	\begin{proof}
		The equivalence $\overline{\modu}\, \Lambda\cong \overline{\modu}\, \Gamma$ induces an equivalence $\rep (Q,\overline{\operatorname{mod}}\, \Lambda)\cong \rep (Q,\overline{\operatorname{mod}}\, \Gamma)$. Since by Theorem \ref{Theorem: Characterization of indecomposablesArtin} there is a bijection between the isomorphism classes of indecomposable objects in $\rep (Q,\overline{\operatorname{mod}}\, \Lambda)$ (respectively, $\rep (Q,\overline{\operatorname{mod}}\, \Gamma)$), and the non-injective indecomposable objects in $\mono_Q(\Lambda)$ (respectively, $\mono_Q(\Gamma)$), we get the bijection $\phi_Q$.
	\end{proof}
	
	\begin{cor}
		Assume $\Lambda$ and $\Gamma$ are stably equivalent. Then $\mono_Q(\Lambda)$ is representation-finite if and only if $\mono_Q(\Gamma)$ is representation-finite.
	\end{cor}
	\begin{proof}
		Since $\mono_Q(\Lambda)$ and $\mono_Q(\Gamma)$ have finitely many indecomposable injective objects up to isomorphism, it follows that $\mono_Q(\Lambda)$ and $\mono_Q(\Gamma)$ are representation-finite if and only if they have finitely many indecomposable non-injective objects up to isomorphism. The claim now follows from Theorem \ref{Theorem:SimpleStableEquiv}.
	\end{proof}
	
	Next we investigate when $\phi_Q$ in Theorem \ref{Theorem:SimpleStableEquiv} can be extended to a bijection between all indecomposable objects.

	\begin{thm}\label{Derived equivalence}
		Assume $\Lambda$ and $\Gamma$ are selfinjective and derived equivalent. Then there exists a bijection between the isomorphism classes of indecomposable objects in $\mono_Q(\Lambda)$ and $\mono_Q(\Gamma)$. 
	\end{thm}
	
	\begin{proof}
		By a result of Rickard \cite[Corollary 2.2]{Ric89b}, a derived equivalence induces a stable equivalence for selfinjective algebras, and thus by Theorem  \ref{Theorem:SimpleStableEquiv} we have a bijection $\phi_Q$ between the indecomposable non-injective objects in $\mono_Q(\Lambda)$ and $\mono_Q(\Gamma)$. Furthermore, since $\Lambda$ and $\Gamma$ are derived equivalent, the number of indecomposable projective $\Lambda$- and $\Gamma$-modules are the same. Hence, the number of indecomposable injective $\Lambda$- and $\Gamma$-modules are also the same. Therefore, by Theorem \ref{Theorem: Characterization of indecomposablesArtin} the number of indecomposable injective objects in $\mono_Q(\Lambda)$ and $\mono_Q(\Gamma)$ are equal. Hence, $\phi_Q$ can be extended to a bijection between all indecomposable objects, which proves the claim. 
	\end{proof}
	
	In general, the number of indecomposable injective objects in $\mono_Q(\Lambda)$ and $\mono_Q(\Gamma)$ need not coincide even if $\Lambda$ and $\Gamma$ are injectively stably equivalent as the following example shows:
	
	\begin{ex}
		Let $\Lambda=\Bbbk(\mathtt{1}\to \mathtt{2})$ and let $\Gamma=\Bbbk[x]/(x^2)$, and let $Q$ be the quiver with one vertex and no arrows. Then $\overline{\modu}\, \Lambda\cong \overline{\modu}\, \Gamma\cong \modu \Bbbk$. However, $\mono_{Q}(\Lambda)\cong \modu \Lambda$ has two indecomposable injective objects while $\mono_{Q}(\Gamma)\cong \modu \Gamma$ has only one. 
	\end{ex}
	
	However, the conclusion holds for certain selfinjective algebras over algebraically closed fields:
	
	\begin{cor}\label{selfinjective_stable_number}
		Let $\Lambda$ and $\Gamma$ be two connected selfinjective $\Bbbk$-algebras of finite representation type where $\Bbbk$ is an algebraically closed field. Assume $\Lambda$ and $\Gamma$ are stably equivalent. Then, there exists a bijection between isomorphism classes of indecomposable objects in $\mono_Q(\Lambda)$ and $\mono_Q(\Gamma)$.
	\end{cor}
	
	\begin{proof}
		By \cite[Corollary 2.2]{Asa99} two selfinjective algebras of finite representation type are stably equivalent if and only if they are derived equivalent. The claim follows now from Theorem \ref{Derived equivalence}. 
	\end{proof}
	
	Next we investigate when $\phi_Q$ commutes with an induced map on the split  Grothendieck groups of $\operatorname{mod}\Lambda$ and $\operatorname{mod}\Gamma$. For this we need some preliminary results. 
	
	Recall that the \emphbf{socle} of a module $M$, denoted $\operatorname{soc}M$, is the sum of all its simple submodules. It induces a left exact functor $\operatorname{soc}(-)\colon \operatorname{mod} \Lambda\to \operatorname{mod}\Lambda$ which sends a morphism $g\colon M\to N$ to its restriction $\operatorname{soc}(g)\colon \operatorname{soc}M\to \operatorname{soc}N$. Let $\mathcal{S}_\Lambda$ and $\mathcal{S}_\Gamma$ denote the subcategories of semisimple $\Lambda$- and $\Gamma$-modules with no nonzero injective summands, respectively. For an object $M\in \overline{\operatorname{mod}}\, \Lambda$ we write $\widehat{M}$ for a $\Lambda$-module which has no nonzero injective summand and which is isomorphic to $M$ in $\overline{\operatorname{mod}}\, \Lambda$. Note that $\widehat{M}$ is unique up to isomorphism.
	
	\begin{lem}\label{Lemma:SocleFunctorStable}
		The associations 
		\[
		M\mapsto \operatorname{soc}\widehat{M} \quad \text{and} \quad (M\xrightarrow{g}N)\mapsto (\operatorname{soc}\widehat{M}\xrightarrow{\operatorname{soc}(\widehat{g})}\operatorname{soc}\widehat{N})
		\]
		induce a functor $\operatorname{soc}\colon \overline{\operatorname{mod}}\, \Lambda\to \mathcal{S}_\Lambda$, where $\widehat{g}$ is any choice of a lift of $g$ to $\operatorname{mod}\Lambda$. Furthermore, this functor is right adjoint to the inclusion functor $\mathcal{S}_\Lambda\to \overline{\operatorname{mod}}\, \Lambda$.
	\end{lem}
	
	\begin{proof}
		Let $M,N\in \overline{\operatorname{mod}}\, \Lambda$. To prove that we have a well-defined functor, it suffices to show that the map 
		\[
		\operatorname{Hom}_{\Lambda}(\widehat{M},\widehat{N})\to \operatorname{Hom}_{\Lambda}(\operatorname{soc}\widehat{M},\operatorname{soc}\widehat{N}) \quad \quad g\mapsto \operatorname{soc}(g)
		\]
		vanishes on any morphism factoring through an injective object. So assume $g\colon \widehat{M}\to \widehat{N}$ can be written as a composite $\widehat{M}\xrightarrow{g_1} I\xrightarrow{g_2} \widehat{N}$ where $I$ is an injective $\Lambda$-module. Now by Lemma \ref{Lemma:RadicalProperties} \eqref{Lemma:RadicalProperties:2} we know that the inclusion $\operatorname{ker}g_2\to I$ is an injective envelope. Since the socle of an injective envelope is an isomorphism, it follows that the map $\operatorname{soc}\operatorname{ker}g_2\to \operatorname{soc}I$ is an isomorphism. Since the socle is left exact, we get that $\operatorname{soc}(g_2)=0$, and hence $\operatorname{soc}(g)=0$. 
		
		To prove that $\operatorname{soc}\colon \overline{\operatorname{mod}}\, \Lambda\to \mathcal{S}_\Lambda$ is right adjoint to the inclusion functor, it suffices to show that there is a natural isomorphism
		\[
		\overline{\operatorname{Hom}}_{\Lambda}(S,M)\xrightarrow{\cong} \operatorname{Hom}_\Lambda (S,\operatorname{soc}\widehat{M})
		\]
		when $S$ is simple non-injective and $M\in \overline{\operatorname{mod}}\, \Lambda$. First note that there is an isomorphism 
		\[
		\overline{\operatorname{Hom}}_{\Lambda}(S,M)\xrightarrow{\cong} \overline{\operatorname{Hom}}_{\Lambda}(S,\widehat{M})
		\]
		since $M$ and $\widehat{M}$ are isomorphic in $\overline{\operatorname{mod}}\, \Lambda$. Let  $g\colon S\to \widehat{M}$ be a morphism in $\operatorname{mod}\Lambda$ which factors through an injective module. Then it factors through the injective envelope $I$ of $S$ via a morphism $I\to \widehat{M}$. If $g$ is nonzero, then it must be a monomorphism, and hence the induced morphism $I\to \widehat{M}$ must be a monomorphism. Since $I$ is injective, the monomorphism must split, so $I$ must be a summand of $\widehat{M}$. This contradicts the definition of $\widehat{M}$. Hence there are no nonzero morphisms $S\to \widehat{M}$ which factor through an injective object. It follows that the canonical map
		\[
		\operatorname{Hom}_{\Lambda}(S,\widehat{M})\xrightarrow{\cong}\overline{\operatorname{Hom}}_{\Lambda}(S,\widehat{M})
		\]
		is an isomorphism. Finally, any morphism from a simple module to $\widehat{M}$ must factor through the socle of $\widehat{M}$. Hence, we have an isomorphism  $\operatorname{Hom}_{\Lambda}(S,\widehat{M})\xrightarrow{\cong}\operatorname{Hom}_{\Lambda}(S,\operatorname{soc}\widehat{M})$. Combining these isomorphisms, we get the result.    
	\end{proof}
	
	In the following let $K_0(\operatorname{mod}\Lambda)$ and $K_0(\operatorname{mod}\Gamma)$ denote the split Grothendieck groups of the categories $\operatorname{mod}\Lambda$ and $\operatorname{mod}\Gamma$, respectively. The image of a module $M$  in the split Grothendieck group is denoted by $[M]$. Let $G_\Lambda$ and $G_\Gamma$ be the free subgroups of $K_0(\operatorname{mod}\Lambda)$ and $K_0(\operatorname{mod}\Gamma)$, respectively, which are generated by all elements $[M]$ where $M$ is indecomposable and not simple injective. 
	
	\begin{lem}\label{Lemma:IsoGrothendieckGroups}
		Assume we are given an equivalence $\Phi\colon \overline{\modu}\, \Lambda\cong \overline{\modu}\, \Gamma$ which restricts to an equivalence between the simple non-injective $\Lambda$- and $\Gamma$-modules. Then there exists an isomorphism 
		\[
		\phi\colon G_\Lambda\xrightarrow{\cong} G_\Gamma
		\]
		uniquely defined by:
		\begin{itemize}
			\item If $[M]\in G_\Lambda$ is indecomposable non-injective, then $\phi([M])=[N]$ where $N$ is the indecomposable non-injective $\Gamma$-module whose image in $\overline{\operatorname{mod}}\, \Gamma$ is isomorphic to $\Phi(M)$. 
			\item If $[I]\in G_\Lambda$ where $I$ is indecomposable, injective, and not simple, then $\phi([I])=[J]$ where $J$ is the indecomposable injective $\Gamma$-module whose socle is isomorphic to $\Phi(\operatorname{soc}I)$ in $\overline{\modu}\, \Gamma$.
		\end{itemize}
		If in addition there exists a bijection $\psi$ between the isomorphism classes of simple injective $\Lambda$- and $\Gamma$-modules, then $\phi$ can be extended uniquely to an isomorphism
		\[
		\phi\colon K_0(\operatorname{mod}\Lambda)\xrightarrow{\cong} K_0(\operatorname{mod}\Gamma).
		\]
		by setting $\phi([S])=[\psi(S)]$ for any simple injective $\Lambda$-module $S$.
		\begin{proof}
			It is clear by construction that $\phi$ gives a bijection between the sets of isomorphism classes of indecomposable $\Lambda$- and $\Gamma$-modules which are not simple injective. Since $G_{\Lambda}$ and $G_\Gamma$ are the free groups on these sets, this shows that $\phi\colon G_\Lambda \to G_\Gamma$ is an isomorphism. Under the assumption that $\psi$ exists, we see that $\phi$ restricts to a bijection between all isomorphism classes of indecomposable $\Lambda$- and $\Gamma$-modules. Hence, $\phi\colon K_0(\operatorname{mod}\Lambda)\xrightarrow{} K_0(\operatorname{mod}\Gamma)$ must be an isomorphism.
		\end{proof}     
		
	\end{lem}

	\begin{thm}\label{Theorem:StableEquivBijection}
		Assume we are given an equivalence $\Phi\colon\overline{\modu}\, \Lambda\xrightarrow{\cong} \overline{\modu}\, \Gamma$ which restricts to an equivalence between the simple non-injective $\Lambda$- and $\Gamma$-modules. Let $\phi_Q$ denote the bijection in Theorem \ref{Theorem:SimpleStableEquiv} and let $\phi$ denote the isomorphism in Lemma \ref{Lemma:IsoGrothendieckGroups}. The following hold:
		\begin{enumerate}
			\item\label{Theorem:StableEquivBijection:1} $\phi_Q$ commutes pointwise with $\phi$, i.e. 
			\[ [\phi_Q(\module{M})_\mathtt{i}]=\phi([\module{M}_\mathtt{i}]).
			\] 
			\item\label{Theorem:StableEquivBijection:2} Assume the existence of a bijection between the simple injective $\Lambda$- and $\Gamma$-modules, and let $\phi \colon K_0(\operatorname{mod}\Lambda)\xrightarrow{\cong} K_0(\operatorname{mod}\Gamma)$ be the extension given in Lemma \ref{Lemma:IsoGrothendieckGroups}. Then $\phi_Q$ can be extended uniquely to a bijection between the isomorphism classes of all indecomposable objects in $\mono_Q(\Lambda)$ and $\mono_Q(\Gamma)$ such that
			\[ [\phi_Q(\module{M})_\mathtt{i}]=\phi([\module{M}_\mathtt{i}]).
			\]
		\end{enumerate}
		
	\end{thm}

	\begin{proof}
		Consider the functors $\operatorname{soc}\colon \overline{\operatorname{mod}}\, \Lambda \to \mathcal{S}_\Lambda$ and $\operatorname{soc}\colon \overline{\operatorname{mod}}\, \Gamma \to \mathcal{S}_\Gamma$ from Lemma \ref{Lemma:SocleFunctorStable}. We claim that the following square 
		\begin{equation}\label{Equation:CommutativeSquare}
		\begin{tikzcd}
		\overline{\modu}\, \Lambda\arrow{d}{\cong}\arrow{r}{\operatorname{soc}}&\mathcal{S}_\Lambda \arrow{d}{\cong}  \\
		\overline{\modu}\, \Gamma\arrow{r}{\operatorname{soc}}{}&\mathcal{S}_\Gamma
		\end{tikzcd}
		\end{equation}
		commutes up to natural isomorphism, where the vertical functors are given by $\Phi$ and its restriction to the simple non-injective modules. Indeed, this follows from the horizontal functors being left adjoint to the inclusion functors $\mathcal{S}_\Lambda\to \overline{\modu}\, \Lambda$ and $\mathcal{S}_\Gamma\to \overline{\modu}\, \Gamma$ by Lemma \ref{Lemma:SocleFunctorStable}, and the fact that the vertical functors commute with the inclusion functors.
		Hence, postcomposing with the functors in \eqref{Equation:CommutativeSquare}, we get a diagram
		\begin{equation}\label{Equation:PointwiseCommutativeSquare}
		\begin{tikzcd}
		\operatorname{rep}(Q,\overline{\operatorname{mod}}\, \Lambda)\arrow{d}{\cong}\arrow{r}{\operatorname{soc}}&\operatorname{rep}(Q,\mathcal{S}_\Lambda)\arrow{d}{\cong}  \\
		\operatorname{rep}(Q,\overline{\operatorname{mod}}\, \Gamma)\arrow{r}{\operatorname{soc}}{}&\operatorname{rep}(Q,\mathcal{S}_\Gamma)
		\end{tikzcd}
		\end{equation}
		of functors which commutes up to natural isomorphism. 
		
		Now let $\module{M}$ be an indecomposable object in $\operatorname{rep}(Q,\overline{\operatorname{mod}}\, \Lambda)$, and let $\module{N}$ be the corresponding indecomposable object in $\operatorname{rep}(Q,\overline{\operatorname{mod}}\, \Gamma)$ under the left vertical equivalence in \eqref{Equation:PointwiseCommutativeSquare}.  Let $\widehat{\module{M}}$ and $\widehat{\module{N}}$ be objects as in Lemma \ref{Lemma: Lift from stable cat}, so that $\Mimo \module{M}=\Mimo \widehat{\module{M}}$ and $\Mimo \module{N}=\Mimo \widehat{\module{N}}$. By construction we have that $\phi_Q(\Mimo \module{M})=\Mimo \module{N}$. Also
		\[
		(\Mimo \module{M})_{\mathtt{k}}\cong \widehat{\module{M}}_{\mathtt{k}}\oplus  I_{\mathtt{k}} \quad \text{and} \quad (\Mimo \module{N})_{\mathtt{k}}\cong \widehat{N}_{\mathtt{k}}\oplus I_{\mathtt{k}}'
		\]
		for all $\mathtt{k}\in Q_0$ where $I_{\mathtt{k}}$ and $I'_{\mathtt{k}}$ are injective $\Lambda$- and $\Gamma$-modules, respectively. By construction of $\phi$ it follows that $\phi([\widehat{\module{M}}_{\mathtt{k}}])=[\widehat{\module{N}}_{\mathtt{k}}]$, so we only need to show that $\phi[I_{\mathtt{k}}]=[I_{\mathtt{k}}']$. Now by the Mimo-construction in Example \ref{Example: Mimo construction} 
		we get 
		\[
		I_{\mathtt{k}}\cong \bigoplus_{\substack{p\in Q_{\geq 0}\\ t(p)=\mathtt{k}}}J_{s(p)} \quad \text{and} \quad I_{\mathtt{k}}'\cong \bigoplus_{\substack{p\in Q_{\geq 0}\\ t(p)=\mathtt{k}}}J_{s(p)}' 
		\]
		where $J_\mathtt{i}$ and $J_{\mathtt{i}}'$ are the injective envelopes of $\operatorname{ker}\widehat{M}_{\mathtt{i},\operatorname{in}}$ and $\operatorname{ker}\widehat{N}_{\mathtt{i},\operatorname{in}}$, respectively. Since $\phi$ is additive, it suffices to show that $\phi([J_\mathtt{i}])=[J_\mathtt{i}']$ for each $\mathtt{i}\in Q_0$.  By definition of $\phi$, this is equivalent to requiring $\phi([\operatorname{soc}J_{\mathtt{i}}])=[\operatorname{soc}J_{\mathtt{i}}']$. Since the socle of a module and the socle of its injective envelope are isomorphic, it follows that $[\operatorname{soc}J_{\mathtt{i}}]=[\operatorname{soc}\operatorname{ker}\widehat{M}_{\mathtt{i},\operatorname{in}}]$ and $[\operatorname{soc}J'_{\mathtt{i}}]=[\operatorname{soc}\operatorname{ker}\widehat{N}_{\mathtt{i},\operatorname{in}}]$. Hence, we need to show that $\phi([\operatorname{soc}\operatorname{ker}\widehat{M}_{\mathtt{i},\operatorname{in}}])=[\operatorname{soc}\operatorname{ker}\widehat{N}_{\mathtt{i},\operatorname{in}}]$. But this follows immediately from the commutativity of \ref{Equation:PointwiseCommutativeSquare}, which proves \eqref{Theorem:StableEquivBijection:1}.
		
		Now assume we have a bijection between the simple $\Lambda$- and $\Gamma$-modules as in \eqref{Theorem:StableEquivBijection:2}. We want to extend $\phi_Q$ to a bijection between all indecomposable objects in $\mono_Q(\Lambda)$ and $\mono_Q(\Gamma)$. To do this we need to define it on the injective objects. Assume $\module{M}\in \mono_Q(\Lambda)$ is indecomposable injective. Then it is of the form $f_!(J(\mathtt{i}))$ for an indecomposable injective $\Lambda$-module $J$ and a vertex $\mathtt{i}$ in $Q$, see Theorem \ref{Theorem: Characterization of indecomposablesArtin}. We define $\phi_Q(f_!(J(\mathtt{i})))=f_!(J'(\mathtt{i}))$, where $J'$ is the unique (up to isomorphism) $\Gamma$-module satisfying $[J']=\phi([J])$. Clearly this gives a bijection between the indecomposable injective objects in $\mono_Q(\Lambda)$ and $\mono_Q(\Gamma)$, and hence between all indecomposable objects. Finally, by the formula \eqref{Formula:f_!} in Example \ref{Example: Phyla} we have
		\[
		f_!(J(\mathtt{i}))_{\mathtt{k}}= \bigoplus_{\substack{p\in Q_{\geq 0}\\ s(p)=\mathtt{i}, t(p)=\mathtt{k}}}J
		\]
		and so $[\phi_Q(f_!(J(\mathtt{i})))_{\mathtt{k}}]=\phi([f_!(J(\mathtt{i}))_{\mathtt{k}}])$ for all $\mathtt{k}\in Q_0$, which proves \eqref{Theorem:StableEquivBijection:2}.
	\end{proof}

	\begin{rmk}
		The assumptions in Theorem \ref{Theorem:StableEquivBijection} are quite restrictive: Assume $\mathbbm{k}$ is perfect field and $\Lambda$ and $\Gamma$ are non-semisimple connected and selfinjective $\mathbbm{k}$-algebras. Let $\Phi\colon \overline{\operatorname{mod}}\, \Lambda\xrightarrow{\cong} \overline{\operatorname{mod}}\, \Gamma$ be an equivalence, and assume it restricts to an equivalence between the simple non-injective $\Lambda$- and $\Gamma$-modules. We claim that if $\Phi$ is induced from an exact functor $F\colon \operatorname{mod}\Lambda\to \operatorname{mod}\Gamma$ which preserves projectives, then $\Lambda$ and $\Gamma$ must be Morita equivalent. 
		
		Indeed, since $F$ is right exact, $F\cong -\otimes_{\Lambda}M$ where $M\coloneqq F(\Lambda)$ is a $\Lambda$-$\Gamma$-bimodule. Since $F$ is exact and preserves projectives, $M$ must be projective both as a left $\Lambda$-module and as a right $\Gamma$-module. By the same argument as for \cite[Proposition 2.4]{Lin96} we have an isomorphism 
		\[
		M\cong M'\oplus M''
		\]
		 where $M'$ is indecomposable non-projective and $M''$ is projective as $\Lambda$-$\Gamma$-bimodules. Since $M''$ is projective as a bimodule, $-\otimes_\Lambda M''$ sends any $\Lambda$-module to a projective $\Gamma$-module, and therefore induces the zero functor on the stable categories. Hence $-\otimes_\Lambda M'$ induces the same functor as $-\otimes_{\Lambda}M$ on the stable categories, i.e. the functor $\Phi$. Now if $S$ is a simple $\Lambda$-module, then $S\otimes_\Lambda M'$ has no nonzero projective summands by \cite[Proposition 2.3]{Lin96}. Since $S\otimes_\Lambda M'$ is isomorphic to a simple $\Gamma$-module in the stable category $\overline{\operatorname{mod}}\, \Gamma$, it must itself be a simple $\Gamma$-module. Therefore, by the proof of \cite[Proposition 2.5]{Lin96} the functor $-\otimes_\Lambda M'\colon \operatorname{mod}\Lambda\to \operatorname{mod}\Gamma$ must be an equivalence. This shows that $\Lambda$ and $\Gamma$ are Morita equivalent.
		
		Similarly, assume we are given a derived equivalence $-\otimes^{\mathbb{L}}_\Lambda T\colon D^b(\Lambda)\xrightarrow{\cong} D^b(\Gamma)$ by a tilting complex $T$. Then by \cite[Corollary 5.5]{Ric91} the induced equivalence $\overline{\operatorname{mod}}\, \Lambda\xrightarrow{\cong}\overline{\operatorname{mod}}\, \Gamma$ between the stable categories is given by an exact functor $\operatorname{mod}\Lambda\to \operatorname{mod}\Gamma$. Hence, if the stable equivalence induced from the derived equivalence gives a bijection between the simple objects, then $\Lambda$ and $\Gamma$ must be Morita equivalent. 
		
		It follows from this observation that most of the interesting examples of stable equivalence in Theorem \ref{Derived equivalence} and Corollary \ref{selfinjective_stable_number} do not satisfy Theorem \ref{Theorem:StableEquivBijection}. Our main example that satisfy Theorem \ref{Theorem:StableEquivBijection} are local uniserial algebras of Loewy length $3$, which we discuss in the next subsection. 
	\end{rmk}

	\subsection{Local uniserial rings of Loewy length $3$}\label{Subsection:Local uniserial rings of Loewy length 3}
	
	Let $\Lambda$  and $\Gamma$ be commutative local rings which are uniserial, i.e. they have a unique compositition series. Assume furthermore that they have Loewy length $3$ and the same residue field. In particular, $\Lambda$ and $\Gamma$ are commutative artinian rings. There are three indecomposable $\Lambda$- and $\Gamma$-modules, and they are uniquely determined by their length. This can be seen for example by using \cite[Theorem VI.2.1 a)]{ARS95}. We denote the indecomposable $\Lambda$- and $\Gamma$-modules by $M_1,M_2,M_3$ and $N_1,N_2,N_3$, respectively, so that $M_i$ and $N_i$ have length $i$. 
	
	Our goal is to construct a bijection between the indecomposable objects in $\mono_Q(\Lambda)$ and $\mono_Q(\Gamma)$ as in Theorem \ref{Theorem:StableEquivBijection}. This is particularly useful when $\Lambda=\mathbb{Z}/(p^3)$ and  $\Gamma=\mathbbm{k}[x]/(x^3)$ where  $\mathbbm{k}=\mathbb{Z}/(p)$, since the indecomposables in $\mono_Q(\Gamma)$ are in general easier to compute than the ones in $\mono_Q(\Lambda)$. For example, for $\mono_Q(\Gamma)$ one can use covering theory (e.g. see \cite{Moo09} and \cite{RS08b}). We are not aware of such methods for $\mono_Q(\Lambda)$.

	\begin{prop}\label{Proposition:StableEquivLocalUniserial}
		Let $\Lambda$ and $\Gamma$ be commutative local uniserial rings of Loewy length smaller than or equal to $3$ with residue field $\mathbbm{k}$. Then there exists an equivalence $\overline{\modu}\, \Lambda\cong \overline{\modu}\, \Gamma$ which preserves the simple object. 
	\end{prop}
	
	\begin{proof}
		We only prove the case of Loewy length $3$. The cases of Loewy length $2$ and $1$ follow similarly. Note first that $\Lambda$ and $\Gamma$ are selfinjective with a unique injective module given by $M_3$ and $N_3$, respectively. Therefore the indecomposables in $\overline{\operatorname{mod}}\, \Lambda$ and $\overline{\operatorname{mod}}\, \Gamma$ are $M_1,M_2$, and $N_1,N_2$, respectively. Furthermore, their hom-spaces are
		\begin{equation}\label{Equation:NatIsoSquare}
		\overline{\operatorname{Hom}}_{\Lambda}(M_i,M_j)\cong \mathbbm{k} \quad \text{and} \quad \overline{\operatorname{Hom}}_{\Gamma}(N_i,N_j)\cong \mathbbm{k}
		\end{equation}
		for $1\leq i,j\leq 2$. Let $f_{i,j}$ and $g_{i,j}$ be the basis vector of $\overline{\operatorname{Hom}}_{\Lambda}(M_i,M_j)$ and $\overline{\operatorname{Hom}}_{\Gamma}(N_i,N_j)$, respectively, so that $f_{i,i}=\operatorname{id}_{M_i}$ and $g_{i,i}=\operatorname{id}_{N_i}$ for $i=1,2$. Then we have the relations
		\begin{align*}
		&f_{2,1}\circ f_{1,2}=0 \quad \text{and} \quad f_{1,2}\circ f_{2,1}=0 \\
		& g_{2,1}\circ g_{1,2}=0 \quad \text{and} \quad g_{1,2}\circ g_{2,1}=0
		\end{align*}
		and hence the associations $M_i\mapsto N_i$ and $f_{i,j}\mapsto g_{i,j}$ extend to an equivalence $\overline{\operatorname{mod}}\, \Lambda\xrightarrow{\cong} \overline{\operatorname{mod}}\, \Gamma$. Since this equivalence preserves the simple object, we are done.
	\end{proof}
	
	\begin{rmk}
		Proposition \ref{Proposition:StableEquivLocalUniserial} does not hold when the Loewy length is greater than $3$. For example, consider $\Lambda=\mathbb{Z}/(p^n)$ and $\Gamma=\Bbbk[x]/(x^n)$ with $\Bbbk=\mathbb{Z}/(p)$ and $n\geq 4$. Then 
		\[
		\overline{\Hom}_\Lambda(\mathbb{Z}/(p^2), \mathbb{Z}/(p^2))\cong \mathbb{Z}/(p^2)
		\]
		and there is no object with that endomorphism ring in $\overline{\modu}\,\Gamma$. 
	\end{rmk}
	
	Let $\Lambda$ be a local uniserial ring of Loewy length $3$. Then there is a bijection between finitely generated $\Lambda$-modules and partitions $(\alpha_1\geq \alpha_2\geq \dots \geq \alpha_s)$ with $3\geq \alpha_1$. Explicitly, it sends the partition $\alpha=(\alpha_1\geq \alpha_2\geq \dots \geq \alpha_s)$ to the module
	\[
	M(\alpha)=\bigoplus_{i=1}^sM_{\alpha_i}
	\]
	where $M_{\alpha_i}$ is the indecomposable $\Lambda$-module of Loewy length $\alpha_i$.
	Given a representation $\module{M}\in \operatorname{rep}(Q,\operatorname{mod}\Lambda)$, the \emphbf{partition vector} of $\module{M}$ is the tuple $(\alpha^{\mathtt{i}})_{\mathtt{i}\in Q_0}$ where $\alpha^\mathtt{i}$ is the unique partition for which $M(\alpha^i)\cong\module{M}_\mathtt{i}$. Note that this is called the type of $\module{M}$ in \cite{Sch08}. 
	
	\begin{thm}\label{Theorem:BijectionLoewyLength3}
		Let $\Lambda$ and $\Gamma$ be commutative local uniserial rings of Loewy length smaller than or equal to $3$ with same residue field $\mathbbm{k}$. Then there exists a bijection which preserves partition vectors between indecomposable objects in $\mono_Q(\Lambda)$ and in $\mono_Q(\Gamma)$.
	\end{thm}
	
	\begin{proof}
		This follows from Proposition \ref{Proposition:StableEquivLocalUniserial} and Theorem  \ref{Theorem:StableEquivBijection}.
	\end{proof}
		
	The \emphbf{length vector} of a representation $\module{M}\in \operatorname{rep}(Q,\Lambda)$, is defined to be the tuple $\ell(\module{M})=(\ell(\module{M}_\mathtt{i}))_{\mathtt{i}\in Q_0}$ where $\ell(\module{M}_{\mathtt{i}})$ denotes the length of the $\Lambda$-module $\module{M}_{\mathtt{i}}$. 
	
	\begin{ex}\label{Example:A_3LoewyLength3}
		Let $\mathbb{F}_p=\mathbb{Z}/(p)$ be the finite field with $p$ elements. We compare $\mono_Q(\mathbb{F}_p[x]/(x^3))$ and   $\mono_Q(\mathbb{Z}/(p^3))$ for different choices of quivers $Q$, using Theorem \ref{Theorem:BijectionLoewyLength3}. 
		
		For $Q=\mathtt{1}\to \mathtt{2}$ the indecomposable objects in $\mono_Q(\mathbb{F}_p[x]/(x^3))$ have been classified in \cite[Section 6.3]{RS08b}. They are the unique representations of the form 
		\[
		\mathbb{F}_p[x]/(x^i)\xrightarrow{f_{j,i}} \mathbb{F}_p[x]/(x^j)
		\]
		where $f_{j,i}$ is a monomorphism and $0\leq i\leq j\leq 3$ with $j\neq 0$, and
		\[
		\mathbb{F}_p[x]/(x^2)\xrightarrow{\begin{pmatrix}\pi\\\iota\end{pmatrix}}\mathbb{F}_p\oplus \mathbb{F}_p[x](x^3)
		\]
		where $\pi$ is the canonical projection and $\iota$ is the canonical inclusion. It follows that the indecomposable objects in $\mono_Q(\mathbb{Z}/(p^3))$ are the representations of the form 
		\[
		\mathbb{Z}/(p^i)\xrightarrow{g_{j,i}} \mathbb{Z}/(p^j)
		\]
		where $g_{j,i}$ is a monomorphism and $0\leq i\leq j\leq 3$ with $j\neq 0$, and
		\[
		\mathbb{Z}/(p^2)\xrightarrow{\begin{pmatrix}\pi'\\\iota'\end{pmatrix}}\mathbb{Z}/(p)\oplus \mathbb{Z}/(p^3).
		\]
		where $\pi'$ and $\iota'$ are the canonical projection and inclusion, respectively. Now \cite[Theorem 1.2]{RRW84} implies that the indecomposable objects $M_1\to M_2$ in $\mono_Q(\mathbb{Z}/(p^n))$ where $M_1\neq 0$ are in bijection with the indecomposable finitely generated valuated $p$-groups (in the sense of \cite{RW79}) whose value is bounded by $n$. Using this and the fact that the group $M_1$ at vertex $\mathtt{1}$ in the classification above is always cyclic, we recover \cite[Corollary 4.3]{RRW84}.
		
		For $Q=\mathtt{1}\to \mathtt{2}\to \mathtt{3}$, the isomorphism classes of  indecomposable objects in $\mono_Q(\mathbb{F}_p[x]/(x^3))$ are uniquely determined by their dimension vector \cite[Theorem 3.2]{Moo09}. There are $23$ indecomposables up to isomorphism, and the different dimension vectors that occur are
		\begin{multline*}
		\{001, 002, 011, 012, 111, 112, 122, 222, 003, 013, 023 \\
		113, 123, 223, 333, 024, 124, 224, 234, 244, 135, 245, 246 \}.
		\end{multline*}
		Since the bijection in  Theorem \ref{Theorem:BijectionLoewyLength3}  preserves partition vectors, it also preserves length vectors. So we can conclude that the indecomposable objects in $\mono_Q(\mathbb{Z}/(p^3))$ are uniquely determined by their length vectors, and the different length vectors that occur are given in the list above. The list of partition vectors for the indecomposable objects in $\mono_Q(\mathbb{F}_p[x]/(x^3))$ (and hence in $\mono_Q(\mathbb{Z}/(p^3))$),  are given in Figure 1 in Section 6 of \cite{XZZ14}, see also \cite[Example 5.1]{Lu20}.
		
		For $Q_1=\mathtt{1}\to \mathtt{2}\leftarrow \mathtt{3}$ and $Q_2=\mathtt{1}\leftarrow \mathtt{2}\to \mathtt{3}$ we have equivalences
		\[
		\overline{\mono}_{Q_1}(\mathbb{F}_p[x]/(x^3))\cong \overline{\mono}_{Q}(\mathbb{F}_p[x]/(x^3))\cong \overline{\mono}_{Q_2}(\mathbb{F}_p[x]/(x^3))
		\]
		by \cite[Theorem 1]{LS22}, where $Q$ is the quiver $\mathtt{1}\to \mathtt{2}\to \mathtt{3}$ considered in the previous paragraph. In particular, $\mono_{Q_1}(\mathbb{F}_p[x]/(x^3))$ and $\mono_{Q}(\mathbb{F}_p[x]/(x^3))$ and $\mono_{Q_2}(\mathbb{F}_p[x]/(x^3))$ have the same number of indecomposable objects. The indecomposables in $\mono_{Q_1}(\mathbb{F}_p[x]/(x^3))$ and in $\mono_{Q_2}(\mathbb{F}_p[x]/(x^3))$ are classified in \cite[Figure 7]{LS22} and \cite[Figure 8]{LS22}, respectively, in terms of their partition vector.
		Using this, we get the indecomposables for  $\mono_{Q_1}(\mathbb{Z}/(p^3))$ and $\mono_{Q_2}(\mathbb{Z}/(p^3))$.

		For $Q=\mathtt{1}\to \mathtt{2}\to \mathtt{3}\to \mathtt{4}$, there are 84 isomorphism classes of indecomposable objects in $\mono_Q(\mathbb{F}_p[x]/(x^3))$, and a list of them is given in Figure 2 in Section 6 of \cite{XZZ14}, see also \cite[Example 5.2]{Lu20}. The indecomposables are described in terms of their restriction to each vertex, which is equivalent to giving their partition vector. Hence, we get a list of the indecomposable objects in $\mono_Q(\mathbb{Z}/(p^3))$ in terms of their partition vector. 
		
		If $Q=\mathtt{1}\to \mathtt{2}\to \mathtt{3}\to \mathtt{4}\to \mathtt{5}$, then it follows from \cite[Theorem 1.3]{Sim02} that the categories $\mono_Q(\mathbb{F}_p[x]/(x^3))$ and $\mono_Q(\mathbb{Z}/(p^3))$ have infinitely many indecomposable objects. Furthermore, if $\mathbbm{k}$ is an algebraically closed field, then $\mono_Q(\mathbbm{k}[x]/(x^3))$ is of tame representation type by \cite[Theorem 1.5]{Sim02}. Hence, one could try to determine the indecomposable objects in $\mono_Q(\mathbb{F}_p[x]/(x^3))$ using a similar approach as in \cite{Moo09} and \cite{RS08b}. Then one could use Theorem \ref{Theorem:BijectionLoewyLength3} to transfer the results to determine the indecomposables in $\mono_Q(\mathbb{Z}/(p^3))$.
	\end{ex}

	\subsection{The monomorphism category for  $\operatorname{rad}^2$-zero Nakayama algebras}\label{Subsection:Rad^2-zeroNakayama}
	
	Recall that an Artin algebra $\Lambda$ is a \textbf{Nakayama algebra} if all indecomposable left and right $\Lambda$-modules are uniserial, i.e. have a unique composition series. In this section, we will describe the indecomposable objects of $\mono_Q(\Lambda)$ when $\Lambda$ is a radical square zero Nakayama algebra. We start with the following well-known lemma, a proof of which we provide for convenience. 
	
	\begin{lem}\label{lem:injectively-stable-rad2-Nakayama}
		Let $\Lambda$ be a radical square zero Nakayama algebra. Then $\overline{\modu}\, \Lambda$ is semisimple. More precisely, it is equivalent to the module category of a product of skew fields  -- as many as there are isomorphism classes of non-injective simple modules.
	\end{lem}
	
	\begin{proof}
		Since $\Lambda$ is a Nakayama algebra, it is well-known, see e.g. \cite[Theorem VI.2.1]{ARS95}, that every indecomposable $\Lambda$-module $M$ is of the form $P/\radoperator^m P$ for an indecomposable projective module $P$. Since $\radoperator^2 \Lambda=0$, it follows that $m\leq 1$. Suppose that $M$ is not simple. Then $m=0$ and $M\cong P$ for some indecomposable projective. However, in this case \cite[Lemma IV.2.15]{ARS95} implies that $M$ is also injective. Therefore, $M=0$ in $\overline{\modu}\, \Lambda$. It follows that every object in $\overline{\modu}\, \Lambda$ is isomorphic to a direct sum of simple modules. Thus, $\overline{\modu}\, \Lambda$ is a semisimple category. The claim follows.
	\end{proof}
	We need the following variant of Gabriel's theorem for path algebras over skew fields, which is a special case of the main result of \cite{DR76}. 
	
	\begin{thm}\label{Theorem:Gabriel-for-skew-fields}
		Let $D$ be a skew field and let $Q$ be a finite quiver. Then $\modu DQ$ is of finite representation type if and only if $Q$ is Dynkin. In this case, there is a bijection between the set of isomorphism classes of indecomposable $DQ$-modules and the set of positive roots $\Phi^+$ for the corresponding Dynkin diagram.
	\end{thm}
	
	Combining these we obtain the classification of indecomposables in $\mono_Q(\Lambda)$. If $\Lambda=\mathbbm{k}[x]/(x^2)$ with $k$ algebraically closed, then the first part is equivalent to \cite[Theorem 4.6 (i)]{Lu20}.
	
	\begin{thm}\label{Theorem:rad2Nakayama}
		Let $Q$ be a finite connected acyclic quiver and let $\Lambda$ be a non-semisimple radical square zero Nakayama algebra. Let $m$ be the number of simple $\Lambda$-modules and let $t$ be the number of non-injective simple $\Lambda$-modules. The following hold.
		\begin{enumerate}
			\item The category $\mono_Q(\Lambda)$ is of finite representation type if and only if $Q$ is Dynkin. 
			
			\item The number of indecomposable injective objects in $\mono_Q(\Lambda)$ is $m \cdot |Q_0|$. 
			
			\item The number of indecomposable non-injective objects is $t \cdot |\Phi^+|$, where $\Phi^+$ is the set of positive roots of the (corresponding) Dynkin diagram. 
		\end{enumerate}
	\end{thm}

	\begin{proof}
		Let $\mathbb{L}$ be the endomorphism ring of the direct sum of the simple non-injective $\Lambda$-modules (one for each isomorphism class). Combining  Theorem \ref{Theorem: Characterization of indecomposables} with Lemma \ref{lem:injectively-stable-rad2-Nakayama}, there is a bijection between the indecomposable non-injective objects in $\mono_Q(\Lambda)$ and the indecomposable objects in $\modu \mathbb{L}Q$. As there are only finitely many indecomposable injective objects in $\mono_Q(\Lambda)$, $\mono_Q(\Lambda)$ is of finite representation type if and only if $\modu \mathbb{L}Q$ is of finite representation type. 
		
		In this case, we count indecomposables.  There are exactly $m\cdot |Q_0|$ indecomposable injective modules in $\mono_Q(\Lambda)$, as each of them is of the form $f_!(J)$ for some indecomposable injective module in $\modu \Lambda Q_0$. According to Theorem  \ref{Theorem:Gabriel-for-skew-fields}, there are exactly $t\cdot |\Phi^+|$ indecomposable non-injective modules in $\mono_Q(\Lambda)$, so the claim follows. 
	\end{proof}
	
	\begin{rmk}\label{Remark:NumberIndecA_n}
		It is well-known that the number of positive roots of $\mathbb{A}_n$ is $\binom{n+1}{2}$. By Theorem \ref{Theorem:rad2Nakayama} this is also the number of indecomposable non-injective object in $\Mono_{Q}(\mathbbm{k}[x]/(x^2))$ when the underlying diagram of $Q$ is $\mathbb{A}_n$. This recovers an observation in \cite[6.1 (i)]{XZZ14}.
	\end{rmk}
	
	\begin{rmk}
		An obvious question given the preceding theorem is whether this is something special about $\mono_Q(\Lambda)$ or whether the same result holds even for $\modu \Lambda Q$. This is however not the case. For this remark we restrict to  a finite-dimensional algebra $\Lambda$ over an algebraically closed field $\Bbbk$. Furthermore assume that $\Lambda$ and $Q$ are connected. There are three cases to distinguish: $\Lambda\cong \Bbbk$, $\Lambda\cong \left(\begin{smallmatrix}\Bbbk&\Bbbk\\0&\Bbbk\end{smallmatrix}\right)$, and the rest. 
		
		If $\Lambda\cong \Bbbk$, then indeed Gabriel's theorem states that $\Lambda Q$ is of finite representation type if and only if $Q$ is Dynkin. 
		
		If $\Lambda\cong \left(\begin{smallmatrix}\Bbbk&\Bbbk\\0&\Bbbk\end{smallmatrix}\right)$, then $\Lambda Q\cong T_2(\Bbbk Q)$. The representation type of $T_2(A)$ for a finite-dimensional algebra $A$ was studied in \cite{LS00} and from Theorem 4 therein it follows that $\Lambda Q$ is of finite representation type if and only if $Q$ is of type $A_n$ for $n\leq 4$, with arbitrary orientation. For $A_2$ or non-linearly oriented $A_3$, a proof of representation-finiteness can already be found in \cite[Proposition 1.2, Theorem 1.3]{AR75/76}, for linearly oriented $A_3$, see \cite[Theorem 2.4(a)]{Les94}.  
		
		Lastly, if $\Lambda$ is any other radical square zero Nakayama algebra, then $\Lambda Q$ is of finite representation type if and only if $Q$ is of type $A_n$ for $n\leq 3$. For $\Lambda\cong \Bbbk[x]/(x^2)$, this follows from \cite[Proposition 13.1]{GLS16}. For general $\Lambda$ (linear over an algebraically closed field) an analogous proof using covering theory and the knitting algorithm yields the result. (The cases of $Q=A_2$ or linearly oriented $A_3$ can also  be found in \cite[Proposition 1.2]{AR75/76} and \cite[Theorem 2.4(b)]{Les94}. In case $Q=A_4$, one can use the Happel--Vossieck list \cite{HV83} to find a subquiver with relations of the form 
		
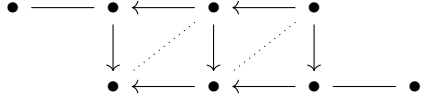
\[\begin{tikzcd}\bullet\arrow[no head]{r}&\bullet\arrow{d}&\bullet\arrow{l}\arrow[no head, dotted]{ld}\arrow{d}&\bullet\arrow[no head, dotted]{ld}\arrow{l}\arrow{d}\\
		&\bullet&\bullet\arrow{l}&\bullet\arrow{l}&\bullet\arrow[no head]{l}\end{tikzcd}\]
		to conclude that the algebra is representation-infinite.) 
	\end{rmk}
	
	\begin{ex}
		Let $\Bbbk$ be a field and let $\Lambda=\Bbbk[x]/(x^2)$. Let $Q$ be a linearly oriented quiver of type $\mathbb{A}_n$. In this case, there is only one simple $\Lambda$-module which is non-injective. According to Theorem \ref{Theorem: Characterization of indecomposablesArtin}, there are $n$ indecomposable injective modules given by $f_!(\Lambda(\mathtt{i}))$ where $\Lambda(\mathtt{i})$ denotes the representation of $\mathbb{A}_n$ given by $\Lambda$ at vertex $\mathtt{i}$, and zero elsewhere. As a representation, these look as follows:
		\[0\to \dots\to 0\to \Lambda\stackrel{\id}{\to} \Lambda\stackrel{\id}{\to} \dots\stackrel{\id}{\to} \Lambda\]
		On the other hand, the indecomposable non-injectives are given by applying the $\Mimo$-construction to objects in $\modu \Bbbk \mathbb{A}_n$, since $\overline{\modu}\, \Lambda\cong \modu \Bbbk$. The indecomposable objects of $\modu \Bbbk \mathbb{A}_n$ are given by the interval modules 
		\[0\to \dots \to 0\to \Bbbk\stackrel{\id}{\to} \dots \stackrel{\id}{\to} \Bbbk\to 0\to \dots\to 0.\]
		where the the last $\Bbbk$ is in position $\mathtt{j}$. It is easy to check that for such an interval module $M$ we have $\ker M_{\mathtt{k},\arrowin}=0$ unless $\mathtt{k}=\mathtt{j}+1$, in which case $\ker M_{\mathtt{k}, \arrowin}=\Bbbk$, which has injective envelope $\Lambda$. Therefore, following the explicit description of the $\Mimo$-construction in Example \ref{Example:MimoFormulaQuiverRep}, we obtain that the indecomposable non-injective objects in $\Mono_{\mathbb{A}_n}(\Lambda)$ are given by 
		\[0\to \dots\to 0\to \Bbbk\stackrel{\id}{\to} \dots\stackrel{\id}{\to} \Bbbk\stackrel{\iota}{\to} \Lambda\stackrel{\id}{\to}\dots\stackrel{\id}{\to} \Lambda,\]
		where $\iota$ is a chosen embedding of $\Bbbk$ into $\Lambda$. This recovers the description of the indecomposables in \cite[Theorem 3.1]{Moo09} for $n=3$. For arbitrary $n$, counting the number of indecomposable non-injective, we see that it is equal to $\binom{n+1}{2}$ as noted in Remark \ref{Remark:NumberIndecA_n}.
		
		It follows from Theorem \ref{Theorem:BijectionLoewyLength3} that $\mono_Q(\Bbbk[x]/(x^2))$ and $\mono_Q(\mathbb{Z}/(p^2))$ have the same number of indecomposable objects of the same form. In fact, this could also be seen by following the same steps as above when $\Bbbk=\mathbb{Z}/(p)$ and $\Lambda=\mathbb{Z}/(p^2)$. 
		
		More generally, the same arguments yield that for $\Lambda$ an arbitrary radical square zero Nakayama algebra and $Q$ a linearly oriented type $\mathbb{A}_n$ quiver, the indecomposable injective objects of $\mono_Q(\Lambda)$ are given by
		\[0\to \dots\to 0\to I\stackrel{\id}{\to} \dots\stackrel{\id}{\to} I,\]
		where $I$ runs through the indecomposable injective $\Lambda$-modules, while the indecomposable non-injective objects of $\mono_Q(\Lambda)$ are given by 
		\[0\to \dots\to 0\to L\stackrel{\id}{\to} \dots\stackrel{\id}{\to} L\stackrel{\iota}{\to} I(L)\stackrel{\id}{\to} \dots\stackrel{\id}{\to} I(L)\]
		where $L$ runs through the simple non-injective $\Lambda$-modules up to isomorphism, $I(L)$ denotes an injective envelope of $L$, with $\iota\colon L\to I(L)$ a chosen embedding. If we set $n=2$ or $n=3$ and let $\Lambda$ be the algebra $\begin{tikzcd}1\arrow[r,yshift=0.1cm,"\alpha"above]&2\arrow[l,yshift=-0.1cm,"\beta" below]\end{tikzcd}$ with relations $\alpha\circ \beta=0=\beta\circ \alpha$, then this recovers the description of the indecomposable objects obtained in \cite[Section 6.2]{XZZ14}. 
	\end{ex}
	
	\begin{ex}
		Let $Q$ be the quiver $\mathtt{1}\to \mathtt{2}\leftarrow \mathtt{3}\to \mathtt{4}$. Let $\Lambda=\Bbbk[x]/(x^2)$. Similarly to the preceding example, the 4 indecomposable injective objects are given by 
		\begin{align*}f_!(\Lambda(\mathtt{1}))&=(\Lambda\stackrel{1}{\to} \Lambda\leftarrow 0\to 0),\qquad 
		f_!(\Lambda(\mathtt{2}))=(0\to \Lambda\leftarrow 0\to 0),\\ f_!(\Lambda(\mathtt{3}))&=(0\to \Lambda\stackrel{1}{\leftarrow} \Lambda\stackrel{1}{\to} \Lambda),\qquad
		f_!(\Lambda(\mathtt{4}))=(0\to 0\leftarrow 0\to \Lambda).
		\end{align*}
		On the other hand, the indecomposable non-injective objects are given by applying the Mimo-construction to the following indecomposable objects in $\modu \Bbbk Q$, since $\overline{\modu}\, \Lambda\cong \modu \Bbbk$:
		\begin{align*}
		&0 \to \Bbbk \leftarrow 0 \to 0,\qquad 0\to 0\leftarrow 0\to \Bbbk,\qquad    
		\Bbbk\stackrel{1}{\to} \Bbbk \leftarrow 0 \to 0,\qquad 0\to \Bbbk \stackrel{1}{\leftarrow} \Bbbk\stackrel{1}{\to} \Bbbk,\\
		&\Bbbk\stackrel{1}{\to} \Bbbk \stackrel{1}{\leftarrow} \Bbbk \stackrel{1}{\to} \Bbbk,\qquad 0\to \Bbbk\stackrel{1}{\leftarrow} \Bbbk \to 0,\qquad
		0\to 0 \leftarrow \Bbbk\stackrel{1}{\to} \Bbbk,\\ 
		&\Bbbk\stackrel{1}{\to} \Bbbk \stackrel{1}{\leftarrow} \Bbbk\to 0,\qquad
		0\to 0 \leftarrow \Bbbk\to 0,\qquad \Bbbk\to 0\leftarrow 0\to 0.
		\end{align*}
		This yields the following indecomposable objects in $\mono_Q(\Lambda)$:
		\begin{align*}
		&0\to \Bbbk\leftarrow 0\to 0, \qquad 0\to 0\leftarrow 0\to \Bbbk,\qquad
		\Bbbk\stackrel{1}{\to} \Bbbk \leftarrow 0\to 0, \qquad 0\to \Bbbk\stackrel{1}{\leftarrow} \Bbbk \stackrel{1}{\to}\Bbbk,\\
		&\Bbbk\stackrel{\left(\begin{smallmatrix}1\\\iota\end{smallmatrix}\right)}{\to} \Bbbk\oplus \Lambda \stackrel{\left(\begin{smallmatrix}1\\0\end{smallmatrix}\right)}{\leftarrow} \Bbbk\stackrel{1}{\to} \Bbbk, \qquad 0\to \Bbbk\stackrel{1}{\leftarrow} \Bbbk\stackrel{\iota}{\to} \Lambda,\qquad 0\to \Lambda \stackrel{\iota}{\leftarrow} \Bbbk\to \Bbbk,\\ &\Bbbk\stackrel{\left(\begin{smallmatrix}1\\\iota\end{smallmatrix}\right)}{\to} \Bbbk\oplus \Lambda \stackrel{\left(\begin{smallmatrix}1\\0\end{smallmatrix}\right)}{\leftarrow} \Bbbk\stackrel{\iota}{\to} \Lambda,\qquad
		0\to \Lambda \stackrel{\iota}{\leftarrow} \Bbbk \stackrel{\iota}{\to}\Lambda, \qquad \Bbbk\stackrel{\iota}{\to} \Lambda  \leftarrow 0\to 0.
		\end{align*}
		From this and the previous example we see that the dimension vectors of $\mono_Q(\Lambda)$ depend on the orientation of $Q$. Replacing $\Lambda$ by a general radical square zero Nakayama algebra is done similarly as in the preceding example. 
	\end{ex}

	\subsection{The Kronecker quiver}\label{Subsection:Kronecker}
	Assume $\mathbbm{k}$ is algebraically closed. Let 
	$Q=\begin{tikzcd}1\arrow[r,yshift=0.1cm,""above]\arrow[r,yshift=-0.1cm,"" below]&2\end{tikzcd}$ be the Kronecker quiver and let $\Lambda=\mathbbm{k}[x]/(x^2)$. An object of $\mono_Q(\Lambda)$ consists of a triple $(U,V,T)$ where $V$ is a finite-dimensional $\mathbbm{k}$-vector space, $U\oplus U$ is a $\mathbbm{k}$-subspace of $V$, and $T$ is a linear operator on $V$ satisfying $T^2=0$, and which restricts to a linear operator of the form
	\[
	\begin{pmatrix}
	T'&0\\
	0&T'
	\end{pmatrix}\colon U\oplus U\to U\oplus U
	\]
	on $U\oplus U$. We use  Theorem \ref{Theorem: Characterization of indecomposablesArtin} to describe the indecomposables in $\mono_Q(\Lambda)$. Indeed, the indecomposable injective objects are given by
	\[
	f_!(\Lambda(\mathtt{1}))= \begin{tikzcd}\Lambda\arrow[r,yshift=0.1cm,"i_1"above]\arrow[r,yshift=-0.1cm,"i_2"below]&\Lambda^2\end{tikzcd} \quad \text{and} \quad f_!(\Lambda(\mathtt{2}))= \begin{tikzcd}0\arrow[r,yshift=0.1cm,""above]\arrow[r,yshift=-0.1cm,"" below]&\Lambda\end{tikzcd}
	\]
 where $i_1$ and $i_2$ are the inclusions of the first and second summand of $\Lambda^2$, respectively.
	By Theorem \ref{Theorem: Characterization of indecomposablesArtin} we can obtain the indecomposable non-injective objects in $\mono_Q(\Lambda)$ from the indecomposable objects in $\rep (Q,\overline{\operatorname{mod}}\,\Lambda)\cong \rep (Q,\operatorname{mod}\mathbbm{k})$. The latter is just the category of representations of the Kronecker quiver over $\mathbbm{k}$. This has tame representation type and its indecomposable finite-dimensional representations are well-known, e.g. see \cite[Section VIII.7]{ARS95}. To describe them we let $V_n$ denote the vector space of homogenous polynomials of degree $n$ in variables $y$ and $z$, and we let $V_n^*$ denote its $\mathbbm{k}$-dual. Then the indecomposables of $\rep (Q,\operatorname{mod}\mathbbm{k})$ are the preprojective and preinjective representations
	\[
	P_n=\begin{tikzcd}V_{n-1}\arrow[r,yshift=0.1cm,"y "above]\arrow[r,yshift=-0.1cm,"z" below]&V_n\end{tikzcd} \quad \text{and} \quad I_n=\begin{tikzcd}V_n^*\arrow[r,yshift=0.1cm,"(y)^*"above]\arrow[r,yshift=-0.1cm,"(z)^*" below]&V_{n-1}^*\end{tikzcd} \quad  \quad n\geq 0
	\]
	and the regular representations 
	\[
	R_{p^n}=\begin{tikzcd}V_{n-1}\arrow[r,yshift=0.1cm,"y"above]\arrow[r,yshift=-0.1cm,"z" below]&V_n/\mathbbm{k}p^n\end{tikzcd} \quad \text{where } n\geq 1 \text{ and } 0\neq p\in V_1.
	\]
	Here, $R_{p^n}\cong R_{q^n}$ if and only if $p=ay+bz$ and $q=cy+dz$ where $(a\colon b)=(c\colon d)$ in $\mathbb{P}^1$. Hence, the regular representations are indexed by the projective line. Since there are exact sequences
	\begin{equation}\label{Equation:Homogenous}
	0\to V_{n-2} \xrightarrow{\begin{pmatrix}
		-z\\y\end{pmatrix}}V_{n-1}\oplus V_{n-1}\xrightarrow{\begin{pmatrix}
		y&z
		\end{pmatrix}}V_n\to 0 \quad n\geq 1
	\end{equation}
	\begin{equation}\label{Equation:HomogenousDual}
	0\to V^*_{n+1} \xrightarrow{\begin{pmatrix}
		-(z)^*\\(y)^*\end{pmatrix}}V^*_{n}\oplus V^*_{n}\xrightarrow{\begin{pmatrix}
		(y)^*&(z)^*
		\end{pmatrix}}V^*_{n-1}\to 0 \quad n\geq 1
	\end{equation}
	\begin{equation}\label{Equation:HomogeneousQuotient}
	0\to V_{n-2}\oplus \mathbbm{k}p^{n-1} \xrightarrow{\begin{pmatrix}
		-z &a\\y &b\end{pmatrix}}V_{n-1}\oplus V_{n-1}\xrightarrow{\begin{pmatrix}
		y&z
		\end{pmatrix}}V_n/\mathbbm{k}p^n\to 0 \quad n\geq 1
	\end{equation}
	we get
	\[
	L_1\Kopf_X(P_n)=(0,V_{n-2}) \quad  \quad L_1\Kopf_X(I_n)=(0,V_{n+1}^*) \quad \quad L_1\Kopf_X(R_{p^n})=(0,V_{n-2}\oplus \mathbbm{k}p^{n-1}).
	\]
	Using the formula in Example \ref{Example:MimoFormulaQuiverRep} we can calculate the Mimo of $P_n$, $I_n$ and $R_{p^n}$. For simplicity we write $V[x]/(x^2)=V\otimes_{\mathbbm{k}} \mathbbm{k}[x]/(x^2)$ for a $\mathbbm{k}$-vector space $V$. Then
	\[
	\Mimo P_n\cong \begin{tikzcd}(V_{n-1}\arrow[r,yshift=0.1cm,"g_1"above]\arrow[r,yshift=-0.1cm,"g_2" below]&V_n\oplus V_{n-2}[x]/(x^2))\end{tikzcd} \quad  \text{with} \quad g_1=\begin{pmatrix}
	y \\ g_1'
	\end{pmatrix} \quad \text{and} \quad g_2=\begin{pmatrix}
	z \\ g_2'
	\end{pmatrix}
	\]
	where $V_{n-1}$ and $V_n$ are considered as $\mathbbm{k}[x]/(x^2)$-modules with trivial action by $x$, and where $g_1'=xg_1''$ and $g_2'=xg_2''$ and $\begin{pmatrix}
	g_1''&g_2''
	\end{pmatrix}\colon V_{n-1}\oplus V_{n-1}\to V_{n-2}$ is a choice of a retraction of the leftmost map in the exact sequence \eqref{Equation:Homogenous}.
	Similarly,
	\[
	\Mimo I_n\cong \begin{tikzcd}(V^*_{n}\arrow[r,yshift=0.1cm,"h_1"above]\arrow[r,yshift=-0.1cm,"h_2" below]&V^*_{n-1}\oplus V^*_{n+1}[x]/(x^2))\end{tikzcd} \quad  \text{with} \quad h_1=\begin{pmatrix}
	(y)^* \\ h_1'
	\end{pmatrix} \quad \text{and} \quad h_2=\begin{pmatrix}
	(z)^* \\ h_2'
	\end{pmatrix}
	\]
	where $h_1'=x h_1''$ and $h_2'=x h_2''$ and $\begin{pmatrix}
	h_1''&h_2''
	\end{pmatrix}\colon V^*_{n}\oplus V^*_{n}\to V^*_{n+1}$ is a choice of a retraction to the leftmost map in the exact sequence \eqref{Equation:HomogenousDual}.
	Finally, $\Mimo R_{p^n}$ is given by 
	\[
	\Mimo R_{p^n}\cong \begin{tikzcd}(V_{n-1}\arrow[r,yshift=0.1cm,"k_1"above]\arrow[r,yshift=-0.1cm,"k_2" below]&V_n\oplus (V_{n-2}\oplus \mathbbm{k}p^{n-1})[x]/(x^2))\end{tikzcd} \quad \text{with} \quad k_1=\begin{pmatrix}
	y \\ k_1'
	\end{pmatrix} \quad \text{and} \quad k_2=\begin{pmatrix}
	z \\ k_2'
	\end{pmatrix}  
	\]
	where $k_1'=x k_1''$ and $k_2'=x k_2''$
	and $\begin{pmatrix}
	k_1''&k_2''
	\end{pmatrix}\colon V_{n-1}\oplus V_{n-1}\to V_{n-2}\oplus \mathbbm{k}p^{n-1}$ is a choice of a retraction to the leftmost map in the exact sequence \eqref{Equation:HomogeneousQuotient}.
	It follows from Theorem \ref{Theorem: Characterization of indecomposablesArtin} that 
	\[
	f_!(\Lambda(\mathtt{1})),\text{ }f_!(\Lambda(\mathtt{2})),\text{ }\Mimo P_n,\text{ } \Mimo I_n,\text{ and }\Mimo R_{p^n}
	\]
	for all possible $n$ and $p$ are up to isomorphism all the indecomposable objects in $\mono_Q(\Lambda)$.
	
	We can also consider the category $\Mono_Q(\Lambda)$ of all monic representations, consisting of triples $(U,V,T)$ as above, but where $U$ and $V$ are not necessarily finite-dimensional. By Theorem \ref{Theorem: Characterization of indecomposablesArtin} the indecomposable non-injective objects in $\Mono_Q(\Lambda)$ can be obtained from the indecomposable objects in $\rep (Q,\overline{\operatorname{Mod}}\,\Lambda)\cong \rep (Q,\operatorname{Mod}\mathbbm{k})$ using the Mimo-construction. For example, for the generic module
	\[
	G=\begin{tikzcd}\mathbbm{k}(t)\arrow[r,yshift=0.1cm,"t "above]\arrow[r,yshift=-0.1cm,"1" below]&\mathbbm{k}(t)\end{tikzcd}
	\]
	we have an exact sequence
	\[
	0\to \mathbbm{k}(t) \xrightarrow{\begin{pmatrix}
		-1\\t\end{pmatrix}}\mathbbm{k}(t)\oplus \mathbbm{k}(t)\xrightarrow{\begin{pmatrix}
		t&1
		\end{pmatrix}}\mathbbm{k}(t)\to 0
	\]
	and hence we get an indecomposable object in $\Mono_Q(\Lambda)$
	\[
	\Mimo G=\begin{tikzcd}(\mathbbm{k}(t)\arrow[r,yshift=0.1cm,"l_1 "above]\arrow[r,yshift=-0.1cm,"l_2" below]&\mathbbm{k}(t)\oplus k(t)[x]/(x^2)) \quad  \text{with} \quad l_1=\begin{pmatrix}
	t \\ -x
	\end{pmatrix} \quad \text{and} \quad l_2=\begin{pmatrix}
	1 \\ 0
	\end{pmatrix}.\end{tikzcd}
	\]
	In general, classifying all indecomposables in $\Mono_Q(\Lambda)$ is difficult, since it is difficult for the category $\rep (Q,\operatorname{Mod}\mathbbm{k})$. For example, there is an exact embedding of the category of representations of the 3-Kronecker quiver into 
	$\operatorname{rep}(Q,\operatorname{Mod}\mathbbm{k})$, see \cite{Rin99}. 
	
	\section{Applications to modulations}\label{Section:ApplicationToModulations}
	
	In this section we apply our results to representations of modulations over radical square zero selfinjective Nakayama algebras over $\mathbbm{k}$. In particular, we recover results in \cite{LW22}, and give a characterization for when the GLS algebras introduced in \cite{GLS16} are of finite Cohen--Macaulay type, assuming the entries  in the symmetrizers are $\leq 2$. 
	
	\subsection{Prospecies of radical square zero cyclic Nakayama algebras}
	
	Fix a field $\mathbbm{k}$, a finite acyclic quiver $Q$, and a \emphbf{prospecies} on $Q$ in the sense of \cite{Kul17}, i.e. for each vertex $\mathtt{i}\in Q_0$ a finite-dimensional $\mathbbm{k}$-algebra $\Lambda_\mathtt{i}$ and for each arrow $\alpha\colon \mathtt{i}\to \mathtt{j}$ in $Q$ a $\Lambda_\mathtt{j}$-$\Lambda_\mathtt{i}$-bimodule $M_\alpha$ which is projective as a left $\Lambda_\mathtt{j}$-module and right $\Lambda_\mathtt{i}$-module. Associated to this we have a modulation  $\mathfrak{B}$  on $Q$ where $\mathcal{B}_\mathtt{i}=\operatorname{mod}\Lambda_\mathtt{i}^{\operatorname{op}}$ is the category of finite-dimensional left $\Lambda_\mathtt{i}$-modules  and $$F_\alpha=M_\alpha\otimes_{\Lambda_\mathtt{i}}-\colon \operatorname{mod}\Lambda_\mathtt{i}^{\operatorname{op}}\to \operatorname{mod}\Lambda_\mathtt{j}^{\operatorname{op}}$$ is given by the tensor product, see Example \ref{Example: Phyla}. We also have the tensor algebra
	\[
	T(M)\coloneqq \Lambda\oplus M\oplus (M\otimes_{\Lambda}M)\oplus \dots
	\]
	where $\Lambda=\prod_{\mathtt{i}\in Q_0}\Lambda_{\mathtt{i}}$ and $M=\bigoplus_{\alpha\in Q_1}M_\alpha$ is a $\Lambda$-bimodule in the natural way. The category $\operatorname{rep}\mathfrak{B}$ of $\mathfrak{B}$-representations is equivalent to the category $\operatorname{mod}T(M)^{\operatorname{op}}$ of finitely generated left $T(M)$-modules, see \cite[Lemma 2.3.4]{Geu17}. 
	If the $\Lambda_\mathtt{i}$'s are products of the field $\mathbbm{k}$, then $T(M)$ is path algebra of a quiver, which we describe.
	
	\begin{defn}\label{Definition:ModulationOverField}
		Let $\mathcal{M}=(t_\mathtt{i},m_{k,l}^\alpha)$ be a tuple consisting  of a non-negative integer $t_\mathtt{i}$ for each vertex $\mathtt{i}$ in $Q$ and a non-negative integer $m_{k,l}^\alpha$ for each arrow $\alpha\colon \mathtt{i}\to \mathtt{j}$ in $Q$ and each pair of integers $1\leq k\leq t_\mathtt{i}$ and $1\leq l\leq t_\mathtt{j}$. The quiver $Q(\mathcal{M})$ is defined as follows: 
		\begin{itemize}
			\item $Q(\mathcal{M})_0=\{(\mathtt{i},m)\mid \mathtt{i}\in Q_0\text{ and }1\leq m\leq t_\mathtt{i} \}$.
			\item The number of arrows from $(\mathtt{i},k)$ to $(\mathtt{j},l)$ is $\sum m_{k,l}^\alpha$ where the sums runs over all arrows $\alpha$ with source $\mathtt{i}$ and target $\mathtt{j}$.
		\end{itemize}
	\end{defn} 
	\begin{lem}\label{Lemma:ModulationOverField}
		Assume $\Lambda_{\mathtt{i}}\cong \mathbbm{k}\times \dots \times \mathbbm{k}$ is a product of $t_\mathtt{i}$-copies of $\mathbbm{k}$ for each $\mathtt{i}\in Q_0$. Let $e^{\mathtt{i}}_m$ be the idempotent corresponding to the $m$'th copy of $\mathbbm{k}$ in $\Lambda_\mathtt{i}$. For an arrow $\alpha\colon \mathtt{i}\to \mathtt{j}$  in $Q$ and integers $1\leq k\leq t_\mathtt{i}$ and $1\leq l\leq t_\mathtt{j}$ let  $m_{k,l}^\alpha\coloneqq \operatorname{dim}_{\mathbbm{k}}e^{\mathtt{j}}_lM_\alpha e^{\mathtt{i}}_k$. Then $T(M)\cong \mathbbm{k}(Q(\mathcal{M}))$.
	\end{lem}
	
	\begin{proof}
		Since $\Lambda=\prod_{\mathtt{i}\in Q_0}\Lambda_{\mathtt{i}}$ and $\Lambda_\mathtt{i}$ is a product of $t_\mathtt{i}$ copies of $\mathbbm{k}$, the algebra $\Lambda$ is a product of $\mathbbm{k}$'s indexed over the vertex set of $Q(\mathcal{M})$. By \cite[Proposition III.1.3]{ARS95} the claim follows.
	\end{proof}
	By assumption, the functors $F_\alpha=M_\alpha\otimes_{\Lambda_{\mathtt{i}}}-$ are exact and preserve projective modules. If the algebras $\Lambda_{\mathtt{i}}$ are selfinjective, then the functors also preserve injective modules. Therefore the endofunctor $X$ defined from the modulation satisfies the standing assumptions in this paper. Furthermore, in this case the monomorphism category $\Mono(X)$ of $\mathfrak{B}$ coincides with the category $\operatorname{Gproj}T(M)$ of finitely generated Gorenstein projective right $T(M)$-modules, see \cite[Proposition 3.8]{Kul17}. In particular, by Theorem \ref{Theorem: Canonical functor representation equivalence} we have an epivalence $\overline{\operatorname{Gproj}}\, T(M)^{\operatorname{op}}\to \operatorname{rep}\overline{\mathfrak{B}}$ where $\overline{\mathfrak{B}}$ denotes the modulation in Example \ref{Example:ModulationStableCats}. Since $T(M)$ is $1$-Gorenstein, see \cite[Proposition 3.5]{Kul17}, the stable category $\overline{\operatorname{Gproj}}\, T(M)$ is equivalent to the singularity category of $T(M)$ \cite{Buc86}.
	
	We now consider modulations of radical square zero cyclic Nakayama algebras. This covers both the modulations in \cite{GLS16} and in \cite{LW22}. Here by a radical square zero cyclic Nakayama algebra we mean the path algebra of the cyclic quiver \[
	\tilde{A}_n=\begin{tikzcd}
	1 \arrow[r,""] & 2 \arrow[r, ""] & \dots  \arrow[r, ""]& n \ar[lll,bend right=20,"" above]
	\end{tikzcd}
	\]
	for some integer $n$, modulo the ideal making the composite of any two arrows zero. 
	
	\begin{thm}\label{Theorem:MainThmModulations}
		Assume $\Lambda_\mathtt{i}$ is either a radical square zero cyclic Nakayama algebras or $\mathbbm{k}$ for $\mathtt{i}\in Q_0$. Set 
		\[
		t_{\mathtt{i}}= \begin{cases}
		0, & \text{if}\ \Lambda_\mathtt{i}=\mathbbm{k} \\
		\text{number of simples of }\Lambda_{\mathtt{i}}, & \text{otherwise}.
		\end{cases}
		\]
		For $1\leq k\leq t_{\mathtt{i}}$ let $e^{\mathtt{i}}_k$ be the idempotent at vertex $k$ of $\Lambda_\mathtt{i}$. For each arrow $\alpha\colon \mathtt{i}\to \mathtt{j}$ write $M_\alpha\cong M_\alpha'\oplus M_\alpha''$ where $M''_{\alpha}$ is a maximal projective summand of $M_\alpha$ as a $\Lambda_\mathtt{j}$-$\Lambda_{\mathtt{i}}$-bimodule. For $1\leq k\leq t_{\mathtt{i}}$ and $1\leq l\leq t_\mathtt{j}$ let $m_{k,l}^\alpha$ be the corank of the linear transformation $e^{\mathtt{j}}_lM'_\alpha e^{\mathtt{i}}_{k+1}\to e^{\mathtt{j}}_lM'_\alpha e^{\mathtt{i}}_k$ induced from the arrow $k\to (k+1)$ in $\Lambda_\mathtt{i}$ (where $t_\mathtt{i}+1$ is identified with $1$). Let $\mathcal{M}=(t_\mathtt{i},m_{k,l}^\alpha)$ be the tuple formed by these integers. The following hold:
		
		\begin{enumerate}
			\item\label{Theorem:MainThmModulations:1} We have an equivalence $\operatorname{rep}\overline{\mathfrak{B}}\cong \operatorname{mod}\mathbbm{k}(Q(\mathcal{M}))^{\operatorname{op}}$
			\item\label{Theorem:MainThmModulations:2} We have an epivalence
			\[
			\overline{\operatorname{Gproj}}\,T(M)^{\operatorname{op}}\to \operatorname{mod}\mathbbm{k}(Q(\mathcal{M}))^{\operatorname{op}}. 
			\]
			\item\label{Theorem:MainThmModulations:3} $T(M)$ is Cohen--Macaulay finite if and only if $Q(\mathcal{M})$ is Dynkin. 
			\item\label{Theorem:MainThmModulations:4} We obtain a bijection 
			\[
			\renewcommand{\arraystretch}{1.8}
			\begin{array}{ccc}
			\renewcommand{\arraystretch}{1.1}
			\begin{Bmatrix}
			\text{Indecomposable objects} \\
			\text{in $\operatorname{mod}\mathbbm{k}(Q(\mathcal{M}))^{\operatorname{op}} $}
			\end{Bmatrix}^{\cong}
			&
			\xrightarrow{\cong}
			&
			\renewcommand{\arraystretch}{1.2}
			\begin{Bmatrix}
			\text{Indecomposable non-injective} \\
			\text{objects in $\operatorname{Gproj}T(M)$}
			\end{Bmatrix}^{\cong}
			\end{array}
			\renewcommand{\arraystretch}{1}
			\]
			by composing the equivalence in \eqref{Theorem:MainThmModulations:1} with the Mimo-construction described in Example \ref{Example:MimoStableModulation}.
		\end{enumerate}
	\end{thm}
	
	\begin{proof}
		Since Gorenstein projective modules are the same as monomorphic representations, parts \eqref{Theorem:MainThmModulations:2} and \eqref{Theorem:MainThmModulations:4} follow from part \eqref{Theorem:MainThmModulations:1} and Theorems \ref{Theorem: Canonical functor representation equivalence} and \ref{Theorem: Characterization of indecomposables}. Furthermore, since epivalences induce bijections between the indecomposable objects, part \eqref{Theorem:MainThmModulations:3} follows from part \eqref{Theorem:MainThmModulations:2}. Hence, we only need to show part \eqref{Theorem:MainThmModulations:1}. Since $\Lambda_\mathtt{i}$ is a radical square zero Nakayama algebras, $\overline{\operatorname{mod}}\,\Lambda_\mathtt{i}^{\operatorname{op}}$ is equivalent to the module category of a product of copies of $\mathbbm{k}$. Hence, we can apply Lemma \ref{Lemma:ModulationOverField} to $\overline{\mathfrak{B}}$, so it suffices to show that the integers $t_\mathtt{i}$ and $m_{k,l}^\alpha$ defined in the theorem are equal to the ones in Lemma \ref{Lemma:ModulationOverField}. This is clear for the $t_\mathtt{i}$'s, since $\overline{\operatorname{mod}}\,\Lambda_\mathtt{i}^{\operatorname{op}}\cong \operatorname{mod}k^{t_\mathtt{i}}$  where $t_\mathtt{i}=0$ gives the zero category. Let $S_k^{\mathtt{i}}$ be the simple $\Lambda_\mathtt{i}$-module concentrated at vertex $k$ of $\Lambda_\mathtt{i}$. Note that the integer $m_{k,l}^\alpha$ for $\overline{\mathfrak{B}}$ in Lemma \ref{Lemma:ModulationOverField} is equal to the number of summands of $S_l^{\mathtt{j}}$ in $M_\alpha\otimes_{\Lambda_{\mathtt{i}}}S^{\mathtt{i}}_k$. Since $M_\alpha''$ is projective as a bimodule, $M_\alpha''\otimes_{\Lambda_{\mathtt{i}}}S^{\mathtt{i}}_k$ must be a projective $\Lambda_\mathtt{j}$-module, and hence has no summands of the form $S_l^{\mathtt{j}}$. Therefore the integer $m_{k,l}^\alpha$ for $\overline{\mathfrak{B}}$ in Lemma \ref{Lemma:ModulationOverField} is equal to the number of summands of $S_l^{\mathtt{j}}$ in $M'_\alpha\otimes_{\Lambda_{\mathtt{i}}}S^{\mathtt{i}}_k$. Since $M'_\alpha\otimes_{\Lambda_{\mathtt{i}}}S^{\mathtt{i}}_k$ has no projective summands by \cite[Proposition 2.3]{Lin96}, the number must be equal to the dimension of $e^{\mathtt{j}}_lM'_\alpha\otimes_{\Lambda_{\mathtt{i}}}S^{\mathtt{i}}_k$. Tensoring $e^{\mathtt{j}}_lM'_\alpha$ with the exact sequence $\Lambda e^{\mathtt{i}}_{k+1}\to \Lambda e^{\mathtt{i}}_{k}\to S^{\mathtt{i}}_k\to 0$ gives an exact sequence
		\[
		e^{\mathtt{j}}_lM'_{\alpha} e^{\mathtt{i}}_{k+1}\to e^{\mathtt{j}}_lM'_{\alpha} e^{\mathtt{i}}_k\to e^{\mathtt{j}}_lM'_\alpha\otimes_{\Lambda_{\mathtt{i}}}S^{\mathtt{i}}_k\to 0.
		\]
		Since the corank of the leftmost map is equal to the dimension of its cokernel, which is $e^{\mathtt{j}}_lM'_\alpha\otimes_{\Lambda_{\mathtt{i}}}S^{\mathtt{i}}_k$, this proves the claim.
	\end{proof}
	
	\begin{rmk}
		As noted in the proof, the integers $m_{k,l}^\alpha$ in Theorem \ref{Theorem:MainThmModulations} could equivalently be defined as the number of summands of $S_l^{\mathtt{j}}$ in $M_\alpha\otimes_{\Lambda_{\mathtt{i}}}S^{\mathtt{i}}_k$, or as the dimension of $e^{\mathtt{j}}_lM'_\alpha\otimes_{\Lambda_{\mathtt{i}}}S^{\mathtt{i}}_k$. They are also equal to the nullity of the linear transformations $e^{\mathtt{j}}_lM'_{\alpha} e^{\mathtt{i}}_{k-1}\to e^{\mathtt{j}}_lM'_{\alpha} e^{\mathtt{i}}_{k-2}$ associated to the arrow $(k-1)\to (k-2)$ (where $0$ and $-1$ are identified with $t_{\mathtt{i}}$ and $t_{\mathtt{i}}-1$, respectively).
	\end{rmk}
	
	\subsection{$\imath$Quiver algebras and algebras associated to symmetrizable Cartan matrices}\label{Subsection:ModulationsiQuiverGLS}
	
	Let $Q$ be a finite acyclic quiver with an involutive automorphism $\tau$ respecting the arrows. In \cite{LW22} they call  such a pair $(Q,\tau)$ an $\imath$\emphbf{quiver}, and associate an algebra $\Lambda^{\imath}$ to it. Their goal is to extend the work of Bridgeland on the realization of quantum groups via Hall algebras to $\imath$quantum groups. In particular, semi-derived Hall algebras of algebras of the form $\Lambda^{\imath}$ are isomorphic to universal quasi-split $\imath$quantum groups of finite type \cite[Theorems G and I]{LW22}. 
	
	It turns out that the category of finitely generated $\Lambda^{\imath}$-modules is equivalent to the category of representations of a prospecies satisfying the conditions in Theorem \ref{Theorem:MainThmModulations}, see \cite[Section 2.4]{LW22}. Furthermore, the monomorphism category of the corresponding modulation is equal to the category of Gorenstein projective $\Lambda^{\imath}$-modules, and therefore plays an important role when computing the semi-derived Hall algebra, see \cite[Theorem C]{LW22}. For Dynkin quivers the monomorphism category is in addition equivalent to the category of finitely generated projectives over the regular Nakajima--Keller--Scherotzke categories considered in \cite{LW21b}.
	
	Explicitly, the prospecies they consider is as follows: Choose a representative for each $\tau$-orbit of a vertex in $Q$, and let $\mathbb{I}_\tau$ be the set of these representatives. The quiver $Q'$ has as vertices the set $\mathbb{I}_\tau$, it has no double arrows, and there is an arrow from $\mathtt{i}$ to $\mathtt{j}$ in $Q'$ if and only if there is an arrow from a vertex in the $\tau$-orbit of $\mathtt{i}$ to the $\tau$-orbit of $\mathtt{j}$ in $Q$. The prospecies on $Q'$ is given by a tuple $(\mathbb{H}_\mathtt{i},{}_\mathtt{j}\mathbb{H}_\mathtt{i})$ where 
	\[
	\mathbb{H}_\mathtt{i}=
	\begin{cases}
	\mathbbm{k}[x]/(x^2), & \text{if}\ \tau(\mathtt{i})=\mathtt{i} \\
	\mathbbm{k}(\begin{tikzcd}
	1 \arrow[r,bend right=20,"x" below] & 2 \ar[l,bend right=20,"y" above]
	\end{tikzcd})/(xy,yx) & \text{if}\ \tau(\mathtt{i})\neq \mathtt{i}.
	\end{cases}
	\]
	For the description of the bimodules ${}_\mathtt{j}\mathbb{H}_\mathtt{i}$ see page 16 in \cite{LW22}. 
	
	\begin{prop}
		Let $(Q,\tau)$ be an $\imath$quiver, let $(\mathbb{H}_\mathtt{i},{}_\mathtt{j}\mathbb{H}_\mathtt{i})$ be the associated prospecies on the quiver $Q'$ as above, and let $\mathcal{M}=(t_\mathtt{i},m_{k,l}^\alpha)$ be the associated tuple of integers in Theorem \ref{Theorem:MainThmModulations}. Then $Q'(\mathcal{M})=Q$.
	\end{prop} 
	
	\begin{proof}
		First note that the association 
		\[
		(\mathtt{i},1)\mapsto \mathtt{i} \quad \text{and} \quad (\mathtt{i},2)\mapsto \tau(\mathtt{i})
		\]
		gives a bijection between $Q'(\mathcal{M})_0$ and $Q_0$. Also, all nonzero ${}_\mathtt{j}\mathbb{H}_\mathtt{i}$ are of Loewy length $1$ as $\mathbb{H}_\mathtt{j}$-$\mathbb{H}_\mathtt{i}$-bimodules, since they vanish when multiplying with any combination of two of the nilpotent elements $\varepsilon_\mathtt{i}$ and $\varepsilon_{\mathtt{j}}$ of $\mathbb{H}_i$ and $\mathbb{H}_j$, see the description on page 16 in \cite{LW22}. Since all nonzero projective $\mathbb{H}_\mathtt{j}$-$\mathbb{H}_\mathtt{i}$-bimodules have Loewy length $2$, it follows that ${}_\mathtt{j}\mathbb{H}_\mathtt{i}$ has no nonzero summands which are projective. Therefore, by Theorem \ref{Theorem:MainThmModulations} the number of arrows from $(\mathtt{i},k)$ to $(\mathtt{j},l)$ in $Q'(\mathcal{M})$ is equal to the corank of the $\mathbbm{k}$-morphism $e^{\mathtt{j}}_l{}_\mathtt{j}\mathbb{H}_\mathtt{i}e^{\mathtt{i}}_{k+1}\to e^{\mathtt{j}}_l{}_\mathtt{j}\mathbb{H}_\mathtt{i}e^{\mathtt{i}}_{k}$ as in Theorem \ref{Theorem:MainThmModulations}. Using the $\mathbbm{k}$-linear basis of the bimodules ${}_\mathtt{j}\mathbb{H}_\mathtt{i}$ on page 16 in \cite{LW22}, we see that this is equal to the number of arrows in $Q$ from the vertex $\tau^{k-1}(\mathtt{i})$ to $\tau^{l-1}(\mathtt{j})$. The claim follows.
	\end{proof}
	
	\begin{rmk}
		By Theorem \ref{Theorem:MainThmModulations} we have an epivalence
		\[
		\overline{\operatorname{Gproj}}\,(\Lambda^{\imath})^{\operatorname{op}}\to \operatorname{mod}(\mathbbm{k}Q)^{\operatorname{op}}.
		\]  
		In particular, it induces a bijection from the indecomposable non-projective Gorenstein projective $\Lambda^{\imath}$-modules to the indecomposable $\mathbbm{k}Q$-modules, which recovers \cite[Corollary 3.21]{LW22}. Furthermore, by Theorem \ref{Theorem:MainThmModulations} \eqref{Theorem:MainThmModulations:4} we have an explicit description of the inverse to this bijection. It would be interesting to investigate how this inverse can be used to study $\Lambda^{\imath}$ and its Hall-algebra.
	\end{rmk}
	
	Given a symmetric Cartan matrix $C=(c_{\mathtt{i},\mathtt{j}})_{\mathtt{i},\mathtt{j}\in I}$ with (acyclic) orientation $\Omega\subset I\times I$, one can associate a path algebra $\mathbbm{k}Q$ whose quiver $Q$ has vertex set $I$ and has $\lvert c_{\mathtt{i},\mathtt{j}}\rvert$ arrows from $i$ to $j$ if $(j,i)\in \Omega$. This was extended in \cite{GLS16}, where they associate an algebra $H=H(C,D,\Omega)$ to the data of a symmetrizable Cartan matrix $C=(c_{\mathtt{i},\mathtt{j}})_{\mathtt{i},\mathtt{j}\in I}$ with symmetrizer $D=\operatorname{diag}(d_\mathtt{i}\mid \mathtt{i}\in I)$ and (acyclic) orientation $\Omega\subset I\times I$. The category of finitely generated left modules over $H(C,D,\Omega)$ is equivalent to representations of a prospecies $(H_\mathtt{i},{}_{\mathtt{j}}H_{\mathtt{i}})$ over a quiver $Q'$. Explicitly, the quiver $Q'$ has vertex set $I$, it has no double arrows, and there is an arrow from $\mathtt{i}$ to $\mathtt{j}$ if $(j,i)\in \Omega$. The algebra $H_\mathtt{i}$ is equal to $\mathbbm{k}[x]/(x^{d_\mathtt{i}})$, and the bimodule ${}_{\mathtt{j}}H_{\mathtt{i}}$ is described in \cite[Section 5]{GLS16}. If $d_\mathtt{i}\leq 2$ for all $\mathtt{i}\in I$, then we can apply Theorem \ref{Theorem:MainThmModulations} to this prospecies. The following proposition gives a description of $Q(\mathcal{M})$ in this case. We use it to deduce Theorem \ref{Theorem:GLSALgebras} in the introduction.
	
	\begin{prop}\label{Proposition:ApplicationToGLS}
		Let $C=(c_{\mathtt{i},\mathtt{j}})_{\mathtt{i},\mathtt{j}\in I}$ be a symmetrizable Cartan matrix with symmetrizer $D=\operatorname{diag}(d_\mathtt{i}\mid \mathtt{i}\in I)$, and let $\Omega\subset I\times I$ be an orientation of $C$. Assume $d_\mathtt{i}\leq 2$ for all $\mathtt{i}\in I$. Let $I'\subset I$ be the subset consisting of all $\mathtt{i}$ for which $d_\mathtt{i}=2$, and let $\mathcal{M}=(t_\mathtt{i},m_{k,l}^\alpha)$ be the tuple defined from the prospecies $(H_\mathtt{i},{}_{\mathtt{j}}H_{\mathtt{i}})$ as in Theorem \ref{Theorem:MainThmModulations}. Then $Q'(\mathcal{M})$ is equal to the quiver defined by the symmetric Cartan matrix $C|_{I'\times I'}$ with orientation $\Omega|_{I'\times I'}$.
	\end{prop}
	
	\begin{proof}
		By definition, $t_{\mathtt{i}}=1$ if $\mathtt{i}\in I'$, and $t_{\mathtt{i}}=0$ otherwise. Hence, the vertex set $Q'(\mathcal{M})_0$ can be identified with $I'$.  Now assume $t_\mathtt{i}=1=t_\mathtt{j}$. Note that all nonzero ${}_\mathtt{j}H_\mathtt{i}$ are of Loewy length $1$ as $H_\mathtt{j}$-$H_\mathtt{i}$-bimodules, since they vanish when multiplying by any combination of two of the nilpotent elements $\varepsilon_\mathtt{i}$ and $\varepsilon_{\mathtt{j}}$ of $H_i$ and $H_j$, see the description in \cite[Section 5]{GLS16}. Since all nonzero projective $H_\mathtt{j}$-$H_\mathtt{i}$-bimodules have Loewy length $2$, it follows that ${}_\mathtt{j}H_\mathtt{i}$ have no nonzero projective summands. Therefore by Theorem \ref{Theorem:MainThmModulations} the number of arrows from $\mathtt{i}$ to $\mathtt{j}$ in $Q'(\mathcal{M})_0$ is equal to the corank of the map ${}_{\mathtt{j}}H_{\mathtt{i}}\xrightarrow{} {}_{\mathtt{j}}H_{\mathtt{i}}$ given by multiplication with $x$ on the right. Now ${}_{\mathtt{j}}H_{\mathtt{i}}\cong H_{\mathtt{i}}^{\lvert c_{\mathtt{i},\mathtt{j}}\rvert}$ as a right $H_\mathtt{i}$-module, see Section $5$ in \cite{GLS16}. Hence, the corank is equal to $\lvert c_{\mathtt{i},\mathtt{j}}\rvert$, which proves the claim.
	\end{proof}
	
	\begin{proof}[Proof of Theorem \ref{Theorem:GLSALgebras}]
		Since the algebra $H=H(C,D,\Omega)$ is $1$-Gorenstein by \cite[Theorem 1.2]{GLS16} (see also the discussion above), the singularity category of $H$ is equivalent to the stable category $\overline{\operatorname{Gproj}}\,H$ by Buchweitz' theorem \cite{Buc86}. By Theorem \ref{Theorem:MainThmModulations} there is an epivalence from $\overline{\operatorname{Gproj}}\,H$ to $\operatorname{mod}\mathbbm{k}(Q'(\mathcal{M}))^{\operatorname{op}}$, and hence there is a bijection between their indecomposable objects. By Proposition \ref{Proposition:ApplicationToGLS} the quiver $Q'(\mathcal{M})$ is obtained from the symmetric Cartan matrix $C|_{I'\times I'}$ with orientation $\Omega|_{I'\times I'}$. Hence, by Gabriel's theorem there are finitely many isomorphism classes of indecomposable $\mathbbm{k}Q'(\mathcal{M})$-modules if and only if $C|_{I'\times I'}$ is Dynkin, and in that case they are in bijection with the positive roots of $C|_{I'\times I'}$. This proves the claim.
	\end{proof}

	\section*{Acknowledgements}
	
	We would like to thank Karin M. Jacobsen for asking the question which lead to Theorem \ref{Theorem:EquivalenceHereditary}. We would also like to thank Henning Krause for mentioning how to apply \cite[Theorem 13.1.28]{Kra22} to Dedekind domains in Remark \ref{Remark:DedekindDomain}. We would like to thank the referee for useful comments and suggestions.
	
	\bibliographystyle{alpha}
	\bibliography{publication}

\begin{thebibliography}{GKKP22}

\bibitem[ABM98]{ABM98}
Ibrahim Assem, Apostolos Beligiannis, and Nikolaos Marmaridis.
\newblock Right triangulated categories with right semi-equivalences.
\newblock In {\em Algebras and modules, {II} ({G}eiranger, 1996)}, volume~24 of
  {\em CMS Conf. Proc.}, pages 17--37. Amer. Math. Soc., Providence, RI, 1998.

\bibitem[AR76]{AR75/76}
Maurice Auslander and Idun Reiten.
\newblock On the representation type of triangular matrix rings.
\newblock {\em J. London Math. Soc. (2)}, 12(3):371--382, 1975/76.

\bibitem[Arn00]{Arn00}
David~M. Arnold.
\newblock {\em Abelian groups and representations of finite partially ordered
  sets}, volume~2 of {\em CMS Books in Mathematics/Ouvrages de
  Math\'{e}matiques de la SMC}.
\newblock Springer-Verlag, New York, 2000.

\bibitem[ARS95]{ARS95}
Maurice Auslander, Idun Reiten, and Sverre~Olaf Smal{\o}.
\newblock {\em {Representation Theory of {A}rtin Algebras}}.
\newblock Cambridge University Press, 1995.

\bibitem[Asa99]{Asa99}
Hideto Asashiba.
\newblock The derived equivalence classification of representation-finite
  selfinjective algebras.
\newblock {\em J. Algebra}, 214(1):182--221, 1999.

\bibitem[Aus71]{Aus71}
Maurice Auslander.
\newblock Representation dimension of artin algebras.
\newblock {\em Queen Mary College Mathematics Notes}, 1971.
\newblock republished in \emph{Selected works of Maurice Auslander}. Amer.
  Math. Soc., Providence 1999.

\bibitem[Bau95]{Bau95}
Hans~Joachim Baues.
\newblock Homotopy types.
\newblock In {\em Handbook of algebraic topology}, pages 1--72. North-Holland,
  Amsterdam, 1995.

\bibitem[BBOS20]{BBOS20}
Ulrich Bauer, Magnus~B. Botnan, Steffen Oppermann, and Johan Steen.
\newblock Cotorsion torsion triples and the representation theory of filtered
  hierarchical clustering.
\newblock {\em Adv. Math.}, 369:107171, 51, 2020.

\bibitem[Bir35]{Bir35}
Garrett Birkhoff.
\newblock Subgroups of {A}belian {G}roups.
\newblock {\em Proc. London Math. Soc. (2)}, 38:385--401, 1935.

\bibitem[BM94]{BM94}
Apostolos Beligiannis and Nikolaos Marmaridis.
\newblock Left triangulated categories arising from contravariantly finite
  subcategories.
\newblock {\em Comm. Algebra}, 22(12):5021--5036, 1994.

\bibitem[Bor94]{Bor94a}
Francis Borceux.
\newblock {\em Handbook of categorical algebra. 2}, volume~51 of {\em
  Encyclopedia of Mathematics and its Applications}.
\newblock Cambridge University Press, Cambridge, 1994.
\newblock Categories and Structures.

\bibitem[Buc21]{Buc86}
Ragnar-Olaf Buchweitz.
\newblock {\em Maximal {C}ohen-{M}acaulay modules and {T}ate cohomology},
  volume 262 of {\em Mathematical Surveys and Monographs}.
\newblock American Mathematical Society, Providence, RI, [2021] \copyright
  2021.

\bibitem[B{\"u}h10]{Bue10}
Theo B{\"u}hler.
\newblock Exact categories.
\newblock {\em Expo. Math.}, 28(1):1--69, 2010.

\bibitem[Che11]{Che11}
Xiao-Wu Chen.
\newblock The stable monomorphism category of a {F}robenius category.
\newblock {\em Math. Res. Lett.}, 18(1):125--137, 2011.

\bibitem[Che12]{Che12}
Xiao-Wu Chen.
\newblock Three results on {F}robenius categories.
\newblock {\em Math. Z.}, 270(1-2):43--58, 2012.

\bibitem[CL20]{CL20}
Xiao-Wu Chen and Ming Lu.
\newblock Gorenstein homological properties of tensor rings.
\newblock {\em Nagoya Math. J.}, 237:188--208, 2020.

\bibitem[DELO21]{DELO21}
Zhenxing Di, Sergio Estrada, Li~Liang, and Sinem Odaba\c{s}\i.
\newblock Gorenstein flat representations of left rooted quivers.
\newblock {\em J. Algebra}, 584:180--214, 2021.

\bibitem[DLLY22]{DLLY22}
Zhenxing Di, Liping Li, Li~Liang, and Nina Yu.
\newblock Representations over diagrams of categories and abelian model
  structures, 2022.
\newblock arXiv:2210.08558.

\bibitem[DR76]{DR76}
Vlastimil Dlab and Claus~Michael Ringel.
\newblock {Indecomposable representations of graphs and algebras}.
\newblock {\em Memoirs of the American Mathematical Society}, 6:v+57, 1976.

\bibitem[EE05]{EE05}
Edgar Enochs and Sergio Estrada.
\newblock Projective representations of quivers.
\newblock {\em Comm. Algebra}, 33(10):3467--3478, 2005.

\bibitem[EEGR09]{EEG09}
E.~Enochs, S.~Estrada, and J.~R. Garc\'{\i}a~Rozas.
\newblock Injective representations of infinite quivers. {A}pplications.
\newblock {\em Canad. J. Math.}, 61(2):315--335, 2009.

\bibitem[EHHS13]{EHHS13}
Hossein Eshraghi, Rasool Hafezi, Esmaeil Hosseini, and Shokrollah Salarian.
\newblock Cotorsion theory in the category of quiver representations.
\newblock {\em J. Algebra Appl.}, 12(6):1350005, 16, 2013.

\bibitem[EM65]{EM65}
Samuel Eilenberg and John~C. Moore.
\newblock {Adjoint functors and triples}.
\newblock {\em Illinois Journal of Mathematics}, 9:381--398, 1965.

\bibitem[EOT04]{EOT04}
Edgar Enochs, Luis Oyonarte, and Blas Torrecillas.
\newblock Flat covers and flat representations of quivers.
\newblock {\em Comm. Algebra}, 32(4):1319--1338, 2004.

\bibitem[Fai66]{Fai66}
Carl Faith.
\newblock On {K}\"{o}the rings.
\newblock {\em Math. Ann.}, 164:207--212, 1966.

\bibitem[FGR75]{FGR75}
Robert~M. Fossum, Phillip~A. Griffith, and Idun Reiten.
\newblock {\em {Trivial extensions of abelian categories}}, volume 456 of {\em
  {Lecture Notes in Mathematics}}.
\newblock Springer, Berlin-New York, 1975.
\newblock Homological algebra of trivial extensions of abelian categories with
  applications to ring theory.

\bibitem[FS79]{FS79}
Ferdinand~Georg Frobenius and Ludwig Stickelberger.
\newblock Ueber {G}ruppen von vertauschbaren {E}lementen.
\newblock {\em Journal f\"{u}r die Reine und Angewandte Mathematik. [Crelle's
  Journal]}, 86:217--262, 1879.

\bibitem[Gab72]{G72}
Peter Gabriel.
\newblock {{U}nzerlegbare {D}arstellungen {I}}.
\newblock {\em Manuscripta Mathematica}, 6:71--103, 1972.

\bibitem[Geu17]{Geu17}
Jan Geuenich.
\newblock {\em Quiver modulations and potentials}.
\newblock Ph{D} thesis, 2017.

\bibitem[GK05]{GK05}
Peter~B. Gothen and Alastair~D. King.
\newblock Homological algebra of twisted quiver bundles.
\newblock {\em J. London Math. Soc. (2)}, 71(1):85--99, 2005.

\bibitem[GKKP22]{GKKP19}
Nan Gao, Julian K\"{u}lshammer, Sondre Kvamme, and Chrysostomos Psaroudakis.
\newblock A functorial approach to monomorphism categories for species {I}.
\newblock {\em Communications in Contemporary Mathematics}, 24(6):Paper No.
  2150069, 55, 2022.

\bibitem[GLS17]{GLS16}
Christof Geiss, Bernard Leclerc, and Jan Schr{\"o}er.
\newblock {Quiver with relations for symmetrizable {C}artan matrices {I}:
  {F}oundations}.
\newblock {\em Inventiones Mathematicae}, 209(1):61--158, 2017.

\bibitem[Hap87]{Hap87}
Dieter Happel.
\newblock {On the derived category of a finite-dimensional algebra}.
\newblock {\em Commentarii Mathematici Helvetici}, 62(3):339--389, 1987.

\bibitem[Har69]{Har69}
Hideki Harui.
\newblock On injective modules.
\newblock {\em J. Math. Soc. Japan}, 21:574--583, 1969.

\bibitem[Hel60]{Hel60}
Alex Heller.
\newblock The loop-space functor in homological algebra.
\newblock {\em Trans. Amer. Math. Soc.}, 96:382--394, 1960.

\bibitem[Hil07]{Hil07}
Harold Hilton.
\newblock On {S}ub-{G}roups of a {F}inite {A}belian {G}roup.
\newblock {\em Proc. London Math. Soc. (2)}, 5:1--5, 1907.

\bibitem[Hir00]{Hir00}
Y.~Hirano.
\newblock On injective hulls of simple modules.
\newblock {\em J. Algebra}, 225(1):299--308, 2000.

\bibitem[HJ19]{HJ19a}
Henrik Holm and Peter J{\o}rgensen.
\newblock Cotorsion pairs in categories of quiver representations.
\newblock {\em Kyoto J. Math.}, 59(3):575--606, 2019.

\bibitem[HMA21]{HM21}
Rasool Hafezi and Intan Muchtadi-Alamsyah.
\newblock Different exact structures on the monomorphism categories.
\newblock {\em Appl. Categ. Structures}, 29(1):31--68, 2021.

\bibitem[HRW84]{RRW84}
Roger Hunter, Fred Richman, and Elbert Walker.
\newblock Subgroups of bounded abelian groups.
\newblock In {\em Abelian groups and modules ({U}dine, 1984)}, volume 287 of
  {\em CISM Courses and Lect.}, pages 17--35. Springer, Vienna, 1984.

\bibitem[HV83]{HV83}
Dieter Happel and Dieter Vossieck.
\newblock Minimal algebras of infinite representation type with preprojective
  component.
\newblock {\em Manuscripta Math.}, 42(2-3):221--243, 1983.

\bibitem[Jan69]{Jan69}
J.~P. Jans.
\newblock On co-{N}oetherian rings.
\newblock {\em J. London Math. Soc. (2)}, 1:588--590, 1969.

\bibitem[Jat76]{Jat76}
Arun~Vinayak Jategaonkar.
\newblock Certain injectives are {A}rtinian.
\newblock In {\em Noncommutative ring theory ({I}nternat. {C}onf., {K}ent
  {S}tate {U}niv., {K}ent, {O}hio, 1975)}, Lecture Notes in Math., Vol. 545,
  pages 128--139. Springer, Berlin, 1976.

\bibitem[Kel64]{Kel64}
G.~M. Kelly.
\newblock On the radical of a category.
\newblock {\em J. Austral. Math. Soc.}, 4:299--307, 1964.

\bibitem[Kel90]{Kel90}
Bernhard Keller.
\newblock Chain complexes and stable categories.
\newblock {\em Manuscripta Math.}, 67(4):379--417, 1990.

\bibitem[Kel91]{Kel91}
Bernhard Keller.
\newblock Derived categories and universal problems.
\newblock {\em Comm. Algebra}, 19(3):699--747, 1991.

\bibitem[Kel96]{Kel96}
Bernhard Keller.
\newblock Derived categories and their uses.
\newblock In {\em Handbook of algebra, {V}ol. 1}, volume~1, pages 671--701.
  1996.

\bibitem[KLM13]{KLM13}
Dirk Kussin, Helmut Lenzing, and Hagen Meltzer.
\newblock Nilpotent operators and weighted projective lines.
\newblock {\em J. Reine Angew. Math.}, 685:33--71, 2013.

\bibitem[Kra97]{Kra97}
Henning Krause.
\newblock Stable equivalence preserves representation type.
\newblock {\em Comment. Math. Helv.}, 72(2):266--284, 1997.

\bibitem[Kra15]{Kra15}
Henning Krause.
\newblock Krull--{S}chmidt categories and projective covers.
\newblock {\em Expositiones Mathematicae}, 33(4):535--549, 2015.

\bibitem[Kra22]{Kra22}
Henning Krause.
\newblock {\em Homological theory of representations}, volume 195 of {\em
  Cambridge Studies in Advanced Mathematics}.
\newblock Cambridge University Press, Cambridge, 2022.

\bibitem[Kro70]{Kro70}
Leopold Kronecker.
\newblock Auseinandersetzung einiger eigenschaften der klassenzahl idealer
  complexer zahlen.
\newblock {\em Monatsbericht der K{\"o}niglich-Preussischen Akademie der
  Wissenschaften zu Berlin}, pages 881--889, 1870.

\bibitem[KS15]{KS15}
Justyna Kosakowska and Markus Schmidmeier.
\newblock Operations on arc diagrams and degenerations for invariant subspaces
  of linear operators.
\newblock {\em Trans. Amer. Math. Soc.}, 367(8):5475--5505, 2015.

\bibitem[KS22]{KS22}
Justyna Kosakowska and Markus Schmidmeier.
\newblock The socle tableau as a dual version of the {L}ittlewood-{R}ichardson
  tableau.
\newblock {\em J. Lond. Math. Soc. (2)}, 106(2):1357--1379, 2022.

\bibitem[K{\"u}l17]{Kul17}
Julian K{\"u}lshammer.
\newblock {Pro-species of Algebras {I}: Basic properties}.
\newblock {\em Algebras and Representation Theory}, 20(5):1215--1238, 2017.

\bibitem[KV87]{KV87}
Bernhard Keller and Dieter Vossieck.
\newblock Sous les cat\'{e}gories d\'{e}riv\'{e}es.
\newblock {\em C. R. Acad. Sci. Paris S\'{e}r. I Math.}, 305(6):225--228, 1987.

\bibitem[Kva20]{Kva20b}
Sondre Kvamme.
\newblock A generalization of the {N}akayama functor.
\newblock {\em Algebr. Represent. Theory}, 23(4):1319--1353, 2020.

\bibitem[Les94]{Les94}
Zbigniew Leszczy\'{n}ski.
\newblock On the representation type of tensor product algebras.
\newblock {\em Fund. Math.}, 144(2):143--161, 1994.

\bibitem[Lin96]{Lin96}
Markus Linckelmann.
\newblock Stable equivalences of {M}orita type for self-injective algebras and
  {$p$}-groups.
\newblock {\em Math. Z.}, 223(1):87--100, 1996.

\bibitem[LO17]{LO17}
Boris Lerner and Steffen Oppermann.
\newblock {A recollement approach to {G}eigle-{L}enzing weight projective
  varieties}.
\newblock {\em Nagoya Mathematical Journal}, 226:71--105, 2017.

\bibitem[LS00]{LS00}
Zbigniew Leszczy\'{n}ski and Andrzej Skowro\'{n}ski.
\newblock Tame triangular matrix algebras.
\newblock {\em Colloq. Math.}, 86(2):259--303, 2000.

\bibitem[LS22]{LS22}
Xiu-Hua Luo and Markus Schmidmeier.
\newblock A reflection equivalence for {G}orenstein-projective quiver
  representations, 2022.
\newblock Preprint, arXiv:2204.04695.

\bibitem[Lu20]{Lu20}
Ming Lu.
\newblock Singularity categories of representations of algebras over local
  rings.
\newblock {\em Colloq. Math.}, 161(1):1--33, 2020.

\bibitem[LW21a]{LW21a}
Ming Lu and Weiqiang Wang.
\newblock Hall algebras and quantum symmetric pairs {II}: {R}eflection
  functors.
\newblock {\em Comm. Math. Phys.}, 381(3):799--855, 2021.

\bibitem[LW21b]{LW21b}
Ming Lu and Weiqiang Wang.
\newblock Hall algebras and quantum symmetric pairs {III}: {Q}uiver varieties.
\newblock {\em Adv. Math.}, 393:Paper No. 108071, 70, 2021.

\bibitem[LW22]{LW22}
Ming Lu and Weiqiang Wang.
\newblock Hall algebras and quantum symmetric pairs {I}: {F}oundations.
\newblock {\em Proc. Lond. Math. Soc. (3)}, 124(1):1--82, 2022.
\newblock With an appendix by Lu.

\bibitem[LZ10]{LZ10}
Zhi-Wei Li and Pu~Zhang.
\newblock A construction of {G}orenstein-projective modules.
\newblock {\em J. Algebra}, 323(6):1802--1812, 2010.

\bibitem[LZ13]{LZ13}
Xiu-Hua Luo and Pu~Zhang.
\newblock {Monic representations and {G}orenstein-projective modules}.
\newblock {\em Pacific Journal of Mathematics}, 264(1):163--194, 2013.

\bibitem[{Mac}98]{McL98}
Saunders {Mac Lane}.
\newblock {\em {Categories for the working mathematician}}, volume~5 of {\em
  {Graduate Texts in Mathematics}}.
\newblock Springer-Verlag, New York, second edition, 1998.

\bibitem[Mat58]{Mat58}
Eben Matlis.
\newblock Injective modules over {N}oetherian rings.
\newblock {\em Pacific J. Math.}, 8:511--528, 1958.

\bibitem[Mat60]{Mat60}
Eben Matlis.
\newblock Modules with descending chain condition.
\newblock {\em Trans. Amer. Math. Soc.}, 97:495--508, 1960.

\bibitem[Mil04]{Mil04}
G.~A. Miller.
\newblock On the subgroups of an abelian group.
\newblock {\em Ann. of Math. (2)}, 6(1):1--6, 1904.

\bibitem[Mil05]{Mil05}
G.~A. Miller.
\newblock Determination of all the characteristic subgroups of any abelian
  group.
\newblock {\em Amer. J. Math.}, 27(1):15--24, 1905.

\bibitem[Moo09]{Moo09}
Audrey Moore.
\newblock {\em Auslander-{R}eiten theory for systems of submodule embeddings}.
\newblock ProQuest LLC, Ann Arbor, MI, 2009.
\newblock Thesis (Ph.D.)--Florida Atlantic University.

\bibitem[Moz20]{Moz20}
Sergey Mozgovoy.
\newblock Quiver representations in abelian categories.
\newblock {\em J. Algebra}, 541:35--50, 2020.

\bibitem[MZ15]{MZ15}
Alex Martsinkovsky and Dali Zangurashvili.
\newblock The stable category of a left hereditary ring.
\newblock {\em J. Pure Appl. Algebra}, 219(9):4061--4089, 2015.

\bibitem[Pla76]{Pla76}
V.~V. Plahotnik.
\newblock Representations of partially ordered sets over commutative rings.
\newblock {\em Izv. Akad. Nauk SSSR Ser. Mat.}, 40(3):527--543, 709, 1976.

\bibitem[Ric89]{Ric89b}
Jeremy Rickard.
\newblock Derived categories and stable equivalence.
\newblock {\em J. Pure Appl. Algebra}, 61(3):303--317, 1989.

\bibitem[Ric91]{Ric91}
Jeremy Rickard.
\newblock {Derived equivalences as derived functors}.
\newblock {\em J. London Math. Soc. (2)}, 43(1):37--48, 1991.

\bibitem[Rin99]{Rin99}
Claus~Michael Ringel.
\newblock Tame algebras are wild.
\newblock {\em Algebra Colloq.}, 6(4):473--480, 1999.

\bibitem[RS06]{RS06}
Claus~Michael Ringel and Markus Schmidmeier.
\newblock {Submodule categories of wild representation type}.
\newblock {\em Jounal of Pure and Applied Algebra}, 205(2):412--422, 2006.

\bibitem[RS08a]{RS08b}
Claus~Michael Ringel and Markus Schmidmeier.
\newblock Invariant subspaces of nilpotent linear operators. {I}.
\newblock {\em J. Reine Angew. Math.}, 614:1--52, 2008.

\bibitem[RS08b]{RS08}
Claus~Michael Ringel and Markus Schmidmeier.
\newblock {The {A}uslander-{R}eiten translation in submodule categories}.
\newblock {\em Transactions of the American Mathematical Society},
  360(2):691--716, 2008.

\bibitem[RS24]{RS24}
Claus~Michael Ringel and Markus Schmidmeier.
\newblock Invariant subspaces of nilpotent operators. level, mean, and colevel:
  The triangle $\mathbb{T}(n)$, 2024.

\bibitem[RW79]{RW79}
Fred Richman and Elbert~A. Walker.
\newblock Valuated groups.
\newblock {\em J. Algebra}, 56(1):145--167, 1979.

\bibitem[RW99]{RW99}
Fred Richman and Elbert~A. Walker.
\newblock Subgroups of {$p^5$}-bounded groups.
\newblock In {\em Abelian groups and modules ({D}ublin, 1998)}, Trends Math.,
  pages 55--73. Birkh\"{a}user, Basel, 1999.

\bibitem[RZ14]{RZ14}
Claus~Michael Ringel and Pu~Zhang.
\newblock From submodule categories to preprojective algebras.
\newblock {\em Math. Z.}, 278(1-2):55--73, 2014.

\bibitem[RZ17]{RZ17}
Claus~Michael Ringel and Pu~Zhang.
\newblock Representations of quivers over the algebra of dual numbers.
\newblock {\em Journal of Algebra}, 475:327--360, 2017.

\bibitem[Sch05]{Sch05a}
Markus Schmidmeier.
\newblock A construction of metabelian groups.
\newblock {\em Arch. Math. (Basel)}, 84(5):392--397, 2005.

\bibitem[Sch08]{Sch08}
Markus Schmidmeier.
\newblock Systems of submodules and an isomorphism problem for
  {A}uslander-{R}eiten quivers.
\newblock {\em Bull. Belg. Math. Soc. Simon Stevin}, 15(3):523--546, 2008.

\bibitem[Sch11]{Sch11}
Markus Schmidmeier.
\newblock The entries in the {LR}-tableau.
\newblock {\em Math. Z.}, 268(1-2):211--222, 2011.

\bibitem[Sch12]{Sch12}
Markus Schmidmeier.
\newblock Hall polynomials via automorphisms of short exact sequences.
\newblock {\em Algebr. Represent. Theory}, 15(3):449--481, 2012.

\bibitem[Sim02]{Sim02}
Daniel Simson.
\newblock Chain categories of modules and subprojective representations of
  posets over uniserial algebras.
\newblock In {\em Proceedings of the {S}econd {H}onolulu {C}onference on
  {A}belian {G}roups and {M}odules ({H}onolulu, {HI}, 2001)}, volume~32, pages
  1627--1650, 2002.

\bibitem[Sim18]{Sim18}
Daniel Simson.
\newblock Representation-finite {B}irkhoff type problems for nilpotent linear
  operators.
\newblock {\em J. Pure Appl. Algebra}, 222(8):2181--2198, 2018.

\bibitem[V\'68]{Vam68}
P.~V\'{a}mos.
\newblock The dual of the notion of ``finitely generated''.
\newblock {\em J. London Math. Soc.}, 43:643--646, 1968.

\bibitem[XZ12]{XZ12}
Bao-Lin Xiong and Pu~Zhang.
\newblock {Gorenstein-projective modules over triangular matrix {A}rtin
  algebras}.
\newblock {\em Journal of Algebra and its Applications}, 11(4):1250066, 14,
  2012.

\bibitem[XZZ14]{XZZ14}
Bao-Lin Xiong, Pu~Zhang, and Yue-Hui Zhang.
\newblock {Auslander-{R}eiten translation in monomorphism categories}.
\newblock {\em Forum Mathematicum}, 26(3):863--912, 2014.

\bibitem[Zha11]{Zha11}
Pu~Zhang.
\newblock {Monomorphism categories, cotilting theory, and
  {G}orenstein-projective modules}.
\newblock {\em J. Algebra}, 339:181--202, 2011.

\bibitem[Zha13]{Zha13}
Pu~Zhang.
\newblock Gorenstein-projective modules and symmetric recollements.
\newblock {\em J. Algebra}, 388:65--80, 2013.

\end{thebibliography}

\end{document}